\documentclass[a4paper,12pt,twoside]{amsart}

\usepackage{fullpage}

\newcommand{\alitem}[1]{&\parbox{.85 \linewidth}{#1}}

\usepackage{mymacros,oip}

\title[Lagrangian sections]{Lagrangian sections on mirrors\\
of toric Calabi--Yau 3-folds}
\author[K.~Chan]{Kwokwai Chan}
\address{
Department of Mathematics,
The Chinese University of Hong Kong,
Shatin,
Hong Kong
}
\email{kwchan@math.cuhk.edu.hk}
\author[D.~Pomerleano]{Daniel Pomerleano}
\address{
Imperial College London,
180 Queen's Gate,
London SW7 2AZ,
UK
}
\email{daniel.pomerleano@gmail.com}
\author[K.~Ueda]{Kazushi Ueda}
\address{
Graduate School of Mathematical Sciences,
The University of Tokyo,
3-8-1 Komaba,
Meguro-ku,
Tokyo 153-8914,
Japan
}
\email{kazushi@ms.u-tokyo.ac.jp}
\date{}

\begin{document}

\begin{abstract}
We construct Lagrangian sections
of a Lagrangian torus fibration
on a 3-dimensional conic bundle,
which are SYZ dual to holomorphic line bundles over the mirror toric Calabi--Yau 3-fold.
We then demonstrate a ring isomorphism between the wrapped Floer cohomology of the zero-section
and the regular functions on the mirror toric Calabi--Yau 3-fold.
Furthermore, we show that in the case when the Calabi--Yau 3-fold is affine space,
the zero section generates the wrapped Fukaya category of the mirror conic bundle.
This allows us to complete the proof of one direction of homological mirror symmetry
for toric Calabi--Yau orbifold quotients of the form $\mathbb{C}^3/\Gv$.
We finish by describing some elementary applications of our computations to symplectic topology. 
\end{abstract}

\maketitle


\section{Introduction}

The goal of this paper is to provide evidence for and connect
homological mirror symmetry
\cite{MR1403918}
and the Strominger--Yau--Zaslow (SYZ) conjecture
\cite{MR1429831}
in the context of \emph{3-dimensional conic bundles} of the form
\begin{align}
 Y = \lc (w_1,w_2,u,v) \in (\bCx)^2 \times \bC^2 \relmid
  h(w_1,w_2) = uv \rc.
\end{align}
Mirror symmetry for such varieties goes back at least to
\cite{0005247}.
For simplicity, in this paper we will assume
that the Newton polytope $\Delta$ of $h(w_1,w_2)$ is full-dimensional and that the Laurent polynomial $h(w_1, w_2)$ is general among those whose Newton polytope is $\Delta$.
The mirror partner $\Yv$ is an open subvariety
of a 3-dimensional toric variety
whose fan is the cone over a triangulation
of the Newton polytope of the Laurent polynomial $h(w_1,w_2)$.
These examples have been a fertile testing ground
for mathematical thinking on mirror symmetry,
a pioneering example being \cite{MR1862802, MR1831820}.

These conic bundles have the distinguishing feature
that they are among the few examples of Calabi--Yau threefolds
which admit explicit Lagrangian torus fibrations.
An intriguing feature of these examples is that
the Lagrangian fibrations are only piecewise smooth and
have codimension one discriminant locus,
and thus exhibit important features of the general case.
There have therefore been a number of recent papers
focusing on these examples
through the lens of the SYZ conjecture.
The essential idea for constructing such fibrations appears in
\cite{MR1865243, MR1882328, MR1821145},
and the construction has been carried out in
\cite{MR2201567, MR2487600, MR3502098}.
The details of this construction with some modifications are presented
in \pref{sc:LTF}.

The manifolds $Y$ are affine varieties and
therefore have natural symplectic forms and Liouville structures.
On the other hand, in algebro-geometric terms,
$Y$ is the \emph{Kaliman modification} of $(\bCx)^2 \times \bC$
along the hypersurface $Z$ of $(\bCx)^2$
defined by $h(w_1,w_2)=0$
(see e.g. \cite{9801075}).
When viewed from this perspective,
there is an alternative symplectic form
used in \cite{MR3502098},
which arises by viewing $Y$ as an open submanifold
of the symplectic blow-up
$\Ybar$ of $(\bCx)^2 \times \bC$
along the codimension two subvariety $Z \times 0$.
This form is convenient
for the purposes of the SYZ conjecture,
since one can construct a Lagrangian torus fibration
whose behavior is easy to understand.
We modify their construction
by degenerating the hypersurface $Z$ to a certain tropical localization
introduced in \cite{MR2079993, MR2240909}. This modification plays an important role
throughout the rest of the paper as we explain below.

The starting point for our work is the observation
that the SYZ fibration comes equipped with
certain natural \emph{base-admissible Lagrangian sections}.
A related construction of Lagrangian sections has been carried out
independently by Gross and Matessi in \cite{1503.03816},
though we place an emphasis on the asymptotics of these sections
since these are important for Floer cohomology.
In \pref{sc:ALS}, we prove the following theorem:

\begin{theorem} \label{th:main1}
For every holomorphic line bundle $\cF$ on the SYZ mirror $\Yv$,
there is a base-admissible Lagrangian section $L_\cF$ in $Y$
whose SYZ transform is $\cF$.
Furthermore, such Lagrangian submanifolds are unique up to Hamiltonian isotopy.
\end{theorem}

For the purposes of mirror symmetry
for non-compact manifolds such as $Y$,
it is essential to consider the {\em wrapped Floer cohomology}
of non-compact Lagrangian submanifolds.
The foundational theory of wrapped Fukaya categories
on Liouville domains has been developed
in \cite{MR2602848}.
However, wrapped Floer cohomology is very difficult to calculate directly
and there are only a few computations in the literature.
One of the main objectives of this paper is to prove the following theorem:

\begin{theorem} \label{th:main0}
Let $L_0$ denote the section of the fibration
which is SYZ dual to the structure sheaf $\cO_\Yv$.
Then the wrapped Floer cohomology
$\HW^*(L_0)$ of $L_0$ is concentrated in degree 0, and
there exists an isomorphism
\begin{align} \label{eq:main0}
 \Mir^{L_0} \colon \HW^0(L_0) \simto H^0(\cO_\Yv)
\end{align}
of $\bC$-algebras.
\end{theorem}

The wrapped Floer cohomology ring on the left hand side of \eqref{eq:main0}
comes with a natural basis
given by Hamiltonian chords.
As suggested by 
Tyurin and emphasized by Gross, Hacking, Keel and Siebert
\cite{MR3415066,1204.1991},
the images of these Hamiltonian chords
on the right hand side of \eqref{eq:main0} give generalizations of
theta functions on Abelian varieties.
The proof of \pref{th:main0},
which is a direct consequence of \pref{th:main2} and \pref{th:main3} below,
is a higher dimensional analogue of \cite{MR3265556}.
However, the approach taken in this paper is slightly different.
While Pascaleff's argument exploits a certain TQFT structure
for Lefschetz fibrations due to Seidel,
our argument is based upon a study of how Floer theory behaves
under Kaliman modification.
Though this birational geometry viewpoint governs our approach,
we do not study this problem in maximal generality,
but limit ourselves to this series of examples.

The first step in our argument is to construct
a suitable family of Hamiltonians
which are well behaved with respect to the conic bundle structure.
The essential technical ingredient for wrapped Floer theory
is the existence of suitable $C^0$-estimates for solutions to Floer's equation.
However, base-admissible Lagrangian sections do not naturally fit
into the setup of \cite{MR2602848} and, for this reason,
we adapt their theory of wrapped Floer cohomology
to the above choice of symplectic forms
and base-admissible Lagrangian sections.

In our setup, there are now two directions in which Floer curves can escape to infinity,
namely,  in the ``base direction'' and in the ``fiber direction'' of the conic fibration.
We show that, outside of a compact set in the base,
the Hamiltonian flow on the total space
projects nicely to the Hamiltonian flow on the base.
It follows that there is a maximum principle
for solutions to Floer's equation
with boundary on admissible Lagrangians,
which prevents curves
from escaping to infinity in the ``base direction''.
On the other hand,
it is more difficult to prevent curves from escaping to infinity
in the ``fiber direction'' of the conic fibration.
In fact, curves can indeed escape to infinity.
However, we show in \pref{lm:Gcompactness}
that they must break along \emph{divisor chords},
which are Hamiltonian chords living completely inside the divisor at infinity.
There is a natural auxiliary grading \emph{relative}
to the exceptional divisor $E$.
The key observation is that, with respect to this grading,
the grading of the divisor Hamiltonian chords becomes arbitrarily large
as $m$ gets larger.
By restricting to generators for the Floer complex
whose relative grading lies in a certain range,
we are therefore able to exclude this breaking and
obtain the compactness needed to define wrapped Floer cohomology.

We now turn to the computational aspect of our paper.
The zero-section $L_0$ is a base-admissible Lagrangian
and so it makes sense to consider its adapted wrapped Floer cohomology ring
$\HWad^*(L_0)$.
In \pref{sc:HW}, we prove the following theorem:

\begin{theorem} \label{th:main2}
The adapted wrapped Floer cohomology $\HWad^*(L_0)$
is concentrated in degree 0, and
there exists an isomorphism
\begin{align} \label{eq:main2}
 \Mirad^{L_0} \colon \HWad^0(L_0) \simto H^0(\cO_\Yv)
\end{align}
of $\bC$-algebras.
\end{theorem}

One case where wrapped Floer cohomology can sometimes be directly computed
is in manifolds which are products of lower-dimensional manifolds
and where the Lagrangians and Hamiltonians split
according to the product structures.
An important advantage of base-admissible Lagrangians is that
they live away from the exceptional locus
and hence can be regarded as Lagrangian submanifolds
in $(\bCx)^2 \times \bCx$.
When viewed in this way,
these Lagrangians respect the product structure.
Moreover, the admissible Hamiltonians we use
interplay nicely with this product structure,
and hence the Floer theory of base-admissible Lagrangians
is amenable to direct calculation.

In the case of $L_0$,
we find that all Hamiltonian chords lie away from the exceptional locus
and moreover that as a vector space,
the Floer cohomologies agree
when regarded as living in $(\bCx)^2 \times \bCx$ or $Y$.
However, the product structure on Floer cohomology is deformed.
This deformation can be formalized
in terms of the \emph{relative Fukaya category}
of Seidel and Sheridan \cite{MR3364859,MR3294958}.
One slightly novel feature is that in typical situations,
one works relative to a compactifying divisor at infinity,
while here we work relative to the exceptional divisor $E$.
In this case, one computes the deformed ring structure directly
by exploiting the correspondence between holomorphic curves
in $(\bCx)^2 \times \bC$
with incidence conditions relative to the submanifold $Z \times 0$
and holomorphic curves in $Y$
(one direction of this correspondence is given by projection
and the other direction is given by proper transform).
The relevant enumerative calculation is then done in \pref{lm:enum}
based upon a simple degeneration-of-domain argument.

It is not difficult to see that the Lagrangian $L_0$ is also admissible
in the sense of \cite{MR2602848}
when $Y$ is equipped with its natural finite-type convex symplectic structure
(for a definition, see \cite[Definition 2.2]{MR2497314}).
In \pref{sc:Comp},
we prove the following comparison theorem
between the two types of Floer theories:

\begin{theorem} \label{th:main3}
There exists an isomorphism
\begin{align} \label{eq:main3}
 \phi_{\HW} \colon \HW^*(L_0) \simto \HWad^* (L_0)
\end{align}
of graded $\bC$-algebras.
\end{theorem}

The proof of this theorem is a modification of \cite[Theorem 5.5]{MR2497314},
which applies to Lefschetz fibrations.
While his arguments do not generalize
to symplectic fibrations over higher dimensional bases,
certain simplifying features of our situation allow us
to adapt his arguments in a straightforward way.

We use this result to calculate the zero-th symplectic cohomology $\SH^0(Y)$ of $Y$.
Abouzaid \cite{MR2737980} has introduced a map
\begin{align} \label{eq:CO_L0}
 \CO \colon \SH^0(Y) \to \HW^0(L_0).
\end{align}
In \pref{sc:hms}, we prove the following:

\begin{theorem}
The map \pref{eq:CO_L0}
is an isomorphism.
\end{theorem} 

The argument employed here is inspired by an argument of Pascaleff
\cite{1304.5298}.
Namely, it is not difficult to attain that the map is injective which again follows essentially from the fact that the symplectic fibration is trivial away from the discriminant locus.
This enables us to compute the ``local'' or ``low energy'' terms
in the map $\CO$ by reduction to the case of $(\bCx)^3$.
Proving surjectivity is more delicate
and rests upon making partial computations of the differential
in symplectic cohomology.
This relies on two recent ingredients:
a local computation of the BV-operator in \cite{1405.2084}
and a trick in \cite{1304.5298}
which enables us to rule out higher energy terms in the differential of a certain
``primitive'' cochain.

Finally, we apply all of these calculations to give a proof of homological mirror symmetry in the simplest case
when our polynomial $h(w_1,w_2)$ is $1+w_1+w_2$
(this corresponds to the case
when the mirror $\Yv$ is
$\Spec \lb \bC[x,y,z][(xyz-1)^{-1}] \rb$).
In a dual picture to the one above,
Abouzaid has introduced a map
\begin{align}
 \OC \colon \HH_3(\HW^*(L_0)) \to \SH^0(Y).
\end{align}
In our case,
the machinery behind this map is vastly simplified
because all of our chain complexes lie in degree zero.
Abouzaid showed that if the unit $[1]$ is hit by the map $\OC$,
then $L_0$ split-generates the wrapped Fukaya category $\cW(Y)$.
It follows from what we have proven so far
that the image of $\OC$ must be a principal ideal $(f) \in SH^0(Y)$.
Although it is difficult to compute the entire open string map directly,
the surjectivity of the map therefore follows formally
if we can show that
there are non-zero elements $f_1$, $f_2$, and $f_3$
in the image of $\OC$
which have no common divisors other than units.
This again uses the triviality of the symplectic fibration
away from the discriminant locus to make enough partial computations of the map
to produce these elements.
Combining all of these results, we obtain the following:

\begin{theorem} \label{th:hms}
$L_0$ generates the wrapped Fukaya category of $Y$.
In particular, there is an equivalence
\begin{align}
 \psi \colon D^b \cW(Y) \simto D^b  \coh \Yv
\end{align}
of enhanced triangulated categories
sending $L_0$ to $\cO_{\Yv}$.
\end{theorem}

To push this result somewhat further, we note that the behavior of (wrapped) Fukaya categories under finite covers is well-understood.
To be more precise,
let $N_0$ be a finite index subgroup of the fundamental group
$N$ of $Y$.
Note that $N$ is a free abelian group of rank 2.
Denote the quotient group by $G=N/N_0$ and
let $Y_{N_0}$ be the corresponding finite cover
(which is itself a conic bundle).
The set of connected components of the pull-back of $L_0$
by the covering map $\varpi \colon Y_{N_0} \to Y$
is a torsor over $G$,
which we identify with $G$
by choosing a base point;
$
 \varpi^{-1}(L_0) = \bigoplus_{g \in G} L_g.
$
The group $\Gv \coloneqq \Hom(G, \bCx)$
acts naturally on $\Yv$, and the quotient stack
is denoted by $\ld \Yv / \Gv \rd$.
We write the one-dimensional representation of $\Gv$
associated with an element $g \in G$
as $V_g$.
\pref{th:hms} admits the following generalization to finite covers:

\begin{corollary} \label{cr:hms_McKay}
There is an equivalence
\begin{align} \label{eq:hms_McKay}
 \psi_{N_{0}} \colon D^b \cW(Y_{N_0})
  \simto D^b \coh \ld \Yv/\Gv \rd
\end{align}
of enhanced triangulated categories
sending $L_g$ to $\cO_{\Yv} \otimes V_g$.
\end{corollary}

This computation represents to our knowledge
the first complete computation appearing in the literature
of the wrapped Fukaya category of a six-dimensional Liouville domain,
which is neither a product of lower dimensional manifolds nor a cotangent bundle.
This has applications to symplectic topology,
an elementary example being \pref{pr:embeddings}.
In a different direction,
the abelian McKay correspondence implies homological mirror symmetry
for divisor complements of toric Calabi--Yau varieties
which are related by a variation of GIT quotients
to orbifold quotients $\ld \bC^3/\Gv \rd$.

Now assume that $G$ is sufficiently `big'
in the sense that the crepant resolution of $\bC^3/\Gv$
has a compact divisor. Let
$
 \cF_0 \lb Y_{N_0} \rb
$
be the full subcategory of
$
 \cW \lb Y_{N_0} \rb
$
consisting of Lagrangian spheres, and
$
 \coh_0 \ld \Yv / \Gv \rd
$
be the full subcategory of
$
 \coh \ld \Yv / \Gv \rd
$
consisting of sheaves supported at the origin.

\begin{corollary} \label{cr:compact_embedding}
The equivalence \pref{eq:hms_McKay} induces
a fully faithful functor
\begin{align} \label{eq:compact_embedding}
 D^b \cF_0(Y_{N_0}) \hookrightarrow D^b \coh_0 \Yv / \Gv.
\end{align}
\end{corollary}

It seems likely that the above embedding \pref{eq:compact_embedding}
is inverse to the embedding 
\begin{align} \label{eq:Seidel_embedding}
 D^b \coh_0 \Yv / \Gv \hookrightarrow D^b \cF(Y_{N_0}),
\end{align}
given by combining \cite{MR2651908} and \cite{MR2415388}.
To prove this, one would have to compare symplectic forms used in this paper
with that in \cite{MR2651908};
we do not go into that problem here.
Another interesting problem would be to study the meaning of Bridgeland stability conditions
on the essential image of \pref{eq:compact_embedding}
in terms of symplectic or K\"ahler geometry of $Y_{N_0}$
(see \cite{MR2373143} for the foundation of Bridgeland stability conditions,
\cite{MR2238896, MR2852118} for the space of Bridgeland stability conditions
on some toric Calabi--Yau manifolds,
and \cite{MR1957663,MR3354954} for a conjectural relation
with the Fukaya category).
Nevertheless, the existence of this embedding already has a nice application:

\begin{proposition} \label{pr:disjoin1}
Let $S_1, \cdots, S_r$ be a collection of Lagrangian spheres in $Y_{N_0}$
which are pairwise disjoint.
Then $r \leq |\Gv|$.
\end{proposition}

In fact,
we give a slightly stronger result in \pref{pr:disjoin2}.
Finally,
it is worth mentioning that although we have focused
on the case of conic fibrations over $(\bCx)^2$,
our methods apply to the somewhat simpler case of
(fiber products of) conic bundles over $\bCx$ as well.
In particular,
they enable us to prove that the collections of Lagrangians
studied in \cite{MR3164868} and \cite{MR3439224}
generate the relevant wrapped Fukaya categories.
The latter case is especially interesting,
as this gives rise to examples,
different from those given in \cite{MR3349833, MR3416110},
of symplectic 6-manifolds
whose space of Bridgeland stability conditions
on Fukaya categories
is reasonably well understood
\cite{MR2425708,MR2448280,MR2629510}.

Returning to the case of a general toric Calabi--Yau manifold,
there are two natural ways in which one might try to extend our results
to a homological mirror symmetry statement. The first is to establish a version of homological mirror symmetry
between a suitable Fukaya categories
generated by base-admissible Lagrangians
and categories of coherent sheaves on the mirror $\Yv$.
Base admissible Lagrangians $L$ naturally fiber
over Lagrangians $\Lbar$ in $(\bCx)^2$.
The calculation of $A_\infty$-operations can likely be reduced
using the approach of \pref{sc:HW}
to the enumeration of Floer polygons with boundary on $\Lbar$.
It seems reasonable to hope that the latter polygons can be dealt with
using techniques similar to those in
\cite{MR2529936}.

Let $\bfD$ be the union of the toric divisors in $\Yv$.
A different direction of research begins with the observation that
we have a Bousfield localization sequence
\begin{align}
 0
  \to \QC_{\bfD}(\Yv)
  \to \QC(\Yv)
  \to \QC(\Yv \setminus \bfD)
  \to 0,
\end{align}
where $\QC(\Yv)$ is the unbounded derived category
of quasi-coherent sheaves on $\Yv$
and $\QC_\bfD(\Yv)$ is the full subcategory
consisting of objects
whose cohomologies are supported on $\bfD$.
Since $D^b \coh_{\bfD} \Yv$ generates $\QC_{\bfD}(\Yv)$ and
$\Yv \setminus \bfD$ is affine,
the bounded derived category $D^b \coh \Yv$
is split-generated by $\cO_\Yv$
and $D^b \coh_{\bfD} \Yv$.

This paper describes the subcategory of the wrapped Fukaya category
generated by $L_0$.
Furthermore,
SYZ mirror symmetry predicts that
there is a full-subcategory of this wrapped Fukaya category
consisting of Liouville-admissible Lagrangians
which is equivalent to $D^b \coh_\bfD(\Yv)$.
These Lagrangians are again constructed
using ideas from \cite{MR2529936},
and a similar construction has already appeared in \cite{1503.03816}.

For example,
when $\Yv$ is the total space of the canonical line bundle
over a toric Fano manifold,
we expect to have a collection of objects consisting of $L_0$
and compact Lagrangian spheres
which split-generate the wrapped Fukaya category.
We expect that the results in \cite{MR2529936}
together with an analysis similar to that in \cite{MR2651908}
would give a direct approach
to studying the Floer cohomology of these Lagrangian spheres.
Combining this with the results of this paper then gives
an approach to studying the wrapped Fukaya categories for general $h(w_1,w_2)$.

\emph{Acknowledgement}:
We thank Diego Matessi
for sending a draft of \cite{1503.03816} and useful discussions.
The work of K.~C.~described in this paper was substantially supported
by a grant from the Research Grants Council of the
Hong Kong Special Administrative Region, China
(Project No. CUHK400213).
D.~P.~was supported by Kavli IPMU, EPSRC, and Imperial College.
He would also like to thank Mohammed Abouzaid, Denis Auroux, and Mauricio Romo for helpful discussions concerning the construction of base-admissible Hamiltonians and Jonny Evans for a discussion about symplectic blowups.
K.~U.~was supported by JSPS KAKENHI Grant Numbers
24740043, 15KT0105, 16K13743 and 16H03930.

\section{Lagrangian torus fibrations}
 \label{sc:LTF}

\subsection{Tropical hypersurface}

Let $N = \bZ^2$ be a free abelian group of rank 2, and
$M = \Hom(N, \bZ)$ be the dual group.
A \emph{convex lattice polygon} in $M_\bR = M \otimes \bR$ is
the convex hull of a finite subset of $M$. We assume that all convex lattice polygons in this paper are full-dimensional.
Let $\Delta$ be a convex lattice polygon in $M_\bR$ and
$A = \Delta \cap M$ be the set of lattice points in $\Delta$.
A function $\nu \colon A \to \bR$ defines
a piecewise-linear function
$
 \nubar \colon \Delta \to \bR
$
by the condition that
\begin{align}
 \Conv \{ (\bsalpha, u) \in A \times \bR \mid u \ge \nu(\bsalpha) \}
  = \{ (\bm, u) \in \Delta \times \bR \mid u \ge \nubar(\bsalpha) \}.
\end{align}
The set $\cP_\nu$
consisting of maximal domains of linearity of $\nubar$
and their faces
is a polyhedral decomposition of $\Delta$.
A polyhedral decomposition $\cP$ of $\Delta$
is \emph{coherent} (or \emph{regular})
if there is a function $\nu \colon A \to \bZ$
such that $\cP = \cP_\nu$.
It is a \emph{triangulation}
if all the maximal-dimensional faces are triangles.
A coherent triangulation is \emph{unimodular}
if every triangle can be mapped to the standard simplex
(i.e., the convex hull of $\{ (0,0), (1,0), (0,1) \}$)
by the action of $\GL_2(\bZ) \ltimes \bZ^2$.

Let $\cP$ be a unimodular coherent triangulation of $\Delta$.
A function $\nu \colon A \to \bR$ is \emph{adapted} to $\cP$
if $\cP = \cP_\nu$.
Given a function $\nu \colon A \to \bR$
and an element $(c_\bsalpha)_{\bsalpha \in A} \in \bC^A$,
a \emph{patchworking polynomial} is defined by
\begin{align} \label{eq:ht}
 h_t(\bw)
  &= \sum_{\bsalpha \in A} c_\bsalpha t^{-\nu(\bsalpha)} \bw^\bsalpha
   \in \bC[t^{\pm 1}][M].
\end{align}
For a positive real number $t \in \bR^{> 0}$,
a hypersurface
$
 Z_t
$
of
$
 \NCx
  \coloneqq N_\bR / N
  \cong \Spec \bC[M]
$
is defined by
\begin{align}
 Z_t = \{ \bw \in \NCx \mid h_t(\bw) = 0 \}.
\end{align}
We assume that $Z_t$ is connected. The image
\begin{align} \label{eq:Pit}
 \Pi_t \coloneqq \Log(Z_t)
\end{align}
of $Z_t$
by the logarithmic map
\begin{align} \label{eq:Logt}
 \Log \colon \NCx \to N_\bR, \quad
 \bw = (w_1,w_2) \mapsto \frac{1}{\log |t|} \lb \log|w_1|, \log|w_2| \rb
\end{align}
is called the \emph{amoeba} of $Z_t$. The \emph{tropical polynomial}
associated with $h_t(\bw)$ is the piecewise-linear map
defined by
\begin{align} \label{eq:Lnu}
 L_\nu \colon N_\bR \to \bR, \quad
 \bn \mapsto \max \lc \bsalpha(\bn) - \nu(\bsalpha) \relmid \bsalpha \in A \rc.
\end{align}
The \emph{tropical hypersurface}
(or the \emph{tropical curve})
associated with $L_\nu$ is defined as the locus
$\Pi_\infty \subset N_\bR$
where $L_\nu$ is not differentiable. The polyhedral decomposition of $N_\bR$
defined by the tropical hypersurface $\Pi_\infty$
is dual to the triangulation $\cP$ of $\Delta$.
In particular,
the connected components of the complement of $\Pi_\infty$
can naturally be labeled by $\cP^{(0)} = A$;
\begin{align} \label{eq:C_alpha}
 N_\bR \setminus \Pi_\infty
  = \coprod_{\bsalpha \in A} C_{\bsalpha, \infty}.
\end{align}
The amoeba $\Pi_t$ converges to $\Pi_\infty$
in the Hausdorff topology
as $t$ goes to infinity
\cite{MR2079993, Rullgard_PAC}.

We next introduce some terminology and notation
which will be used throughout the rest of this paper. A \emph{leg} is a one-dimensional polyhedral subcomplex
of the tropical hypersurface $\Pi_\infty$
which is non-compact. In what follows, we will use the variables $r_i$ to denote $\log|w_i|$. Let $\{ \Pi_i \}_{i=1}^\ell$ be the set of legs
of the tropical curve $\Pi_\infty$.
For each leg $\Pi_i$,
we write the endpoints of the edge in $\cP$
dual to $\Pi_i$ as
$
 \bsalpha_i = (\bsalpha_{i,1}, \bsalpha_{i,2}),
 \bsbeta_i = (\bsbeta_{i,1}, \bsbeta_{i,2})
 \in A,
$
so that $\Pi_i$ is defined by
\begin{align} \label{eq:Pii}
 \lc
\begin{aligned}
 \br_{\bsalpha_i-\bsbeta_i}
  &\coloneqq (\alpha_{i,1}-\beta_{i,1}) r_1 + (\alpha_{i,2}-\beta_{i,2}) r_2
  = \nu(\bsalpha_i)-\nu(\bsbeta_i), \\
 \br_{(\bsalpha_i-\bsbeta_i)^{\bot}}
  &\coloneqq (\beta_{i,2}-\alpha_{i,2}) r_1 + (\alpha_{i,1}-\beta_{i,1}) r_2
  \ge a_i,
\end{aligned}
 \right.
\end{align}
for some $a_i \in \bR$. The maximal polyhedral subcomplex of $\Pi_\infty$
which is compact
will be denoted by $\Pi_{\infty,c}$. Define $c_i, c_i', c_i'' \in \bQ$ by
\begin{align} \label{eq:ci}
 c_i = \frac{1}{| \bsalpha_i - \bsbeta_i |},
  \quad
 c_i' = \frac{{\bsalpha_i} \cdot ({\bsalpha_i}-{\bsbeta_i})}{|{\bsalpha_i}-{\bsbeta_i}|^2},
  \quad
 c_i'' = \frac{{\bsalpha_i} \cdot ({\bsalpha_i}-{\bsbeta_i})^\bot}{|{\bsalpha_i}-{\bsbeta_i}|^2},
\end{align}
so that
\begin{align}
 \abs{\br}^2 &\coloneqq r_1^2 + r_2^2 =
  c_i^2 \lb \lb \br_{{\bsalpha_i}-{\bsbeta_i}} \rb^2
   + \lb \br_{({\bsalpha_i}-{\bsbeta_i})^\bot} \rb^2 \rb, \\
 \bsalpha_i &= c_i'({\bsalpha_i}-{\bsbeta_i}) + c_i''({\bsalpha_i}-{\bsbeta_i})^\bot.
\end{align}
Then we have
\begin{align}
 |\bw^{\bsalpha_i}| = |\bw^{{\bsalpha_i}-{\bsbeta_i}}|^{c_i'} |\bw^{({\bsalpha_i}-{\bsbeta_i})^\bot}|^{c_i''} \\
 d \theta_{\bsalpha_i} = c_i' d \theta_{{\bsalpha_i}-{\bsbeta_i}} + c_i'' d \theta_{({\bsalpha_i}-{\bsbeta_i})^\bot}
\end{align}

\begin{example}
A prototypical example is
\begin{align}
 h_t(\bw) = w_1 + w_2 + \frac{1}{w_1 w_2} + t^\varepsilon
\end{align}
for a small positive real number $\varepsilon$.
The corresponding tropical polynomial is given by
\begin{align}
 L(\bn) = \max \{ n_1, n_2, -n_1-n_2, \varepsilon \},
\end{align}
and the tropical curve $\Pi_\infty$ is shown in \pref{fg:trop_curve1}.
The set
\begin{align}
 A = \{ (0,0), (1,0), (0,1), (-1,-1) \}
\end{align}
consists of four elements,
corresponding to four connected components of $N_\bR \setminus \Pi_\infty$.
The tropical curve $\Pi_\infty$ has three legs.

\begin{figure}
\begin{minipage}{.49 \linewidth}
\centering
\input{triangulation1.pst}
\caption{A triangulation}
\label{fg:triangulation1}
\end{minipage}
\begin{minipage}{.49 \linewidth}
\centering
\input{trop_curve1.pst}
\caption{A tropical curve}
\label{fg:trop_curve1}
\end{minipage}
\end{figure}

\end{example}

\subsection{Tropical localization}

Following \cite[Section 4]{MR2240909},
we set $C_{\bsalpha, t} \coloneqq (\log t) \cdot C_\bsalpha$ for each $\bsalpha \in A$,
and choose $\phi_\bsalpha \colon N_\bR \to \bR$ such that
\begin{itemize}
 \item
$
 d(\bn, C_{\bsalpha, t})
  \coloneqq \displaystyle{\min_{\bn' \in C_{\bsalpha,t}}} \| \bn - \bn' \|
  \le (\varepsilon \log t) / 2
$
if and only if $\phi_\bsalpha(\bn)=0$,
 \item
$
 d(\bn, C_{\bsalpha, t}) \ge \varepsilon \log t
$
if and only if $\phi_\bsalpha(\bn)=1$, and
 \item
$
 \left| \dfrac{\partial \phi_\bsalpha}{\partial n_1} \right|
  + \left| \dfrac{\partial \phi_\bsalpha}{\partial n_2} \right|
 < \dfrac{4}{\varepsilon \log t}
$
\end{itemize}
for a small positive real number $\varepsilon$.
For an element $(c_\bsalpha)_{\bsalpha \in A} \in \bC^A$,
define a family $\{ h_{t,s} \}$ of maps
$\NCx \to \bC$ by
\begin{align} \label{eq:hts}
 h_{t,s}(\bw)
  &= \sum_{\bsalpha \in A} c_\bsalpha t^{-\nu(\bsalpha)}
   (1 - s \phi_\bsalpha(\bw)) \bw^\bsalpha,
\end{align}
where we write $\phi_\bsalpha(\bw) \coloneqq \phi_\bsalpha(\text{Log}_t(\bw))$
by abuse of notation.
For a face $\tau \in \cP$, define
\begin{align} \label{eq:Otau}
 O_\tau = \{ \bn \in N_\bR \mid
  (\phi_\bsalpha(\bn) \ne 1 \text{ for } \forall \bsalpha \in \tau)
   \text{ and }
  (\phi_\bsalpha(\bn) = 1 \text{ for } \forall \bsalpha \not \in \tau)
  \}.
\end{align}
Then $O_\tau$ is contained in an $\varepsilon$-neighborhood
of the face of $\Pi_\infty$ dual to $\tau$,
and one has
\begin{align}
 N_\bR = \coprod_{\tau \in \cP} O_\tau.
\end{align}
The set
\begin{align} \label{eq:Zts}
 Z_{t,s} = \{ \bw \in \NCx \mid h_{t,s}(\bw)=0 \}
\end{align}
is a symplectic hypersurface in $\NCx$
for a sufficiently large $t$
\cite[Proposition 4.2]{MR2240909},
and
the pairs $\lb \NCx, Z_{t,s} \rb$ for all $s \in [0,1]$
are symplectomorphic to each other
\cite[Proposition 4.9]{MR2240909} which for $t$ sufficiently large can be taken $C^k$-small and to be supported inside of a tubular neighborhood of $Z_{t,0}$ of small norm in the Kahler metric \cite[Appendix A]{MR2529936}.

The \emph{tropical localization} of $Z_t$ is defined by
\begin{align}
 Z &\coloneqq Z_{t,1}.
\end{align}
We set
$(c_\bsalpha)_{\bsalpha \in A} = \bsone \coloneqq (1, \ldots, 1) \in \bC^A$
and
\begin{align} \label{eq:ht1}
 h(\bw)
  \coloneqq h_{t,s}(\bw) \big|_{s=1, (c_{\bsalpha_i})_{{\bsalpha_i} \in A} = \bsone}
  = \sum_{{\bsalpha_i} \in A} t^{-\nu({\bsalpha_i})}(1-\phi_{\bsalpha_i}(\bw))\bw^{\bsalpha_i}.
\end{align}
The hypersurface $Z$ is localized in the following sense:
\begin{itemize}
 \item
Over $O_\tau$ where $\tau$ is dual to an edge of $\Pi_\infty$
(i.e., when $\tau \in \mathcal{P}^{(1)}$),
all but two terms of $h$ vanish,
and hence $Z$ is defined by
\begin{align}
 h(\bw)
 = t^{-\nu(\bsalpha)} \left( 1 - \phi_{\bsalpha}(\bw) \right) \bw^{\bsalpha}
     + t^{-\nu(\bsbeta)} \left( 1 - \phi_{\bsbeta}(\bw) \right) \bw^{\bsbeta}
 = 0,
\end{align}
where $\bsalpha, \bsbeta \in A$ are the endpoints of $\tau$.
 \item
Over $O_\sigma$ where $\sigma$ is dual to a vertex of $\Pi_\infty$
(i.e., when $\sigma \in \mathcal{P}^{(2)}$),
all but three terms of $h$ vanish,
whence $Z$ is defined by
\begin{align}
 h(\bw) &=
  t^{-\nu(\bsalpha_0)} \left( 1 - \phi_{\bsalpha_0}(\bw) \right) \bw^{\bsalpha_0}
 + t^{-\nu(\bsalpha_1)} \left( 1 - \phi_{\bsalpha_1}(\bw) \right) \bw^{\bsalpha_1} \nonumber \\
 &\hspace{60mm} + t^{-\nu(\bsalpha_1)} \left( 1 - \phi_{\bsalpha_1}(\bw) \right) \bw^{\bsalpha_2}
 \\
 &= 0,
\end{align}
where $\bsalpha_0, \bsalpha_1, \bsalpha_2 \in A$ are the vertices of $\sigma$.
\end{itemize}

It follows that the amoeba
$\Piloc = \Log(Z)$
of the tropical localization $Z$
agrees with the tropical hypersurface $\Pi_\infty$
away from the $\varepsilon$-neighborhood
of the 0-skeleton $\Pi_\infty^{(0)}$. An important property for us is that
each connected component of the complement of a compact set
in $Z$ is defined by a single \emph{algebraic} equation
\begin{align}\label{eq:eqn_Zloc_leg}
 t^{-\nu(\bsalpha)} \bw^\bsalpha+t^{-\nu(\bsbeta)} \bw^\bsbeta = 0,
\end{align}
and fibers over a subset of a leg of the tropical hypersurface $\Pi_\infty$.
In a slight abuse of terminology,
we will refer to this portion of the curve $Z$
as a \emph{leg} of $Z$.

For a sufficiently large number $R_0$ and a relatively small number
$\varepsilon_{\mathrm{n}} \ll R_0$,
we may assume that the union
\begin{align}
 U_{\Pi} \coloneqq U_{\Pi_c} \cup \bigcup_{i=1}^\ell U_{\Pi_i}
\end{align}
of
\begin{align}
 U_{\Pi_c} \coloneqq \lc \br \in N_\bR \relmid
  \abs{\br} < R_0 \rc
\end{align}
and
\begin{align}
 U_{\Pi_i}
  &= \lc \br \in N_\bR \relmid
   \abs{\br_{\bsalpha_i-\bsbeta_i}}<\varepsilon_{\mathrm{n}}
    \text{ and }
   \br_{(\bsalpha_i-\bsbeta_i)^\bot} \ge a_i + \varepsilon_{\mathrm{n}}
   \rc
\end{align}
is a neighborhood of $\Piloc$. We assume that $\varepsilon_{\mathrm{n}}$ is large enough so that it contains the line $\br_{\bsalpha_i-\bsbeta_i}-(\nu(\bsalpha_i)-\nu(\bsbeta_i))=0$, and that if we have non-compact parallel legs $\Pi_i$, we will assume that all of the corresponding neighborhoods $U_{\Pi_{i}}$ are the same.  
We next choose a neighborhood $U_{Z}$ of $Z$ in $\NCx$ and set
\begin{align}
 U_{Z_i} \coloneqq U_Z \cap \Log^{-1}(U_{\Pi_i}), \quad
 U_{Z_c} \coloneqq U_Z \cap \Log^{-1}(U_{\Pi_c}).
\end{align}

We require:
\begin{itemize}
 \item
$\Log(U_{Z}) \subset U_{\Pi}$,
 \item
$\abs{\dfrac{h(\bw)}{w^{\bsalpha}}} < \varepsilon_{\mathrm{h}}$ for some small constant $\varepsilon_{\mathrm{h}}>0$ and any $\bw \in U_{Z_i}$,
\item
$
 h= t^{-\nu(\bsalpha)} \bw^\bsalpha+t^{-\nu(\bsbeta)} \bw^\bsbeta
$
for any $\bw \in U_{Z_i}$,
\item
the set
$\{ \bsalpha \in A \mid \phi_\bsalpha(\bw) \ne 1 \}$
consists of either two or three elements
for any $\bw \in U_{Z}$,
and
\item
$U_{Z}$ does not intersect
$
 N_{\bR^{>0}} \coloneqq
  \lc (w_1, w_2) \in (\bR^{>0})^2 \rc \subset \NCx$.
\end{itemize}

We write the natural projections as
\begin{align} \label{eq:pi12}
 \pr_1 \colon \NCx \times \bC \to \NCx, \quad
 \pr_2 \colon \NCx \times \bC \to \bC.
\end{align}
We choose a tubular neighborhood
$U_{Z \times 0}$ of $Z \times 0$ in $\NCx \times \bC$
such that
\begin{itemize}
 \item
$\pr_1(U_{Z \times 0}) \subset U_Z$ and
 \item
$
 \frac{\abs{u}^2}{|\bw^{({\bsalpha_i}-{\bsbeta_i})^\bot}|^{2c_i''}} + \abs{t^{-\nu(\bsalpha)} + t^{-\nu(\bsbeta)} \bw^{\bsbeta-\bsalpha}}^2
  < \varepsilon_{\mathrm{h}}^2
$
if
$
 (\bw, u) \in U_{Z \times 0}
$
and
$
 \Log(\bw) \in U_{\Pi_i}.
$
\end{itemize}
We set
\begin{align} \label{eq:Ui}
 U_i \coloneqq \lc (\bw, u) \in U_{Z \times 0} \relmid
  \Log(\bw) \in U_{\Pi_i} \rc.
\end{align}
We fix an almost complex structure $J_{\NCx}$ on $\NCx$,
which is adapted to $Z$ in the following sense:
\begin{definition} \label{df:ad_Z}
An $\omega_{\NCx}$ compatible almost complex structure
$J_{\NCx}$ on $\NCx$ is said to be
\emph{adapted to $Z$}
if
\begin{itemize}
 \item
$J_{\NCx}$ agrees with the standard complex structure
$J_{\NCx,\std}$ of $\NCx$
outside the inverse image by $\Log$
of a small neighborhood of the origin
in $N_\bR$, and
 \item
the function $h(\bw)$ is $J_{\NCx}$-holomorphic in $U_{Z}$.
\end{itemize}
\end{definition}

\subsection{Symplectic blow-up}
 \label{sc:bu}

We follow \cite{MR3502098} in this subsection;
see also \cite{MR2487600} for a closely related construction.
As a smooth manifold, the blow-up
\begin{align} \label{eq:blow-up}
 p \colon \Ybarloc \coloneqq \Bl_{Z \times 0}(\NCx \times \bC)
  \to \NCx \times \bC
\end{align}
of $\NCx \times \bC$
along the symplectic submanifold $Z \times 0$
is given by
\begin{align}
 \Ybarloc
  &\coloneqq \lc \lb \bw, u, [v_0:v_1] \rb \in \NCx \times \bC \times \bP^1
   \relmid u v_1 = h(\bw) v_0 \rc.
\end{align}
The compositions of the structure morphism \eqref{eq:blow-up}
and the projections \eqref{eq:pi12}
will be denoted by
\begin{align} \label{eq:pibar}
 \pibar_{\NCx} \coloneqq \pr_1 \circ p \colon \Ybarloc \to \NCx, \quad
 \pibar_{\bC} \coloneqq \pr_2 \circ p \colon \Ybarloc \to \bC.
\end{align}
The exceptional set is given by
\begin{align}
 E \coloneqq p^{-1}(Z \times 0)
  = \lc \lb \bw, u, [v_0:v_1] \rb \in \NCx \times \bC \times \bP^1
   \relmid h(\bw) = u = 0 \rc,
\end{align}
which forms a $\bP^1$-bundle
$
 p|_E \colon E \to Z \times 0
$
over $Z \times 0$.
The total transform of the divisor
$Z \times \bC \subset \NCx \times \bC$
is given by
\begin{align} \label{eq:Ebar}
 \Ebar
  \coloneqq \pibar_{\NCx}^{-1}(Z)
  = E \cup F
  \subset \Ybarloc,
\end{align}
where $F$ is the strict transform of $Z \times \bC$.
There is an $\bS^1$-action on $\Ybarloc$ defined by
\begin{align} \label{eq:S^1-action}
 (\bw, u, [v_0:v_1]) \mapsto
  \lb \bw, e^{\sqrt{-1} \theta} u, \ld v_0:e^{-\sqrt{-1} \theta} v_1 \rd \rb.
\end{align}
Let
\begin{align} \label{eq:D}
 D
  \coloneqq \lc \lb \bw, u, [v_0:v_1] \rb \in \Ybarloc \relmid v_0 = 0 \rc
  \cong \NCx
\end{align}
be the strict transform of the divisor
\begin{align}
 \Dbar \coloneqq \{ (\bw, u ) \in \NCx \times \bC \mid u = 0 \},
\end{align}
and write its complement as
\begin{align}
 Y \coloneqq \Ybarloc \setminus D
  \cong \lc (\bw, u, v) \in \NCx \times \bC^2 \relmid h(\bw) = u v \rc,
\end{align}
where $v = v_1/v_0$.
The restrictions of \eqref{eq:pibar} to $Y$ will be denoted by
\begin{align} \label{eq:pi}
 \pi_{\NCx} \coloneqq \pibar_{\NCx}|_{Y} \colon Y \to \NCx, \quad
 \pi_{\bC} \coloneqq \pibar_{\bC}|_{Y} \colon Y \to \bC.
\end{align}
We also write
\begin{align}
 \pibar_{N_\bR} \coloneqq \Log \circ \pibar_{\NCx} \colon \Ybarloc \to N_\bR, \quad
 \pi_{N_\bR} \coloneqq \Log \circ \pi_{\NCx} \colon Y \to N_\bR.
\end{align}
A tubular neighborhood of $D$ in $\Ybarloc$
is given by
\begin{align} \label{eq:UD}
 U_D = \lc \lb \bw, u, [v_0:1] \rb \in \NCx \times \bC \times \bP^1
    \relmid
  u = h(\bw) v_0, \ |v_0| < \delta \rc
\end{align}
for a small positive number $\delta$.
We identify $U_D$ with $\NCx \times \bD_\delta$ by the map
\begin{align} \label{eq:UD2}
\begin{array}{ccc}
 U_D & \to & \NCx \times \bD_\delta \\
 \vin & & \vin \\
 (\bw, u, [v_0:1]) & \mapsto & (\bw, v_0)
\end{array}
\end{align}
where
$
 \bD_\delta \coloneqq \{ v_0 \in \bC \mid |v_0| < \delta \}
$
is an open disk of radius $\delta$.
The projection will be denoted by
\begin{align}
 \pi_{\bD_\delta} \colon
  U_D \cong \NCx \times \bD_\delta
   \to \bD_\delta.
\end{align}
For a function $f$ on an almost complex manifold,
we set
\begin{align}
 d^c f \coloneqq d f \circ J,
\end{align}
so that
$
 - d d^c f = 2 \sqrt{-1} \partial \partialbar f
$
whenever $J$ is integrable
(or more generally if $f$ is pulled back from an integrable complex manifold
along a $J$-holomorphic map).

We consider the two-form
\begin{align} \label{eq:symp_form}
 \omega_\epsilon
  &\coloneqq p^*\lb \omega_{\NCx \times \bC}
   - \frac{\epsilon}{4\pi}
    d d^c\lb
     \chi(\bw,u) \log \lb |u|^2 +|h(\bw)|^2 \rb \rb \rb
\end{align}
on $\Ybarloc \setminus p^*(Z \times 0)$
for a sufficiently small $\epsilon$,
where
\begin{align}
 \omega_{\NCx \times \bC}
  &= \frac{\sqrt{-1}}{2}
   \lb
    du \wedge d \ubar
    + \frac{dw_1}{w_1} \wedge \frac{d\wbar_1}{\wbar_1}
    + \frac{dw_2}{w_2} \wedge \frac{d\wbar_2}{\wbar_2}
   \rb
\end{align}
is the standard symplectic form on $\NCx \times \bC$, and
the function
$
 \chi \colon \NCx \times \bC \to [0,1]
$
is a smooth $\bS^1$-invariant cut-off function
supported on the tubular neighborhood
$U_{Z \times 0}$ of $Z \times 0$ and
satisfying $\chi \equiv 1$
in a smaller tubular neighborhood $U'_{Z \times 0}$
of $Z \times 0$.
We require that
\begin{align} \label{eq:chiG1}
 \chi|_{U_i} = \chi_i \circ G_i
\end{align}
for a function $\chi_i \colon \bR^{\ge 0} \to [0,1]$,
where
\begin{align} \label{eq:chiG2}
 G_i \colon U_i \to \bR^{\ge 0}, \quad
 (\bw, u) \mapsto
\frac{\abs{u}^2}{|\bw^{({\bsalpha_i}-{\bsbeta_i})^\bot}|^{2c_i''}} + \abs{t^{-\nu(\bsalpha_i)}+t^{-\nu(\bsbeta_i)} \bw^{\bsbeta_i-\bsalpha_i}}^2.
\end{align}
For clarity, we emphasize that
in \pref{eq:symp_form} the operator $d^c$ is defined
with respect to the almost complex structure $J_{\NCx}$
adapted to $Z$ in the sense of \pref{df:ad_Z}.
A crucial feature for us is that in the neighborhoods $U_i$,
this form is actually invariant under the $\bT^2$-action
which preserves the monomial $\bw^{\bsbeta_i-\bsalpha_i}$.
%
This 2-form extends to a 2-form on $\Ybarloc$,
which we write as $\omega_\epsilon$ again by abuse of notation,
since
we may rewrite \eqref{eq:symp_form} as
\begin{align} \label{eq:symp_form2}
 \omega_\epsilon
  &= p^*\omega_{\NCx \times \bC}
   + \frac{\sqrt{-1} \epsilon}{2\pi}
    \partial \partialbar \lb \log \lb |v_0|^2+|v_1|^2 \rb \rb
\end{align}
when $\chi(\bw, u) = 1$.


\begin{proposition} \label{pr:symp_form}
The two-form $\omega_\epsilon$ is a symplectic form
for sufficiently small $\epsilon$.
\end{proposition}

\begin{proof}
When $\chi=1$, the symplectic form is the restriction of the form
\begin{align} \label{eq:symp_form3}
  \omega_\epsilon &=
  p^*\omega_{\NCx \times \bC}
   + \frac{\sqrt{-1} \epsilon}{2\pi}
    \partial \partialbar \lb \log \lb |v_0|^2+|v_1|^2 \rb \rb
\end{align}
on $\NCx \times \bC \times \bP^1$
to $\Ybarloc$.
It follows that,
whenever $\chi=1$,
$\omega_\epsilon$ is the restriction of a compatible symplectic form \pref{eq:symp_form3}
to an almost complex submanifold,
and hence symplectic in this region.
When $\chi(\bw,u) \neq 1$,
the first term in \pref{eq:symp_form} is a symplectic form.
Along the legs,
one explicitly checks that the forms
$d \log \lb |u|^2 +|h(\bw)|^2 \rb$,
$d^c \log \lb |u|^2 +|h(\bw)|^2 \rb$ and
$dd^c \log \lb |u|^2 +|h(\bw)|^2 \rb$
have bounded coefficients in the forms generated by
$du$, $d\ubar$,
$d \log w_1$,
$d \log \wbar_1$,
$d \log w_2$, and
$d \log \wbar_2$
when $\chi(\bw,u) \neq 1$.
This implies that
$
 dd^c \lb \chi(\bw,u) \log \lb |u|^2 +|h(\bw)|^2 \rb \rb
$
also has bounded coefficients
and hence $\omega_\epsilon$ is a symplectic form
for sufficiently small $\epsilon$.
\end{proof}

\begin{remark}
We make one comment
concerning the choice of symplectic form here and that in \cite{MR3502098}.
Observe that the above construction could be repeated
by choosing $U_{Z_{t,0}}$ a suitable tubular neighborhood of $Z_{t,0} \subset \NCx$,
$U_{Z_{t,0} \times 0}$ a tubular neighborhood of $Z_{t,0} \times 0$
such that
$
 \pr_1(U_{Z_{t,0} \times 0}) \subset U_{Z_{t,0}}
$
and $\chi(\bw,u)$ be a smooth function
whose support is in $U_{Z_{t,0} \times 0}$ and such that
$
 dd^c \chi(\bw,u) \log \lb \abs{u}^2+\abs{h_{t,0}}^2 \rb
$
has bounded coefficients
whenever $\chi(\bw,u) \ne 1$.
Let $Y_{t,1}$ be the variety defined by $uv=h_{t,0}$ and
equip it with a symplectic form as in \pref{eq:symp_form}
with $h(\bw)$ replaced by $h_{t,1}(\bw)$.
It follows from \cite[Proposition 4.9]{MR2240909}
that for $\epsilon$ sufficiently small and $t$ sufficiently large
and suitable $\chi(\bw,u)$, $U_{Z_{t,i}}$, $i \in \lc 0,1 \rc$,
there is a symplectomorphism $\phi_{0,1} \colon Y_{t,0} \cong Y_{t,1}$
which is the identity
away from the preimage of a tubular neighborhood of $Z_{t,0}$ of small size.
\end{remark}

We fix a convenient choice of a primitive $\theta_\epsilon$
for the restriction of $\omega_\epsilon$ to $Y$,
which we write $\omega_\epsilon$ by abuse of notation.
The form
\begin{align} \label{eq:prim}
 \thetavc \coloneqq
  -\frac{\epsilon}{4 \pi} d^c
   \lb \chi(\bw,u) \log \lb |u|^2 +|h(\bw)|^2 \rb -\log(|u|^2) \rb
\end{align}
is well-defined
on the subset $Y \subset \Ybarloc$
and gives a primitive for the form
\begin{align}
-\frac{\epsilon}{4\pi}
  d d^c \lb
     \chi(\bw,u) \log \lb |u|^2 +|h(\bw)|^2 \rb \rb.
\end{align}
Now we define
\begin{align} \label{eq:theta-epsilon}
 \theta_\epsilon \coloneqq p^*\theta_{\NCx \times \bC} + \thetavc,
\end{align}
where $\theta_{\NCx \times \bC}$ is the standard primitive of
$\omega_{\NCx \times \bC}$, so that
\begin{align}
 d \theta_\epsilon = \omega_\epsilon.
\end{align}
The $\bS^1$-action \eqref{eq:S^1-action} is Hamiltonian
with respect to the symplectic form \eqref{eq:symp_form}
with the moment map
\begin{align} \label{eq:moment_map}
 \mu &= \pi \abs{u}^2 + \frac{\epsilon}{2} \abs{u} \frac{\partial}{\partial \abs{u}}
  \lb \chi(\bw, u) \log \lb \abs{h(\bw)}^2+\abs{u}^2 \rb \rb.
\end{align}
Our conventions for the moment map follow those of \cite{MR3502098}
(in particular, it differs from the more standard convention by a factor of $2\pi$).
This formula specializes to
\begin{align} \label{eq:moment_map2}
 \mu &=
\begin{cases}
 \displaystyle
  \pi |u|^2 + \epsilon \frac{|u|^2}{|h(\bw)|^2 + |u|^2}
   & \text{where $\chi \equiv 1$ (near $E$)},\\
 \displaystyle
  \pi |u|^2
   & \text{where $\chi \equiv 0$ (away from $E$)}.
\end{cases}
\end{align}
The level set $\mu^{-1}(\lambda)$ is smooth
unless $\lambda = \epsilon$,
where it is singular along the fixed locus
\begin{align}
 \Ztilde = \{ (\bw, u, [v_0:v_1]) \in \Ybarloc \mid h(\bw) = u = v_1 = 0 \}.
\end{align}
The level set $\mu^{-1}(0)$ is the divisor
$D$ defined in \eqref{eq:D}.
The fiber $\pi_{\NCx}^{-1}(\bw)$
over any point $\bw \notin Z$
is a smooth conic.
Let $d\theta$ denote the natural $\bS^1$-invariant angular one-form
on the smooth fiber $\pi_{\NCx}^{-1}(\bw)$.
Then the primitive $\theta_\epsilon$ restricted to the smooth fiber is given by
\begin{align}
 \theta_\epsilon|_{\pi_{\NCx}^{-1}(\bw)}
  = \abs{u} \frac{\partial}{\partial \abs{u}} \lb
   \frac{1}{4}
    \abs{u}^2 + \frac{\epsilon}{4\pi} \chi(\bw, u) \log \lb \abs{h(\bw)}^2+\abs{u}^2\rb
     -\frac{\epsilon}{4 \pi} \log \lb \abs{u}^2 \rb
   \rb d\theta.
\end{align}
In view of \pref{eq:moment_map} and \pref{eq:prim},
we may rewrite this in the much simpler form:
\begin{align} \label{eq:primfiber}
\theta_\epsilon|_{\pi_{\NCx}^{-1}(\bw)}= \frac{1}{2\pi}(\mu-\epsilon) d\theta.
\end{align}
The same formula holds when $\bw \in Z$,
away from the singular points of $\pi_{\NCx}^{-1}(\bw)$.
\begin{comment}
Let $\theta_u$ be the argument of $u$ and
$X = 2 \pi \partial/\partial \theta_u$ be the fundamental vector field
for the $\bS^1$-action \pref{eq:S^1-action}.
Since the primitive $\theta_\epsilon$ is $\bS^1$-invariant,
the moment map is given by
$
 \iota_X \theta_\epsilon
$
up to an additive constant,
which fix by the condition that
the level set $\mu^{-1}(0)$ is the divisor $D$
defined in \eqref{eq:D}.
This gives
\begin{align} 
 \mu &= \pi \abs{u}^2 + \frac{\epsilon}{2} \abs{u} \frac{\partial}{\partial \abs{u}}
  \lb \chi(\bw, u) \log \lb \abs{h(\bw)}^2+\abs{u}^2 \rb \rb,
\end{align}
which specializes to
\begin{align} 
 \mu &=
\begin{cases}
 \displaystyle
  \pi |u|^2 + \epsilon \frac{|u|^2}{|h(\bw)|^2 + |u|^2}
   & \text{where $\chi \equiv 1$ (near $E$)},\\
 \displaystyle
  \pi |u|^2
   & \text{where $\chi \equiv 0$ (away from $E$)}.
\end{cases}
\end{align}
We define an $\bS^1$-invariant angular one-form $d \theta$
on the smooth fiber $\pi_{\NCx}^{-1}(\bw)$
by 
\begin{align}
 \left. \theta_\epsilon \right|_{\pi_{\NCx}^{-1}(\bw)}
  = \frac{1}{2\pi}(\mu-\epsilon) d\theta.
\end{align}
\end{comment}
For $\lambda \in \bR^{>0} \setminus \{ \epsilon \}$,
the map
$
 \pi_{\NCx} \colon Y \to \NCx
$
induces a natural identification
\begin{align}
 \Yredl \coloneqq \mu^{-1}(\lambda)/\bS^1
  \cong \NCx
\end{align}
of the reduced space
and $\NCx$.
The resulting reduced symplectic form
$\omegaredl$ on $\NCx$
can be averaged by the action of the torus
$N_{\bS^1} \coloneqq N_\bR / N$
to obtain a torus-invariant symplectic form
$\omega_{\NCx, \lambda}'$.
\cite[Lemma 4.1]{MR3502098} states that
there exists a family $(\phi_\lambda)_{\lambda \in \bR^{>0}}$
of diffeomorphisms of $\NCx$
such that
\begin{itemize}
 \item
$\phi_\lambda^* \omega_{\NCx, \lambda}' = \omegaredl$,
 \item
$\phi_\lambda = \text{id}$ at every point
whose $N_{\bS^1}$-orbit is disjoint from the support of $\chi$.
 \item
$\phi_\lambda$ depends on $\lambda$ in a piecewise smooth manner.
\end{itemize}
We set
$
 \pi_\lambda \coloneqq \Log \circ \phi_\lambda \colon
  \NCx \to N_\bR
$
and define a continuous, piecewise smooth map by
\begin{align} \label{eq:piB}
 \piB \colon Y \to B \coloneqq N_\bR \times \bR^{>0}, \quad
  x \in \mu^{-1}(\lambda) \mapsto \lb \pi_\lambda([x]), \lambda \rb.
\end{align}
One can easily see as in
\cite[Section 4.2]{MR3502098}
that fibers of $\piB^{-1}$ are smooth Lagrangian tori
outside of the \emph{discriminant locus}
$
 \Log \circ \phi_{\epsilon} (Z) \times \{ \epsilon \}.
$

\subsection{SYZ mirror construction} \label{subsec:SYZ_mirror_constr}

We continue to follow \cite{MR3502098}
in this subsection;
see also \cite{MR2386535,MR2537081,MR2899874}
for closely related constructions.
The critical locus of the SYZ fibration
$
 \pi_B: Y \to B
$
is given by
$
 Z \times \{ (0,0) \} \subset \NCx \times \bC^2,
$
which is the fixed locus
of the $\bS^1$-action.
Hence the discriminant locus of $\pi_B$ is given by
$
 \Gamma = \Piloc' \times \{ \epsilon \} \subset B,
$
where
$
 \Piloc' \coloneqq \pi_\epsilon(Z) \subset N_\bR
$
is essentially the amoeba of $Z$,
except that the map $\pi_\epsilon$
differs from the logarithm map $\Log$
by $\phi_\epsilon$.
The complement of the discriminant locus
will be denoted by $B^\sm \coloneqq B \setminus \Gamma$.
The SYZ fibration induces an integral affine structure on $B^\sm$.
The corresponding local integral affine coordinates $\{ x_j \}_{j=1}^3$
give local systems $T_\bZ B^\sm$ and $T_\bZ^* B^\sm$,
generated by $\{ \partial/\partial x_j \}_{j=1}^3$
and $\{ d x_j \}_{j=1}^3$ respectively.

A choice of a section of $\pi_B$ induces a symplectomorphism
\begin{align}
 \pi^{-1}(B^\sm) \cong T^*B^\sm / T_\bZ^* B^\sm
\end{align}
given by the action-angle coordinates \cite{MR0596430}.
The \emph{semi-flat mirror} of $Y$ is defined by
\begin{align}
 \Yv^\semiflat \coloneqq T B^\sm / T_\bZ B^\sm,
\end{align}
equipped with the natural complex structure $J_{\Yv^\semiflat}$
such that the holomorphic coordinates are given by
$
 \lc z_j = \exp 2\pi \lb x_j + \sqrt{-1} y_j \rb \rc_{j=1}^3.
$
Here $\{ y_j \}_{j=1}^3$ are the coordinates on the fiber
corresponding to $\{ x_j \}_{j=1}^3$.
To obtain the SYZ mirror $\Yv$,
one first correct the semi-flat complex structure by contributions
of the holomorphic disks bounded by Lagrangian torus fibers,
and then add fibers over $\Gamma = B \setminus B^\sm$.

Instead of correcting complex structures of the semi-flat mirror,
\cite{MR3502098} considers the subset
$B^\reg \subset B$
obtained by removing $\pi(p^{-1}(U_{Z} \times \bC))$
from $B$.
Here $U_{Z} \subset \NCx$ is a sufficiently small neighborhood
of $Z$ containing the support of $\chi$.
The connected components of $B^\reg$ are in one-to-one correspondence
with elements of $A$,
and all fibers over $B^\reg$ are \emph{tautologically unobstructed}
(i.e., they do not bound any non-constant holomorphic disks).
Let $U_\bsalpha$ denote the connected component of $B^\reg$
corresponding to $\bsalpha \in A$.
The semi-flat mirror $T U_\bsalpha / T_\bZ U_\bsalpha$
with coordinates $(z_{\bsalpha, 1}, z_{\bsalpha, 2}, z_{\bsalpha,3})$
can be completed to a torus
$
 \Utilde_\bsalpha
  \coloneqq \Spec \bC[z_{\bsalpha, 1}^{\pm 1}, z_{\bsalpha, 2}^{\pm 1}, z_{\bsalpha,3}^{\pm 1}].
$
Motivated by the counting of Maslov index two disks
in a partial compactification of $Y$,
\cite{MR3502098} glues
$\Utilde_\bsalpha$ and $\Utilde_\bsbeta$
for $\bsalpha, \bsbeta \in A$ together by
\begin{align} \label{eq:gluing}
\begin{split}
 z_{\bsalpha, 1} &= (1+z_{\bsbeta,3})^{\beta_1-\alpha_1} z_{\bsbeta,1}, \\
 z_{\bsalpha, 2} &= (1+z_{\bsbeta,3})^{\beta_2-\alpha_2} z_{\bsbeta,2}, \\
 z_{\bsalpha, 3} &= z_{\bsbeta, 3}.
\end{split}
\end{align}
These local coordinates are related to coordinates
$(w_1, w_2, w_3)$ of the dense torus by
\begin{align} \label{eq:toric_coord}
\begin{split}
 z_{\bsalpha,1} &= w_1 w_3^{-\alpha_1}, \\
 z_{\bsalpha,2} &= w_2 w_3^{-\alpha_2}, \\
 z_{\bsalpha,3} &= w_3-1.
\end{split}
\end{align}
Let $\Sigma$ be the fan in $M_\bR \oplus \bR$
associated with the coherent unimodular triangulation $\cP$,
and $X_\Sigma$ be the associated toric variety.
Let further $K$ be an anticanonical divisor in $X_\Sigma$
defined by the function $p \coloneqq \chi_{\bszero,1} - 1$,
where $\chi_{\bn,k} : X_{\Sigma} \to \bC$ is the function
associated with the character $(\bn,k) \in N \oplus \bZ$
of the dense torus of $X_\Sigma$.
By adding torus-invariant curves
to $\bigcup_{\bsalpha \in A} \Utilde_\bsalpha$,
one obtains the complement $\Yv \coloneqq X_\Sigma \setminus K$
of the anti-canonical divisor
\cite[Theorem 1.7]{MR3502098}.

\subsection{Coordinate ring of the mirror manifold}
 \label{sc:coord_ring}

One has
$
 H^0(\cO_{X_\Sigma})
  = \bigoplus_{(\bn,k) \in \cC} \chi_{\bn,k}
$
where
\begin{align}
 \cC \coloneqq \{ (\bn, \ell) \in N \oplus \bZ \mid
  \bn(\bm) + \ell \ge 0 \text{ for any } \bm \in A \}.
\end{align}
If we define a function
$
 \ell_1 \colon N \to \bZ
$
by
\begin{align}
 \ell_1(\bn) \coloneqq \min \lc \ell \in \bZ \relmid (-\bn,\ell) \in \cC \rc,
\end{align}
then the set
$
 \lc p^{i} \chi_{-\bn,\ell_1(\bn)} \rc_{(\bn,i) \in N \times \bZ}
$
forms a basis of the algebra
$
 H^0(\cO_\Yv)
  = H^0(\cO_{X_\Sigma})[p^{-1}].
$
The product structure is given by
\begin{align}
 p^{i} \chi_{-\bn,\ell_1(\bn)}
  \cdot
 p^{i'} \chi_{-\bn',\ell_1(\bn')}
  &=
 p^{i+i'} \chi_{-\bn-\bn',\ell_1(\bn)+\ell_1(\bn')} \\
  &=
 p^{i+i'} \chi_{\bszero,\ell_2(\bn,\bn')}
  \cdot
 \chi_{-\bn-\bn',\ell_1(\bn+\bn')} \\
  &=
      p^{i+i'}(1+p)^{\ell_2(\bn,\bn')}
  \cdot
 \chi_{-\bn-\bn',\ell_1(\bn+\bn')} \\
&=
 \sum_{j=0}^{\ell_2(\bn,\bn')} \binom{\ell_2(\bn,\bn')}{j}
  p^{i+i'+j} \chi_{-\bn-\bn',\ell_1(\bn+\bn')},
   \label{eq:product1}
\end{align}
where the function $\ell_2 \colon N \times N \to \bZ$ is defined by
\begin{align} \label{eq:ell}
 \ell_2(\bn, \bn') = \ell_1(\bn) + \ell_1(\bn') - \ell_1(\bn+\bn').
\end{align}

%

\section{Base-admissible Lagrangian sections}
 \label{sc:ALS}

\subsection{Liouville domains}

A pair $(\Xin, \theta)$
of a compact manifold $\Xin$ with boundary
and a one-form $\theta$ on $\Xin$
is called a \emph{Liouville domain}
if\begin{itemize}
\item $\omega \coloneqq d \theta$ is a symplectic form on $\Xin$,
\item the Liouville vector field $V_{\theta}$,
determined uniquely by the condition $\iota_{V_{\theta}} \omega = \theta$,
points strictly outward along $\partial \Xin$.
\end{itemize}
The one-form $\theta$ is called the \emph{Liouville one-form}.
The manifold
\begin{align}
 \widehat{\Xin} \coloneqq \Xin \cup_{\partial \Xin} [1,\infty) \times \partial \Xin
\end{align}
obtained by gluing the positive symplectization
$
 \lb [1,\infty) \times \partial \Xin, d \lb r \lb \theta|_{\Xin} \rb \rb \rb
$
of the contact manifold $\lb \partial \Xin, \theta|_{\Xin} \rb$ to $\Xin$
along $\partial \Xin$
is called the \emph{Liouville completion} of $(\Xin, \theta)$.
An exact symplectic manifold
obtained as the Liouville completion of a Liouville domain
will be called a \emph{Liouville manifold}.
The extension of the one-form $\theta$ to $X$ will be denoted by $\theta$ again
by abuse of notation.
The coordinate $r$ on the symplectization end
$[1,\infty) \times \partial \Xin$
corresponding to $[1, \infty)$
is called the \emph{Liouville coordinate}.

If $(X, J)$ is a Stein manifold
with an exhaustive plurisubharmonic function
$S \colon X \to \bR$
whose critical values are less than $K \in \bR$,
%
%
%
then the manifold
$
 \Xin \coloneqq S^{-1}((-\infty,K])
$
is a Liouville domain
with a Liouville one-form
$
 \theta \coloneqq -d^c S.
$
Under the additional assumption that the gradient flow of $S$ is complete, the Liouville completion may be identified with $X$.

\subsection{Base-admissible Lagrangian sections}

Let $S$ be an exhaustive plurisubharmonic function
on $\NCx$
defined by
\begin{align} \label{eq:NCxS}
 S(\bw) = \frac{1}{2} |\br|^2 = \frac{1}{2}(r_1^2+r_2^2)
\end{align}
in the logarithmic coordinates
$
 \bw
  = (w_1, w_2)
  = \lb e^{r_1 + \sqrt{-1} \theta_1}, e^{r_2 + \sqrt{-1} \theta_2} \rb.
$
One has
\begin{align}
 d S = r_1 d r_1 + r_2 d r_2
\end{align}
and
\begin{align}
 \theta_{\NCx} \coloneqq
 - d^c S = r_1 d \theta_1 + r_2 d \theta_2.
\end{align}
Let
$
 L_0 \coloneqq N_{\bR^{>0}} \times \bR^{>0}
$
be the positive real locus of
$
 \NCx \times \bCx
  = (\NCx \times \bC) \setminus \Dbar,
$
which is a Lagrangian submanifold
diffeomorphic to $\bR^3$.
Since $L_0$ is disjoint
from the tubular neighborhood
$U_{Z \times 0} \subset \NCx \times \bC$
of the center $Z \times 0$ of the blow-up,
it lifts to a Lagrangian submanifold of $Y$.
By abuse of notation, we write the lifted Lagrangian again as $L_0$.
More generally,
we consider the following type of Lagrangians:


\begin{definition} \label{df:ad_Lag}
An exact Lagrangian section $L$ of the SYZ fibration \eqref{eq:piB} of $Y$
is \emph{base-admissible}
if the following conditions are satisfied:
\begin{itemize}
 \item
$L$ is fibered over a Lagrangian submanifold
$\Lbar$ in $\NCx \setminus U_{Z}$;
\begin{align}
 L = \Lbar \times \bR^{>0}
  \subset (\NCx \setminus U_{Z}) \times \bC
  \subset Y.
\end{align}
 \item
$\Lbar$ is Legendrian at infinity,
i.e.,
$
 \left. \theta_{\NCx} \right|_{\Lbar}=0
$
outside of a compact set.
\end{itemize}
\end{definition}

It is clear from \pref{df:ad_Lag}
that base-admissible Lagrangian sections of $\pi_B \colon Y \to B$
are in one-to-one correspondence
with Lagrangian sections $\Lbar$ of $\Log \colon \NCx \to N_\bR$
which are disjoint from $U_{Z}$ and
satisfying the Legendrian condition at infinity.

\subsection{Framed Lagrangian sections}

For each lattice point $\bsalpha \in A$,
consider the polynomial
\begin{align}
 h^\bsalpha(\bw) = - t^{-\nu(\bsalpha)}(1-\phi_\bsalpha(\bw)) \bw^\bsalpha
  + \sum_{\bsbeta \in A \setminus \{ \bsalpha \}} t^{-\nu(\bsbeta)}(1-\phi_\bsbeta(\bw)) \bw^\bsbeta
\end{align}
obtained by flipping the sign of one term in \eqref{eq:ht1}.
The corresponding hypersurface will be denoted by
\begin{align}
 Z^\bsalpha \coloneqq \{ \bw \in \NCx \mid h^\bsalpha_{t,1}(\bw) = 0 \}.
\end{align}

\begin{lemma}\label{lem:amoebas_coincide}
The amoeba of $Z^\bsalpha$ coincides with that of $Z$.
\end{lemma}

\begin{proof}
If $\Log(\bw) \in O_\tau$ for $\tau \in \cP^{(1)}$
such that $\partial \tau = \{ \bsalpha, \bsbeta \}$,
then one has
\begin{align} \label{eq:Z-edge}
\begin{split}
 h(\bw)
  &= t^{-\nu(\bsalpha)} (1 - \phi_\bsalpha(\bw)) \bw^\bsalpha
    + t^{-\nu(\bsbeta)} (1 - \phi_\bsbeta(\bw)) \bw^\bsbeta, \\
 h^\bsalpha(\bw)
  &= - t^{-\nu(\bsalpha)} (1 - \phi_\bsalpha(\bw)) \bw^\bsalpha
    + t^{-\nu(\bsbeta)} (1 - \phi_\bsbeta(\bw)) \bw^\bsbeta.
\end{split}
\end{align}
If $\Log(\bw) \in O_\sigma$
where $\sigma \in \cP^{(2)}$ is the simplex
whose vertices are $\bsalpha, \bsbeta, \gamma \in A$,
then one has
\begin{align}
\begin{split}
 h(\bw)
  &= t^{-\nu(\bsalpha)} (1 - \phi_\bsalpha(\bw)) \bw^\bsalpha
    + t^{-\nu(\bsbeta)} (1 - \phi_\bsbeta(\bw)) \bw^\bsbeta
    + t^{-\nu(\gamma)} (1 - \phi_\gamma(\bw)) \bw^\gamma, \\
 h^\bsalpha(\bw)
  &= - t^{-\nu(\bsalpha)} (1 - \phi_\bsalpha(\bw)) \bw^\bsalpha
    + t^{-\nu(\bsbeta)} (1 - \phi_\bsbeta(\bw)) \bw^\bsbeta
    + t^{-\nu(\gamma)} (1 - \phi_\gamma(\bw)) \bw^\gamma.
\end{split}
\end{align}
By choosing a coordinate of $M$
in such a way that $\bsbeta-\bsalpha = (1,0)$
and $\gamma-\bsalpha=(0,1)$,
one can easily show that
the amoebas are identical in both cases.
\end{proof}

Choose a sufficiently large $t$
so that the connected components
of the complement of $\Piloc \coloneqq \Log(Z)$
are labeled by $A$ as
\begin{align} \label{eq:Q_alpha}
 N_\bR \setminus \Piloc
  = \coprod_{\bsalpha \in A} \cQ_\bsalpha
\end{align}
just as in \eqref{eq:C_alpha}.
For an interior lattice point $\bsalpha \in A \cap \Int \Delta$,
a \emph{tropical Lagrangian section}
is an exact Lagrangian section of the restriction of
$\Log \colon \NCx \to N_\bR$
to the inverse image of $\cQ_\bsalpha$
with boundary in $Z^\bsalpha$
which agrees with the parallel transport of $\partial L$
along a segment in $\bC$
in a small neighborhood of $\partial L$
\cite[Definitions 3.7 and 3.16]{MR2529936}.
The prototypical example of a tropical Lagrangian section
is the restriction of the positive real Lagrangian
\begin{align}
 \Lbar_0 \coloneqq N_{\bR^{>0}}
  = \{ \bw=(w_1,w_2) \in \NCx \mid w_1, w_2 \in \bR^{>0} \}
\end{align}
to the fibers over $\cQ_\bsalpha$.

\begin{lemma} \label{lm:Uloc}
A tropical Lagrangian section does not intersect
a sufficiently small tubular neighborhood $U_{Z}$ of $Z$.
\end{lemma}

\begin{proof}
Since a tropical Lagrangian section is compact
and $Z$ is closed,
it suffices to show that a tropical Lagrangian section
does not intersect $Z$.

If $\Log(\bw) \in O_\tau$ for $\tau \in \cP^{(1)}$
such that $\partial \tau = \{ \bsalpha, \bsbeta \}$,
then it follows from \eqref{eq:Z-edge} that
$\bw^{\bsbeta-\bsalpha}$ is in $\bR^{>0}$ for $\bw \in Z$ and
$\bR^{<0}$ for $\bw \in Z^\bsalpha$.
It follows that a tropical Lagrangian section does not intersect $Z$
in $\Log^{-1}(O_\tau)$ for $\tau \in \cP^{(1)}$.

If $\Log(\bw) \in O_\sigma$ for $\sigma \in \cP^{(2)}$,
then a tropical Lagrangian section agrees
with the positive real Lagrangian $\Lbar_0$
in the neighborhood
$\partial \cQ \cap O_\sigma$
of the vertex of $\Pi_\infty$
dual to $\tau$
\cite[Lemma 3.18]{MR2529936}.
The positive real Lagrangian is clearly disjoint from $Z$, and
\pref{lm:Uloc} is proved.
\end{proof}

One can use the complex structure of $\NCx$
to view $\NCx$ as the trivial $N_\bR / N$-bundle
$T N_\bR / T_\bZ N_\bR$ over $N_\bR$,
whose universal cover is the trivial $N_\bR$-bundle
$T N_\bR$.
A section of $T N_\bR$ can be identified
with a function on $N_\bR$
with values in $N_\bR$.
Since the tropical Lagrangian section agrees with the positive real Lagrangian
near $\cQ \cap O_\tau$ for $\tau \in \cP^{(2)}$,
a lift $\cQ \to T N_\bR$
of a tropical Lagrangian section $\cQ \to \NCx$
to the universal cover $T N_\bR \to \NCx \cong T N_\bR / T_\bZ N_\bR$
takes values in $N$ near $\cQ \cap O_\tau$.
The Hamiltonian isotopy class of a tropical Lagrangian section
is determined by the values
$(n_\tau)_{\tau \in \cP^{(2)}} \in N^{\cP^{(2)}}$
of its lifts near the vertices of $\Piloc$
\cite[Proposition 3.20]{MR2529936}.
Two lifts come from the same section
if and only if they are related by an overall shift by $N$.

For an edge $\tau \in \cP^{(1)}$
in the interior of $\Delta$,
choose a coordinate of $M$ in such a way that
$\tau$ is the line segment
between $\bsalpha = (0,0)$ and $\bsbeta = (1,0)$.
Then
one has
\begin{align}
 h^\bsalpha(\bw)
  = - t^{-\nu(\bsalpha)} + t^{-\nu(\bsbeta)} w_1
\end{align}
for $\bw \in O_\tau$,
so that
\begin{align}
 \Piloc \cap O_\tau
  = \{ (r_1, r_2) \in N_\bR \mid r_1 = - \nu(\bsalpha) + \nu(\bsbeta) \}.
\end{align}
It follows that a tropical Lagrangian section
is constant in the $w_1$-variable
above $O_\tau$.
A Lagrangian section $N_\bR \to \NCx$ is said to be \emph{framed}
if its restriction to $\cQ_\bsalpha$ is bounded by $Z^\bsalpha$
for any interior lattice point $\bsalpha \in A \cap \Int \Delta$.

If $\sigma$ and $\sigma'$ are elements of $\cP^{(2)}$
adjacent to an edge $\tau \in \cP^{(1)}$
in the interior of $\Delta$,
then the condition that the boundary of a Lagrangian section
lies in $Z$ implies that
\begin{align} \label{eq:comp_edge}
 \la n_\sigma - n_{\sigma'}, \bsalpha - \bsbeta \ra = 0.
\end{align}
For a collection
$(n_\sigma)_{\sigma \in \cP^{(2)}} \in N^{\cP^{(2)}}$
of elements of $N$,
there exists a framed Lagrangian section
whose lift takes the value $n_\sigma$
on $O_\sigma$
if and only if \eqref{eq:comp_edge} is satisfied
for any edge $\tau \in \cP^{(1)}$.


\subsection{Legendrian condition at infinity}

Recall from \pref{df:ad_Lag} that
a Lagrangian submanifold $\Lbar \subset \NCx$ is
Legendrian at infinity
if $d^c S|_{\Lbar} = 0$
outside of a compact set.
A direct calculation shows that
the graph
\begin{align}
 \Gamma_{d f} \coloneqq \lc (r_1, \theta_1, r_2, \theta_2) \in \NCx \relmid
  \theta_1 = \frac{\partial f}{\partial r_1}, \
  \theta_2 = \frac{\partial f}{\partial r_2} \rc
\end{align}
of the differential of a function $f \colon N_\bR \to \bR$
satisfies the Legendrian condition
$d^c S|_{\Gamma_{df}} = 0$
if and only if
\begin{align}
 \lb r_1 \frac{\partial}{\partial r_1} + r_2 \frac{\partial}{\partial r_2} \rb
  \frac{\partial f}{\partial r_1}
   =
 \lb r_1 \frac{\partial}{\partial r_1} + r_2 \frac{\partial}{\partial r_2} \rb
  \frac{\partial f}{\partial r_2}
   = 0.
\end{align}
This happens
if $f$ homogeneous of degree one:
\begin{align}
 \lb r_1 \frac{\partial}{\partial r_1} + r_2 \frac{\partial}{\partial r_2} \rb f = f.
\end{align}

\begin{proposition} \label{pr:Legendrian}
Any framed Lagrangian section can be made Legendrian at infinity
by a Hamiltonian isotopy.
\end{proposition}

\begin{proof}
We can choose a framed Lagrangian
in such a way that it coincides with the positive real Lagrangian
in the neighborhood of each leg of $\Piloc$
outside of a compact set.
Then the potential $f$ of the Lagrangian is linear
in that neighborhood.
Now one can choose arbitrary homogeneous function of degree one
which coincides with $f$ in the compact set
and in the neighborhood of each leg,
and the Lagrangian generated by this function has the desired property.
\end{proof}

\subsection{SYZ transformation}

Let $L$ be a base-admissible Lagrangian section of $\pi_B \colon Y \to B$
associated with a framed Lagrangian section $\Lbar$
of $\Log \colon \NCx \to N_\bR$.
Let further $\tau \in \cP^{(1)}$ be an edge in the interior of $\Delta$
and $\ell \subset \Yv$ be the corresponding torus-invariant curve.
We use the same coordinates as in Section \ref{subsec:SYZ_mirror_constr}.

For each interior lattice point $\bsalpha \in A \cap \Int \Delta$,
a framed Lagrangian section $\Lbar$ restricts
to a tropical Lagrangian section $\Lbar_\bsalpha$
over $\cQ_\bsalpha$.
The fiberwise universal cover of the restriction
of $\Log \colon \NCx \to N_\bR$ to $\cQ_\bsalpha$
can be identified with $T\* \cQ_\bsalpha$,
with the positive real Lagrangian as the zero-section.
We write the lift of $\Lbar_\bsalpha$
as the graph of a one-form
\begin{align}
 \omega = \xi_1 d y_1 + \xi_2 d y_2,
\end{align}
where $\xi_1$ and $\xi_2$ are functions on $N_\bR$
satisfying
\begin{align} \label{eq:Lag_cond}
 \frac{\partial \xi_1}{\partial y_2} - \frac{\partial \xi_2}{\partial y_1}
  &= 0.
\end{align}

The \emph{semi-flat SYZ transform} of $\Lbar_\bsalpha$
is the trivial bundle on $T \cQ_\bsalpha$,
equipped with the connection
\begin{align}
 \nabla_\bsalpha \coloneqq d + 2 \pi \sqrt{-1} \omega = d + 2 \pi \sqrt{-1} (\xi_1 d y_1 + \xi_2 d y_2).
\end{align}
In cases without quantum corrections, this gives a holomorphic line bundle mirror to the given Lagrangian section \cite{Leung-Yau-Zaslow_SLHYM}. In general, however, quantum corrections have to be taken into account \cite{MR3022713, MR3164868, MR3439224}.
In our case, due to the nontrivial gluing formulas \eqref{eq:gluing}, the semi-flat SYZ transforms of $\Lbar_\bsalpha$ and $\Lbar_\bsbeta$ do not coincide over the intersection
\begin{align}
 \Utilde_\bsalpha \cap \Utilde_\bsbeta = T \ell / T_\bZ \ell,
\end{align}
where $\ell \subset \Pi_\infty$ is the edge of the intersection of the connected components $\cQ_\bsalpha, \cQ_\bsbeta \subset N_\bR \setminus \Piloc$
which is dual to $\tau \in \cP^{(1)}$, but are related by
\begin{align}
\nabla_\bsalpha = \nabla_\bsbeta + \sqrt{-1} \langle df,\bsbeta-\bsalpha \rangle d\text{arg}(1+z_3),
\end{align}
where $f$ is a primitive of $\Lbar$, i.e. $\xi_i = \partial f / \partial x_i$ for $i=1,2$.

Since $\Lbar_\bsalpha$ and $\Lbar_\bsbeta$ share the same boundary in $Z$ over $\ell$ and the defining equation for $Z$ is given as in \eqref{eq:eqn_Zloc_leg}, we have $k_{\bsalpha\bsbeta} \coloneqq \langle df,\bsbeta-\bsalpha \rangle \in \bZ$, so we may modify $\nabla_\bsbeta$ to the {\em gauge equivalent} connection
\begin{align}
\nabla_\bsbeta' \coloneqq \nabla_\bsbeta + \sqrt{-1} k_{\bsalpha\bsbeta} d\text{arg}(1+z_3).
\end{align}
Now $\nabla_\bsalpha$ and $\nabla_\bsbeta'$ glue to give a connection $\nabla_{\bsalpha\bsbeta}$ on the chart $\Utilde_\bsalpha \cup \Utilde_\bsbeta \subset \Yv$. It is clear that the cocycle condition is satisfied, so the connections $\{\nabla_{\bsalpha\bsbeta}\}$ define a global $U(1)$-connection over $\Yv$ whose curvature has trivial $(0,2)$-part since $L$ is Lagrangian.
This produces a holomorphic line bundle $\cF(L)$ over $\Yv$, called the \emph{SYZ transform} of $L$.

To determine the isomorphism class of $\cF(L)$,
let $\ell \subset \Pi_\infty$ be an edge
on the boundary of a connected component
$C_{\bsalpha,\infty} \subset N_\bR \setminus \Pi_\infty$
of the complement of the tropical curve $\Pi_\infty$.
We can choose a coordinate on $M$
in such a way that the endpoints of the edge $\tau$ is given by
$\bsalpha = (0,0)$ and $\bsbeta=(1,0)$.
A subset of the torus-invariant curve in $\Yv$
associated with the edge $\tau \in \cP^{(1)}$ dual to $\ell$
can naturally be identified with $T \ell / T_\bZ \ell$.
Let $\sigma, \sigma' \in \cP^{(2)}$ be the faces
adjacent to $\tau$, then
the degree of the restriction of $\cF(L)$
to $T \ell / T_\bZ \ell$ is given by
\begin{align}
 \frac{\sqrt{-1}}{2 \pi} \int_{T \ell/T_\bZ \ell} F_{\nabla_\bsalpha'}^{1,1}
  = - \int_{T \ell/T_\bZ \ell} \frac{\partial \xi_2}{\partial y_2'} d x_2' \wedge d y_2'
  = - \int_{\ell} \frac{\partial \xi_2}{\partial y_2'} d y_2'
  = \xi_2(s_\sigma) - \xi_2(s_{\sigma'}),
\end{align}
where $s_\sigma, s_{\sigma'} \in N_\bR$ are the endpoints of $\ell$
dual to $\sigma, \sigma' \in \cP^{(2)}$.
More generally, it can be shown that the degree of the restriction of $\cF(L)$
to $T \ell / T_\bZ \ell$ is given by $\langle df, (\bsbeta - \bsalpha)^\perp\rangle|_{s_{\sigma'}}^{s_\sigma}$.
This shows that the isomorphism class of $\cF(L)$ depends only
on the Hamiltonian isotopy class of $L$, hence proving \pref{th:main1}.

\section{Standard wrapped Floer theory}
 \label{sc:sWFT}

\subsection{Basic geometric objects}

Let $(X, \theta)$ be a Liouville manifold.
The induced contact structure on $\partial \Xin$ will be denoted by
$
 \xi \coloneqq \ker \lb \theta|_{\partial \Xin} \rb,
$
and the Liouville coordinate on the symplectization end
$
 [1,\infty) \times \partial \Xin
$
will be denoted by $r$.
We assume that the canonical bundle of $X$ is trivial,
and fix its trivialization.

\subsection{}

The vector field $\cZ$ on $X$
dual to the Liouville one-form $\theta$
with respect to the symplectic form $d \theta$
is called the \emph{Liouville vector field}.
It is given by $r \partial_r$ on the symplectization end.

\subsection{}

The \emph{Reeb vector field} $\cR$
on the contact manifold
$\lb \partial \Xin, \theta|_{\Xin} \rb$
is defined by
$\cR \in \Ker d \theta$
and $\theta(\cR)=1$.

\subsection{}

A Lagrangian submanifold
$L \subset X$
is \emph{Liouville-admissible}
if it is the completion
$\Lin \cup [1,\infty) \times \partial \Lin$
of a Lagrangian submanifold
$\Lin \subset \Xin$ such that $\theta|\Lin \in \Omega^1(\Lin)$ vanishes to infinite order along the boundary $\partial \Lin$. We assume that all such Lagrangians are spin and are equipped
with brane structures
in the sense of \cite{MR2441780}. 

\subsection{}

A \emph{Liouville-admissible Hamiltonian} is a positive function
$H \colon X \to \bR^{>0}$
which is $\lambda r$ outside of a compact set
for a positive real number $\lambda$
called the \emph{slope} of $H$.
The set of Liouville-admissible Hamiltonians of slope $\lambda$
is denoted by $\scrHLa(X)_\lambda$
and we set
$
 \scrHLa(X) \coloneqq \bigcup_{\lambda \in \bR^{>0}} \scrHLa(X)_\lambda.
$

\subsection{}

A \emph{Liouville-admissible almost complex structure} is
a compatible almost complex
which outside of a compact set in the symplectization
is the direct sum
of an almost complex structure on $\xi$
and the standard complex structure
on the rank 2 bundle
spanned by the Liouville vector field
and the Reeb vector field, i.e.,
$J \cZ = \cR$.
The set of Liouville-admissible complex structures
is denoted by $\scrJLa(X)$.

\subsection{}

Let $H \colon X \to \bR$ be a function
on a symplectic manifold $X$ and
$L$ be a Lagrangian submanifold.
%
%
A \emph{Hamiltonian chord}
is a trajectory $x \colon [0,1] \to X$
of the Hamiltonian flow
such that $x(0) \in L$ and $x(1) \in L$.
The set of Hamiltonian chords will be denoted by $\scrX(L,X; H)$.
We sometimes write $\scrX(L;H) \coloneqq \scrX(L,X;H)$,
if $X$ is clear from the context.
We also write the set of Hamiltonian chords
in a given relative homotopy class $\gamma \in \pi_1(X,L)$
as
\begin{align}
 \scrX(L;H)_\gamma
  \coloneqq \lc x \in \scrX(L;H) \relmid [x] = \gamma \rc.
\end{align}
A Hamiltonian chord $x$ is \emph{non-degenerate}
if the image $\varphi_1(L)$ of $L$
by the time-one Hamiltonian flow $\varphi_1 \colon X \to X$
intersects $L$ transversally
at the intersection point
corresponding to $x$.

\subsection{}

Let $\Sigma$ be a closed disc with $d+1$ boundary punctures
$\bszeta = \{ \zeta_0, \ldots, \zeta_d \}$,
which are called the \emph{points at infinity}.
We denote by
$\Sigmabar \coloneqq \Sigma \cup \bszeta$
the closed disk
obtained by filling in the punctures.
The connected component of the boundary of $\Sigma$
between $\zeta_i$ and $\zeta_{i+1}$,
which is homeomorphic to an open interval,
will be denoted by $\partial_i \Sigma$.
We also write
$
 \partial \Sigma \coloneqq \bigcup_{i=0}^d \partial_i \Sigma.
$
The moduli space of such discs
and its stable compactification
will be denoted by
$\cR^d$ and $\cRbar^d$ respectively.

\begin{remark}
As most of our calculations will actually take place at the cohomological level, we will be mostly interested in the cases when $d \leq 3$ in this paper.
\end{remark}

\subsection{}

A \emph{strip-like end} around a point $\zeta_i$
is a holomorphic embedding
\begin{align}
\left\{
\begin{aligned}
 & \epsilon \colon \bR^{\le 0} \times [0,1] \to \Sigma, \\
 & \epsilon^{-1}(\partial \Sigma) = \bR^{\le 0} \times \{ 0, 1 \}, \\
 & \lim_{s \to -\infty} \epsilon(s, -) = \zeta_i
\end{aligned}
\right.
\end{align}
if $i=0$, and
\begin{align}
\left\{
\begin{aligned}
 & \epsilon \colon \bR^{\ge 0} \times [0,1] \to \Sigma, \\
 & \epsilon^{-1}(\partial \Sigma) = \bR^{\ge 0} \times \{ 0, 1 \}, \\
 & \lim_{s \to \infty} \epsilon(s, -) = \zeta_i
\end{aligned}
\right.
\end{align}
otherwise. Strip-like ends can be chosen compatibly over the moduli spaces. 

\subsection{}
A \emph{Liouville-admissible Floer data} is a pair
\begin{align}
 (H, J) \in C^\infty([0,1], \scrHLa(X)) \times C^\infty([0,1],\scrJLa(X))
\end{align}
of families of Liouville-admissible Hamiltonians and
Liouville-admissible almost complex structures.

\subsection{}

Let $\Sigma$ be a closed disk
with $d+1$ boundary punctures.
A \emph{Liouville-admissible perturbation data}
$(K,J)$
consists of
\begin{enumerate}
 \item
a 1-form
$
 K \in \Omega^1(\Sigma,\scrHLa(X))
$
on $\Sigma$
with values in Liouville-admissible Hamiltonians
satisfying
\begin{align}
 \alitem{$K|_{\partial \Sigma}=0$, and}\\
 \alitem{$X_K=X_{H_{e}} \otimes \beta $
outside of a compact set,
where $H_e$ is of slope one and
$\beta$ is sub-closed
(i.e., $d\beta \leq 0$), and} \label{eq:perturbationdata}
\end{align}
 \item
a family
$J \in C^\infty(\Sigma, \scrJLa(X))$
of Liouville-admissible almost complex structures on $X$
parametrized by $\Sigma$.
\end{enumerate}
It is \emph{compatible}
with a sequence
$(\bsH, \bsJ) = (H_j, J_j)_{j=0}^d$
of Liouville-admissible Floer data
if
\begin{align} \label{eq:Fleor_compatible}
 \epsilon_j^* K = H_j(t) dt \quad \text{and} \quad
 J(\epsilon_j(s,t)) = J_j(t)
\end{align}
for any $j \in \{0, \ldots, d\}$ and any $t \in [0,1]$.

\subsection{}
A sequence
$
 \bx \coloneqq \lb x_k \in \scrX(L; H_k) \rb_{k=0}^d
$
of Hamiltonian chords
and a perturbation data $(K,J)$
allow us to define
\emph{Floer's equation}
\begin{align} \label{eq:Floer}
\left\{
\begin{aligned}
 & y \colon \Sigma \to \Ybar, \\
 & y(\partial \Sigma) \subset L, \\
 & \lim_{s \to \pm \infty} y(\epsilon_k(s, -)) = x_k, \quad k=0, \ldots, d,\\
 & (d y - X_{K})^{0,1} = 0,
\end{aligned}
\right.
\end{align}
where $X_K$ is the one-form with values in Hamiltonian vector fields on $\Ybar$
associated with $K$.
Outside of a compact set,
our choices of perturbation data agrees with that studied in \cite{MR2602848} and so
the $C^0$ estimates of Section 7 of \emph{loc.~cit.}~still hold.

\subsection{Cohomological constructions}

For a non-degenerate Liouville-admissible Hamiltonian $H$, the \emph{Floer complex} is defined by
\begin{align} \label{eq:Floer_complex}
 \CF^*(L;H)
  \coloneqq \bigoplus_{x \in \scrX(L;H)} |\frako_x|,
\end{align}
where $|\frako_x|$ is the one-dimensional $\bC$-normalized orientation space
associated to $x$
(see \cite[Section 12]{MR2441780}).
For a pair $\bx = (x_0, x_1)$ of Hamiltonian chords,
the matrix element
of the \emph{Floer differential}
$\frakm_1$
is defined
by counting
the number of solutions
to Floer's equation \eqref{eq:Floer}
on the strip
$
 \Sigma \coloneqq \bR \times [0,1]
$
with perturbation data $K=Hdt$
up to $\bR$-translations.
The cohomology of the Floer complex
\eqref{eq:Floer_complex}
is denoted by
$
 \HF^*(L;H),
$
and called the \emph{Floer cohomology}.

\subsection{}

A \emph{Liouville-admissible sequence of Hamiltonians}
is a sequence $(H_m)_{m=1}^\infty$
of Liouville-admissible Hamiltonians
satisfying the following conditions:
\begin{enumerate}
 \item \label{it:LasH1}
For each $m \in \bZ^{>0}$,
the set $\scrX(L;H_m)$ consists only of non-degenerate chords.
 \item \label{it:LasH2}
The slopes $\lambda_m$
of $H_m$
satisfy
$\lambda_m < \lambda_{m+1}$
and $\lambda_m+\lambda_{m'} \le \lambda_{m+m'}$
for any $m, m' \in \bZ^{>0}$.
\end{enumerate}
Note that (\ref{it:LasH2}) implies
$\lim_{m \to \infty} \lambda_m = \infty$.

\subsection{}
 \label{sc:fD1}

We fix a Liouville-admissible sequence
$(H_m)_{m=1}^\infty$ of Hamiltonians
and a sequence $(J_m)_{m=1}^\infty$
of Liouville-admissible almost complex structures.
In addition,
for each $m \in \bZ^{>0}$,
we fix a Liouville-admissible perturbation data
$K(m, m+1)$
on the strip
$\bR \times [0,1]$,
which is compatible with the pair
$((H_m, J_m), (H_{m+1}, J_{m+1}))$
of Floer data.
For any $n > m$,
by gluing $(K(i,i+1))_{i=m}^{n-1}$,
we obtain a perturbation data
$K(m,n)$
on the strip
$\bR \times [0,1]$,
which is compatible with the pair
$((H_m, J_m), (H_{n}, J_{n}))$
of Floer data.
By counting numbers of solutions
to Floer's equation \eqref{eq:Floer}
on the strip
with respect to this perturbation data,
we obtain the \emph{continuation map}
\begin{align} \label{eq:cont1}
 \CF^*(L;H_m) \to \CF^*(L;H_n)
\end{align}
on the Floer cochain complex.
A standard argument in Floer theory shows that
the continuation map
commutes with the Floer differential,
and induces the continuation map
\begin{align} \label{eq:cont2}
\kappa_{m,n}:\HF^*(L;H_m) \to \HF^*(L;H_n)
\end{align}
on the Floer cohomology.
The colimit
\begin{align} \label{eq:HW_def}
 \HW(L) \coloneqq \varinjlim_m
  \HF(L;H_m)
\end{align}
with respect to the continuation map \eqref{eq:cont2} is called
the \emph{wrapped Floer cohomology}.

\subsection{}
 \label{sc:fD2}
The most important Floer theoretic operation
on wrapped Floer cohomology in this paper
will be the pair-of-pants product.
To define it, for any $m,n \in \bZ^{>0}$,
we fix a Liouville-admissible perturbation data
$K(m,n,m+n)$
on a disk with three punctures,
which is compatible with the triple
\begin{align}
 ((H_m, J_m), (H_n, J_n), (H_{m+n}, J_{m+n}))
\end{align}
of Floer data.
This allows us to define a linear map
\begin{align}
 \frakm_2 \colon \CF^*(L;H_n) \otimes \CF^*(L;H_m)
  \to \CF^*(L;H_{m+n})
\end{align}
by counting numbers of solutions to Floer's equation.
A standard arguments in Floer theory
shows that
$\frakm_2$ satisfies the Leibniz rule
with respect to $\frakm_1$,
and hence induces a map on the Floer cohomology.
As is standard in Floer theory,
the fact that this pair-of-pants product is well-behaved comes
from the existence of certain auxilliary moduli spaces
which are enhanced with suitable choices of Floer data.
For example,
in order to show that this product is actually well-defined on the direct limit,
we make the following construction:

\subsection{}
 \label{sc:fD3}

For any $m_1, m_2, m_3 \in \bZ^{>0}$
with $m_1<m_2$,
we fix a one-parameter family
$\lb K_\tau(m_1, m_2, m_3) \rb_{\tau \in [0,1]}$
of Liouville-admissible perturbation data
such that
$K_0(m_1, m_2, m_3)$ is the gluing of $K(m_1,m_3,m_1+m_3)$
with $K(m_1+m_3, m_2+m_3)$ and
$K_1(m_1, m_2, m_3)$ is the gluing of $K(m_1,m_2)$
with $K(m_2,m_3,m_2+m_3)$.
We also fix the analogous data for the case
when the continuation map occur along the other positive strip-like end.
A standard cobordism argument using this family shows that
the product is well-defined on the direct limit.

\subsection{}
 \label{sc:fD4}

For any $m_1,m_2,m_3 \in \bZ^{>0}$ and any $n \in \bZ^{\ge 0}$,
we set
$
 m \coloneqq \sum_{i=1}^3 m_i +n.
$
We then choose a family $K(m_1,m_2,m_3)$ of Liouville-admissible perturbation data
on the universal family of disks with 4 punctures over the moduli space $\cRbar^3$,
which is compatible with $(H_m,J_m)$
along the negative end and $(H_{m_{i}},J_{m_{i}})$
along the three positive ends.
We assume that along one end of the boundary of $\cRbar^3$,
the family $K(m_1,m_2,m_3)$ restricts
to the fiber product of a perturbation datum in $K(m_1,m_2,m-m_3)$
(for some datum in the interior of the above homotopy)
with $K(m-m_3 ,m_3, m)$,
and at the other end of the boundary,
it restricts to the fiber product of a perturbation datum
in $K(m_2,m_3,m-m_1)$
(for some choice of data in the interior of the above homotopy)
with $K(m_1,m-m_1,m)$.
We also require that for sufficiently small gluing parameters,
in the `thin" regions of the holomorphic curves,
the perturbation data restricts exactly to the gluing of these perturbation data
following \cite[Section (9i)]{MR2441780}.

For generic choices of Floer data and perturbation data,
all moduli spaces above may be regularized.
A standard argument in Floer theory using the moduli space of solutions to Floer's equation
with respect to this perturbation data shows that the product on the Floer cohomology is associative.

\begin{remark} The extra flexibility in the parameter $n$ above may seem unusual, but is motivated by our intended applications where the theory is not as nicely behaved as in the standard Liouville case. \end{remark}

\subsection{Chain level structures}
We now explain how to enhance the above constructions to the chain level.
The existence of this chain level construction is important
when we discuss (the (split-)generation of) the derived category.
Still, as most of our computations will take place at the cohomological level,
we will be brief and refer the reader to \cite{MR2602848}.
To define the chain level structure,
we assume
that our Liouville-admissible families $H_m$ satisfy $\lambda_m=m$ and
that
\begin{align} \label{eq:controlLiouville}
 H_m(x)=\lambda_m r
\end{align}
in the region of $X$ defined by $r \ge 2$.
We similarly assume
that all our Liouville-admissible Lagrangians $L$ are conical over this region,
and that for any perturbation data,
\pref{eq:perturbationdata} holds over this region as well.  

\begin{definition}
The \emph{wrapped Floer complex}
of a Liouville-admissible Lagrangian $L$
is defined by
\begin{align}
 \CW^*(L) \coloneqq \oplus_m \CF^*(L,H_m)[q],
\end{align}
where $\deg q=-1$ and $q^2=0$.
\end{definition}
We fix perturbation data of the form $K(m, m+1)$ and
define a differential on this complex via the formula
\begin{align} \label{eq:telescope}
 \mu^1(x+qy)
  = (-1)^{|x|} \frakm_1(x)
   +(-1)^{|y|}(q\frakm_1(y)
   + \kappa_{m,m+1}(y)-y).
\end{align}
The cohomology of this complex gives the wrapped Floer cohomology
defined in \pref{eq:HW_def};
\begin{align}
 H^*(\CW^* (L)) \cong \HW^*(L).
\end{align}
Fix $d \ge 1$ and labels $p_f \in \lc 1, \cdots ,d \rc $
(possibly not distinct) indexed by a finite set $F$.
Let
$
 \bfp \colon F \to \lc 1, \cdots ,d \rc
$
be the map given by $f \mapsto p_f$.

\begin{definition}
A \emph{$\bfp$-flavored popsicle}
is a disk $\Sigma \in \cR^{d}$
with $d+1$ boundary punctures
together with holomorphic maps
$
 \phi_f \colon \Sigma \to Z
$
to the strip $Z \coloneqq \bR \times [0,1]$
which extend to an isomorphism $\Sigmabar \simto \Zbar$
such that $\phi_f(z_0) = -\infty$ and $\phi_f(z_{p_{f}})=\infty$.
The moduli space of $\bfp$-flavored popsicles
is denoted by $\cR^{d,\bfp}$.
\end{definition}

This moduli space admits a compactification
$\cRbar^{d,\bfp}$
over stable discs. Moreover, we can choose universally consistent strip-like ends over $\cRbar^{d,\bfp}$
in the sense of
\cite[Section 2.4]{MR2602848}.
Observe that if $\bfp$ is not injective,
then there is a symmetry group of $\Sym^{\bfp}$ of permutations of $F$ preserving $\bfp$. This admits a natural action on $\cR^{d,\bfp}$ which extends to the compactification $\cRbar^{d,\bfp}$.
Fix flavors $\bfp$ and weights $\bsm = (m_0, \cdots, m_d) \in \lb \bZ^{>0} \rb^{d+1}$
satisfying
\begin{align}
 m_0 = \sum_{i=1}^d m_i + \abs{F}.
\end{align}
We denote by $\cRbar^{d,\bfp, \bsm}$ the moduli space of popsicles with weights $\bsm$
(although this is just a copy of $\cRbar^{d,\bfp}$,
it is useful to separate these).
We equip this with perturbation data compatible with the strip-like ends $H_i dt$ 
and admissible complex structures $J_t$.\footnote{
Some care must be taken in the choice of complex structure over the moduli space,
see \cite[Section 3.2]{MR2602848}.}
This data is chosen universally consistently and equivariantly with respect to the $\Sym^{\bfp}$ actions. Denote these choices by $(K^{\bfp,\bsm}, J^{\bfp,\bsm})$. 

Given a collection of chords $x_i \in \scrX(L;H_i)$,
we can form the moduli space $\cRbar^{d,\bfp, \bsm}(\bx)$
of solutions to Floer's equation.
With generic choices of $(K^{\bfp,\bsm}, J^{\bfp,\bsm})$,
these spaces of maps have the expected dimension.
Whenever the expected dimension is zero,
counting these solutions with appropriate signs gives rise to operations $\mu^{\bfp,\bsm}$.
Out of these operations,
Abouzaid and Seidel constructs an $A_\infty$-structure
on the Floer complex $\CW^*(L,L).$ 

These constructions can easily be adapted
to a collection of Liouville-admissible Lagrangians,
giving rise to an $A_\infty$ category $\cW(X)$
whose objects are these Lagrangians and
whose morphisms are the Floer complexes.
Finally,
we embed this,
via the Yoneda embedding,
into the larger category
$D^\pi \cW(X) \coloneqq \perf \cW(X)$
of perfect $A_{\infty}$-modules over this category.
The smallest full triangulated subcategory of $D^\pi \cW(X)$
containing $\cW(X)$ will be denoted by $D^b \cW(X)$.

\begin{remark} The reader will note that we have used slightly more general Hamiltonians than those of the form $H_m=mH$ for a fixed admissible Hamiltonian. This is for two reasons, the first being that as one of their genericity constraints on the Hamiltonians, Abouzaid and Seidel
\cite[(39)]{MR2602848}
impose the condition that no point on $L$ is both an endpoint and a starting point of a Hamiltonian chord.
This is to rule out certain solutions of zero geometric energy
which can roughly speaking be thought of as constant curves
landing in triple intersections of the Lagrangians
after being perturbed by the Hamiltonian flow
(these solutions are problematic for transversality).
For the Hamiltonians we wish to choose,
all the Hamiltonian chords on the cylindrical end are Hamiltonian orbits.
However,
in our setting,
because we can choose more flexible families of Hamiltonians,
we can choose our data so that such constant curves are excluded a priori.The second consideration is practical:
by working with slightly more general Hamiltonians,
we can obtain smaller models for the Floer cohomology.
\end{remark}

\section{Adapted wrapped Floer theory}
 \label{sc:aWFT}


\subsection{}

Let $S \colon \NCx \to \bR$ be the standard exhaustive
plurisubharmonic function
defined in \eqref{eq:NCxS}.
%
A function
$
 H \colon \NCx \to \bR
$
is
\emph{homogeneous of degree one}
with respect to $S$
if
\begin{align}
 \alitem{
$-d^cS(X_H)=H$, and
}\\
 \alitem{
$dS(X_H)=0$.
}
\end{align}

\subsection{}

A positive function
$\Hb \colon \NCx \to \bR^{>0}$
is an \emph{admissible base Hamiltonian}
if it is homogeneous of degree one
outside of a compact set.

\subsection{}

For a function $H \colon N_\bR \to \bR$,
the composition $H \circ \Log \colon \NCx \to \bR$
will also be denoted by $H$
by abuse of notation.
The Hamiltonian vector field is given by
\begin{align}
 X_H
  = \frac{\partial H}{\partial r_1} \frac{\partial}{\partial \theta_1}
   + \frac{\partial H}{\partial r_2} \frac{\partial}{\partial \theta_2}.
\end{align}
One has
\begin{align}
 - d^c S(X_H)
  = r_1 \frac{\partial H}{\partial r_1} + r_2 \frac{\partial H}{\partial r_2},
\end{align}
so that $- d^cS(X_H) = H$
if and only if $H$ is homogeneous of degree one in the usual sense.

\subsection{}


A positive function $\Hba \colon \Ybarloc \to \bR^{>0}$ is
a \emph{base-admissible Hamiltonian}
if there exists an admissible base Hamiltonian
$\Hb \colon \NCx \to \bR^{> 0}$
and a compact set $K \subset \NCx$
such that 
\begin{itemize}
\item for any $y \in \Ybar \setminus \pibar_{\NCx}^{-1}(K)$,
one has
\begin{align}
 (\pibar_{\NCx})_*(X_{\Hba}(y))= X_{\Hb}(\pibar_{\NCx}(y)).
\end{align}
\item Outside of $\pibar_{\NCx}^{-1}(U_{Z})$, $\Hba$ is a $C^2$-small perturbation
of $\pibar_{\NCx}^{-1}(H_b)$
\end{itemize}
The set of base-admissible Hamiltonians on $\Ybarloc$
is denoted by $\scrHba(\Ybarloc)$.


\begin{proposition} \label{pr:Hba}
There exists a base-admissible Hamiltonian.
\end{proposition}

\begin{proof}
Recall from \eqref{eq:chiG2} that
$\chi |_{U_i} = \chi_i \circ G_i$.
The symplectic form on $U_i$ is invariant
under the $\bS^1$-action on $\NCx$
which preserves $\bw^{\bsalpha-\bsbeta}$.
The essential idea is to produce the base Hamiltonian
by gluing together local moment maps for these actions.
In more detail, set $F_i= |u|^2+|h(\bw)|^2$.
Then we have
 \begin{align} \label{eq:express1}
 \omega
 &= \pibar_{\NCx}^*\omega_{\NCx} + \pibar_{\bC_u}^* \omega_{\bC_u}
  - \frac{\epsilon}{4\pi} dd^c(\chi(G_i)\log(F_i)).
\end{align}
Let $X$ be the vector field on $\Ybar$,
which is $\partial_{\theta_{({\bsalpha_i}-{\bsbeta_i})^\bot}}$
with $(u, \bw)$ as coordinates;
it is characterized by
\begin{align}
 \iota_X d u
  = \iota_X d \bw^{{\bsalpha_i}-{\bsbeta_i}}
  = \iota_X d r_{({\bsalpha_i}-{\bsbeta_i})^\bot} = 0,
  \quad
 \iota_X d \theta_{({\bsalpha_i}-{\bsbeta_i})^\bot} = 1.
\end{align}
Then we have
\begin{align}
& \iota_X (\pibar_{\NCx})^* \omega_{\NCx}
 = - d r_{({\bsalpha_i}-{\bsbeta_i})^\bot}.
\end{align}
By invariance of the functions $F_i$ and $G_i$
under the local circle action generated by $X$,
we have
\begin{align}
 \iota_X dd^c \lb \chi(G_i)\log(F_i) \rb
  &= (\scrL_X - d \iota_X) d^c \lb \chi(G_i)\log(F_i) \rb \\
  &= - d \lb \iota_X d^c\lb \chi(G_i) \log(F_i) \rb \rb,
\end{align} 
so that
\begin{align}
 -\iota_X \omega
 &= d\lb r_{({\bsalpha_i}-{\bsbeta_i})^\bot}- \frac{\epsilon}{4\pi} \iota_X d^c\lb \chi(G_i) \log(F_i) \rb \rb.
\end{align}
If we define
$
 \rho_i \colon U_i \to \bR
$
by
\begin{align}
 \rho_i
  \coloneqq  r_{({\bsalpha_i}-{\bsbeta_i})^\bot}
  -\frac{\epsilon}{4\pi} \iota_X d^c\lb \chi(G_i) \log(F_i)\rb,
\end{align}
then
$
 \rho_i - \br_{({\bsalpha_i}-{\bsbeta_i})^\bot}
$
is a bounded function
whose support is contained
in the support of $\chi$. Let $R$ be a positive number
satisfying $R \gg R_0$.
Let $\Hb$ be a positive function on $N_\bR$
(which is also considered as a function on $\NCx$
by composing with $\Log \colon \NCx \to N_\bR$)
such that
\begin{align}
 \alitem{
$\Hb$ is homogeneous of degree one outside of a compact set,
}\\
 \alitem{
$\Hb|_{U_{\Pi_i}} = c_i \br_{(\bsalpha-\bsbeta)^\bot}/R$
where $c_i$ is defined in \pref{eq:ci}, and
}\\
 \alitem{
$\Hb(\br) = \abs{\br}/R$ outside of a neighborhood of $U_{\Pi}$.
}
\end{align}
Since the function $\frac{c_i \rho_i}{R}$ agrees with $\pibar_{\NCx}^* \Hb$
outside the support of $\chi$,
we may glue $\frac{c_i \rho_i}{R}$ for $i=1, \ldots, \ell$ and
$\pibar_{\NCx}^* \Hb$ together
to obtain a positive function $\rho$
defined on the complement of $\pibar_{\NCx}^{-1}(K)$
for a compact subset $K$ of $\NCx$.
We may extend this function to $\Ybar$ arbitrarily
to obtain a function
which satisfies the necessary axioms.
\end{proof}

We fix a function $\rho$
appearing in the proof of \pref{pr:Hba}
throughout the rest of this paper.
We say that
a base-admissible Hamiltonian $\Hba$ has a \emph{slope}
$\lambda \in \bR^{>0}$
if it is a $C^2$-small perturbation of $\lambda \rho$
which coincides with $\lambda \rho$
outside of the inverse image by $\pibar_{\NCx}$
of a compact set in $\NCx$.

\subsection{}

An $\omega$-compatible almost complex structure
$J_{\Ybarloc}$ on $\Ybarloc$ is said to be \emph{base-admissible}
if the map
$
 \pibar_{\NCx} \colon \Ybarloc \to \NCx
$
is $(J_{\Ybarloc},J_{\NCx})$-holomorphic
outside of a compact set.
The set of base-admissible almost complex structures on $\Ybarloc$
will be denoted by $\Jba(\Ybarloc)$.

\subsection{}

A \emph{base-admissible Floer data} is a pair
\begin{align}
 (H, J) \in C^\infty([0,1],\scrHba(\Ybarloc)) \times C^\infty([0,1],\Jba(\Ybarloc))
\end{align}
of families of base-admissible Hamiltonians and
base-admissible almost complex structures.

\subsection{}

A \emph{base-admissible perturbation data}
$(K,J)$
consists of
\begin{enumerate}
 \item
a 1-form
$
 K \in \Omega^1(\Sigma,\scrHba(\Ybar))
$
on $\Sigma$
with values in base-admissible Hamiltonians
satisfying
\begin{enumerate}
 \item
$K|_{\partial \Sigma}=0$, and
 \item
outside of a compact set in the base, we have
$
 \pi_{\NCx,*} (K) =X_{\Hb} \otimes \gamma
$
for $\gamma$ sub-closed, and
\end{enumerate}
 \item
a family
$J \in C^\infty(\Sigma, \Jba(\Ybar))$
of base-admissible almost complex structures on $\Ybar$
parametrized by $\Sigma$.
\end{enumerate}
It is \emph{compatible}
with a sequence
$(\bsH, \bsJ) = (H_j, J_j)_{j=0}^d$
of base-admissible Floer data
if \eqref{eq:Fleor_compatible} holds
for any $j \in \{0, \ldots, d\}$ and any $t \in [0,1]$.


\begin{lemma} \label{lm:max1}
Let $y \colon \Sigma \to \Ybar$
be a solution to Floer's equation \pref{eq:Floer}
with respect to a base-admissible perturbation data.
Set $p \coloneqq S \circ \pibar_{\NCx} \circ y \colon \Sigma \to \bR$.
If $p$ is not a constant function,
then $p$ does not have a maximum on $\Sigma$
whose maximum value is outside of the compact set
appearing in the definitions of $\scrHba(\Ybar)$ and $\Jba(\Ybar)$.
\end{lemma}

\begin{proof}
It follows from the base-admissibility of $K$ that
the map
$
 w \coloneqq \pi_{\NCx} \circ y \colon \Sigma \to \NCx
$
satisfies the Floer's equation
\begin{align} \label{eq:Fbase}
 (dw- X_{\Hb} \otimes \gamma)^{0,1} = 0
\end{align}
on $X$ for the Hamiltonian $\Hb$
outside of a compact set in $\NCx$.
%
%
We write the almost complex structures on $\NCx$ and $\Sigma$
as $J$ and $j$ respectively, and set
$
 \beta \coloneqq - \partial^c S = - dS \circ J
$
and
$
 p \coloneqq S \circ w.
$
Applying $d S$ to both sides of Floer's equation
\begin{align}
 (dw- X_{\Hb} \otimes \gamma) \circ j
  = J \circ (dw - X_{\Hb} \otimes \gamma)
\end{align}
and using $dS(X_{\Hb})=0$,
one obtains
\begin{equation} \label{dcrhogamma}
 d^c p = - \beta \circ (dw - X_{\Hb} \otimes \gamma).
\end{equation}
By applying $d$ to both sides and using
$
 \omega = d \beta = - d d^c S,
$
one obtains
\begin{equation} \label{laprhogamma}
 -d d^c p
  = w^* \omega - d (\beta(X_{\Hb}) \cdot \gamma).
\end{equation}
Since
$
 \beta(X_{\Hb}) = -\Hb
$
outside of a compact set in $N_\bR$,
one has
\begin{align}
 - d d^c p
  &= w^* \omega - d(w^* \Hb \cdot \gamma) \\
  &= w^* \omega - d(w^*\Hb) \wedge \gamma - w^*\Hb \cdot d\gamma \\
  &= \| dw - X_{\Hb} \otimes \gamma \|^2 - w^*\Hb \cdot d\gamma \\
  &\ge 0
\end{align}
since $\Hb \ge 0$ and $d \gamma \le 0$.
Now $- d d^c$ is an operator of the form
\eqref{eq:elliptic_op},
so the function $p$ satisfies the strong maximum principle.
If the function $p$ attains a maximum
at $\Sigma = \Sigmabar \setminus \bszeta$,
then Hopf's lemma implies that
\begin{align} \label{eq:Hopf1}
 d p(\nu) > 0
\end{align}
for an outward normal vector $\nu$ of $\partial \Sigma$
at some point $x \in \partial \Sigma$.
Let $\tau \in T_x (\partial \Sigma)$ be the tangent vector
such that $\nu = j \tau$.
Then one has
\begin{align}
 dp(\nu)
  &= dS \circ (dw \circ j) (\tau) \\
  &= dS \circ \lb X_{\Hb} \otimes \gamma \circ j
   + J \circ (du-X_{\Hb} \otimes \gamma) \rb (\tau) \\
  &= - \beta \circ (dw-X_{\Hb} \otimes \gamma) (\tau),
\end{align}
where we used $dS(X_{\Hb})=0$ and $\beta=-dS \circ J$.
The first term vanishes by the Legendrian condition $\beta|_{\Lbar} = 0$ at infinity,
and the second term vanishes by $\gamma|_{\partial \Sigma}=0$.
This contradicts \eqref{eq:Hopf1},
and \pref{lm:max1} is proved.
\end{proof}

\subsection{}
 \label{sc:vaJ1}

An almost complex structure $J$ on $\Ybar$
is \emph{fibration-admissible}
if
\begin{enumerate}
\item
$J$ is base-admissible,
\item
the map $\pibar_{\bC}: \Ybar \to \bC$ is $J$-holomorphic on
$
 \pibar_{\bC}^{-1} \lb \lc u \in \bC \relmid \abs{u} > C_0 \rc \rb
$
for some $C_0$,
\item
the divisor $\Ebar$
defined in \eqref{eq:Ebar} is $J$-holomorphic, and
\item
the almost complex structure $J|_{U_D}$ is the product of
$J_{\NCx}$ and the standard complex structure
on $\bD_\delta$
under the identification
\eqref{eq:UD2}.
\end{enumerate}
The set of fibration-admissible almost complex structures on $\Ybar$
will be denoted by $\scrJ(\Ybar)$.

\subsection*{}

The following stronger notion will be used later in \pref{sc:HW}:

\subsection{}

A fibration-admissible almost complex structure $J$ on $\Ybar$ is
\emph{integrably fibration-admissible}
if there exists an almost complex structure $J_{\NCx}$ adapted to $Z$
such that when one equips $\NCx \times \bC$
with the almost complex structure $(J_{\NCx}, J_\bC)$,
the structure map
$
 p = (\pibar_{\NCx}, \pibar_{\bC}) \colon \Ybar \to \NCx \times \bC
$
of the blow-up is pseudo-holomorphic on the union of
\begin{enumerate}
 \item
$\pibar_{\bC}^{-1} \lb \lc u \in \bC \relmid \abs{u} > C_0 \rc \rb$ for some $C_0$,
\item
$\pibar_{\NCx}^{-1} \lb \lc \bw \in \NCx \relmid S(\bw) > C_1 \rc \rb$ for some $C_1 > 0$, and
\item
$U_D \cup \pi^{-1}(U_{Z})$.
\end{enumerate}
The set of integrably fibration-admissible almost complex structures on $\Ybar$
will be denoted by $\Jint(\Ybar)$.


\subsection{}
 \label{sc:avH}

Fix $\mu_0, \mu_1 \in \bR^{>0}$ such that
$\mu_0 \ll \epsilon \ll \mu_1$.
In the exact structure \pref{eq:primfiber},
if we set the boundary to be
$\lc \mu = \mu_0 \rc$ and $\lc \mu=\mu_1 \rc$,
then the Liouville coordinates become
$c_-(\epsilon - \mu)$ and $c_+(\mu-\epsilon)$
for some constant $c_{-}$ and $c_{+}$.
%
A function
$
 \Hv \colon \Ybar \to \bR^{>0}
$
is an \emph{admissible vertical Hamiltonian}
of slope $\lambda \in \bR^{>0}$
if there exist a function
$
 f \colon \bR^{\ge 0} \to \bR^{>0}
$
satisfying
\begin{enumerate}
 \item \label{it:avH0}
$f''(x) \ge 0$ for any $x \in \bR^{\ge 0}$,
 \item \label{it:avH1}
$f(x) = \lambda c_- (\epsilon - x)$ when $x < \mu_0$,
 \item \label{it:avH2}
$f'(x)=0$ in a neighborhood of $x=\epsilon$,
\item \label{it:avH3}
$f(x) = \lambda c_+ (x - \epsilon)$ when $x > \mu_1$, and
 \item \label{it:avH4}
$
 \Hv = f \circ \mu
$
where
$
 \mu \colon \Ybar \to \bR^{\ge 0}
$
is the moment map \eqref{eq:moment_map}.
\end{enumerate}
The set of admissible vertical Hamiltonians
will be denoted by $\scrHv \lb \Ybar \rb$.
%
\begin{figure}
\centering
\input{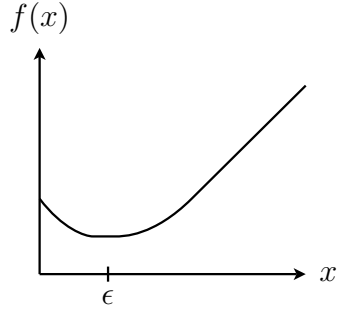}
\caption{An admissible vertical Hamiltonian}
\label{fg:Hver}
\end{figure}
The Hamiltonian vector field associated with $\Hv$ is given by
\begin{align} \label{eq:XHver}
 X_{\Hv} = f'(\mu) \cdot X_{\mu},
\end{align}
where $X_\mu$ is the fundamental vector field
for the $\bS^1$-action \eqref{eq:S^1-action}.
It follows that
\begin{align} \label{eq:Hver1}
 \lb \pi_{\NCx} \rb_*(X_{\Hv})=0.
\end{align}

\subsection{}
 \label{sc:faH}

A \emph{fibration-admissible Hamiltonian}
of slope $\lambda \in \bR^{>0}$
is a function
$
 H \colon \Ybar \to \bR^{> 0}
$
satisfying the following conditions:
\begin{enumerate}
 \item
One has
$
 (\pibar_{\NCx})_*(X_H) = X_{\lambda \rho}
$
outside of a compact set in $\NCx$.
 \item
Whenever $\abs{\mu}<\mu_0$ or $\abs{\mu} > \mu_1$,
one has
$
 H=\Hba+\Hv,
$
where $\Hba$ is a base-admissible Hamiltonian of slope $\lambda$ and
$\Hv$ is an admissible vertical Hamiltonian of slope $\lambda$.
 \item
$X_H$ is tangent to $\Ebar$.
\end{enumerate}
The set of fibration-admissible Hamiltonians of slope $\lambda$
will be denoted by $\scrH_\lambda \lb \Ybar \rb$.

\begin{remark}
To actually construct examples,
we may assume that
$
 H=\Hba+\Hv
$
holds everywhere.
The extra flexibility of our definitions are included for \pref{sc:Comp}.
\end{remark}

\subsection{}

A \emph{fibration-admissible sequence of Hamiltonians}
is a sequence $(H_m)_{m=1}^\infty$
of fibration-admissible Hamiltonians
such that the slopes $\lambda_m$
of $H_m$ for $m \in \bZ^{>0}$
satisfy
$\lambda_m < \lambda_{m+1}$
and $\lambda_m+\lambda_{m'} \le \lambda_{m+m'}$
for any $m, m' \in \bZ^{>0}$.


\subsection{}

A \emph{fibration-admissible perturbation data}
is a base-admissible perturbation data
\begin{align}
 (K, J)
  \in \Omega^1 \lb \Sigma, \Hba \lb \Ybar \rb \rb
   \times C^\infty \lb \Sigma, \scrJ \lb \Ybar \rb \rb
\end{align}
satisfying the following conditions:
\begin{enumerate}
 \item
$X_K$ is tangent to $\Ebar$ and $D$.
 \item
For $\mu \gg \epsilon$,
one has
$
 \lb \pibar_{\bC} \rb_*(X_K)
  = \lb \pibar_{\bC} \rb_* X_{\Hv} \otimes \gamma_{+}
$
for a subclosed one form $\gamma_+$
and a vertical Hamiltonian $\Hv$.
 \item
For points on $\Ybar \setminus \pi_{\NCx}^{-1}(U_{Z})$ with $\mu \ll \epsilon$,
one has
\begin{align} \label{eq:admcondition3}
 \lb \pibar_{\bC} \rb_* (X_K)
  = \lb \pibar_{\bC} \rb_* X_{\Hv} \otimes \gamma_{-}
\end{align}
for a subclosed one form $\gamma_-$
and a vertical Hamiltonian $\Hv$.
\end{enumerate}

\subsection{}

Assume $d \leq 3$ and
let $\cR^d(Y, \bx)$ be the moduli space of solutions
to Floer's equation for perturbation data
in Sections \ref{sc:fD1}, \ref{sc:fD2} and \ref{sc:fD4} (with Liouville admissible data replaced by the corresponding fibration-admissible data).
Similarly,
we let $\cR^2_\tau(Y, \bx)$ be the moduli space of solutions
to Floer's equation for the perturbation data $K_\tau(m_1, m_2, m_3)$
appearing in \pref{sc:fD3},
whose union is denoted by
\begin{align}
 \cR^2_\bullet(Y, \bx) \coloneqq \bigcup_{\tau \in [0,1]} \cR^2_\tau(Y, \bx).
\end{align}

Let $\Lc$ be the closure in $\Ybar$
of a base-admissible Lagrangian section $L$.
By base-admissibility,
solutions to Floer's equation are now constrained
to lie in a compact subspace in $\Ybar$
and so Gromov compactness applies as usual.
Gromov compactness is typically stated for Lagrangians without boundary,
but here it applies because we can extend $\Lc$ slightly in the negative real direction as well.

Since $\Lc$ is contractible,
the relative homotopy group $\pi_2(\Ybar, \Lc)$ is a torsor over $\pi_2(\Ybar)$, and
the possible relative homology classes of Floer curves in $\Ybar$ is a torsor
over the image of $\pi_2(\Ybar)$ in $H_2(\Ybar)$.
From the general properties of blowing up,
we have
\begin{align}
 H_2(\Ybar) \cong H_2(\NCx) \oplus [E_{\bw}],
\end{align}
where $[E_{\bw}]$ is the class
generated by any exceptional sphere
over $\bw \in Z$.
It follows that the image of $\pi_2(\Ybar)$ in $H_2(\Ybar)$ is one-dimensional
and generated by $[E_{\bw}]$.

The moduli spaces $\cR^{d}(Y, \bx)$ and $\cR^2_\tau(Y, \bx)$ embeds naturally
into the Gromov compactifications of maps into $\Ybar$
with some relative homology class $A_{\bx} \in H_2 \lb \Ybar, \Lc \rb$
which is uniquely determined by the intersection with either $E$ or $F$.
The closures of these spaces will be denoted by
$\cRbar^{d}(\Ybar,\bx)$
and $\cRbar^2_\tau(\Ybar,\bx)$ respectively.
Moreover,
if every component of $u \in \cRbar^{d}(\Ybar,\bx)$ (or $\cRbar^2_\tau(\Ybar,\bx)$) avoids $D \cap \Lc$
and is asymptotic to chords in $Y$,
then the image of $u$ is contained in $Y$.
We also set
\begin{align}
 \cRbar^2_\bullet(\Ybar,\bx)
  \coloneqq \bigcup_{\tau \in [0,1]} \cRbar^2_\tau(\Ybar,\bx).
\end{align}

\begin{definition}
Hamiltonian chords for the Lagrangian $\Lc$
which are completely contained in $D$
are called \emph{divisor} chords.
\end{definition}

\subsection{}

We assume that the Hamiltonian flow preserves $U_{Z \times 0}$,
so that all Hamiltonian chords are disjoint from $\pi_{\NCx}^{-1}(U_{Z})$.

\begin{lemma} \label{lm:Gcompactness}
Let $(K, J)$ be a fibration-admissible perturbation data, and
consider a sequence $\lb y_s \rb_{s=1}^\infty$
of maps $y_s \colon \Sigma \to Y$ in $\cR^{d}(\bx)$
converging to
$
 y_\infty \in \cRbar^{d}(\Ybar,\bx).
$
Let $\lc y_{k,\infty} \rc_{k=1}^l$ be the irreducible components
of the maps.
If a component
$y_{k,\infty} \colon \Sigma_k \to \Ybar$ intersects $D$
at a point on the boundary $\partial \Sigma_k$,
then the component $y_{k,\infty}$ lies entirely in $D$.
\end{lemma}

\begin{proof}
Observe that there are four essentially distinct ways
that a limiting component $y_{k, \infty}$ could intersect the divisor $D$,
i.e.,
\begin{itemize}
\item
in the interior of $\Sigma$,
\item
on the boundary $\partial \Sigma$,
\item
$y_{k,\infty}$ lies completely in $D$, or
\item
$y_{k,\infty}$ limits to a divisor chord in $D$
along some strip-like end $\epsilon$.
\end{itemize}
It therefore remains to rule out intersections along a boundary,
which we claim follows from the fact that $d^c (1/|u|) |_{\Lc}=0$,
where $u$ is the base coordinate on $\bC$ discussed above.
Consider a subsequence
$y_s \colon \Sigma \to Y$
with boundary on $L$
which meets the exceptional divisor $E$ at points $z_{k,s}$.
We can define the intersection number
\begin{align}
 y_s \cdot E = \sum_{z_{k,s}} d_{k,s},
\end{align}
where $d_{k,s} \geq 0$ are the local intersection numbers
of $y_s$ with $E$ at $z_{k,s}$ on $\Sigma$.
To define this number efficiently,
observe that Gromov's trick \cite{MR809718}
(see e.g. \cite[Section 8.1]{MR2954391} for an exposition)
allows us to view a solution $y_s$ to Floer's equation
as a pseudo-holomorphic section
$
 \ytilde_s \colon \Sigma \to \Sigma \times \Ybar
$
for a specific
almost complex structure on $\Sigma \times \Ybar$.
When the perturbation data $(K, J)$ are admissible,
both $\Sigma \times E$ and $\Sigma \times D$ are
almost complex submanifolds of codimension 2.
The local intersection number is then the intersection number
of the section with $\Sigma \times E$.
This number is constant in our sequence $\lb y_s \rb$.
Assume for contradiction that
a sequence $\lb y_s \rb$ has a convergent subsequence
which limits to $u_\infty$ that has a component $y_{k,\infty}$
intersecting $D$ along some $\Lc$ at a point $z_\ell$.
Then the intersection points above
limit to intersection points $z_{k,\infty}$
which are in the interior of $y_{\infty}$.

Fix a small ball $\bD_{\epsilon_i}(z_{k,\infty})$
about these points which avoids $z_\ell$.
We choose $s$ large enough so that all of the points $z_{k,s}$ lie
in $\bD_{\epsilon_i}(z_{k,\infty})$.
Then for large $s$,
there must be a local maximum of $1/|u|$ near $z_\ell$.
This is impossible by the same calculation as in \pref{lm:max1}
if we note that $d^c(1/|u|)|_{\Lc}=0$.
\end{proof}

\subsection{}

There are two useful ways of grading Hamiltonian chords.
The first is used to grade Hamiltonian chords in $Y$.
Namely, one grades $Y$ by the holomorphic volume form
\begin{align}
 \Res \ld \frac{1}{h(\bw) - u v}
    \frac{d w_1}{w_1}
     \wedge \frac{d w_2}{w_2}
     \wedge d u \wedge d v \rd
  &= \frac{d w_1}{w_1}
     \wedge \frac{d w_2}{w_2}
     \wedge \frac{d u}{u}
\end{align}
on $Y$ and
graded Lagrangian submanifolds
with respect to this grading of $Y$.
Note that this grading of $Y$ restricts
to the standard grading on $\NCx \times \bCx$.
We denote this standard grading by $\abs{x_i}$.
If $d \ge 2$, then we have
\begin{align} \label{eq:dim_formula1}
 \vdim \cR^{d}(Y, \bx) = \abs{x_0}-\sum_{i \neq 0} \abs{x_i} + d - 2.
\end{align}
If $d=1$ and the perturbation data is invariant under the $\bR$-translation
(this is the case when we define the Floer differential),
then \pref{eq:dim_formula1} continues to hold,
after modding out by automorphisms.
If $d=1$ and the perturbation data is not translation invariant
(this is the case when we define the continuation map),
then we have
\begin{align}
 \vdim \cR^{1}(Y,\bx) = \abs{x_0}-\abs{x_1}.
\end{align}
In the case of perturbation data in \pref{sc:fD4},
we have
\begin{align}
 \dim \cR^{2}_\bullet(Y,\bx) = \abs{x_0}-\sum_{i \neq 0} \abs{x_i} + 1.
\end{align}
For generic $J \in \Jint(\Ybar)$,
all moduli spaces can be cut out transversely
so that the virtual dimension agrees with the actual dimension. 

\subsection{}

The second way of grading allows us to grade all chords,
including divisor chords.
This is by grading chords $x_i$
with respect to the standard volume form on $\NCx \times \bC$.
Equivalently, one may view this as choosing an algebraic volume form
\begin{align}
 \frac{d w_1}{w_1}
  \wedge \frac{d w_2}{w_2}
  \wedge d u
\end{align}
with a simple zero along $E$,
restricting this to $\Ybar  \setminus E$,
and then grading the Lagrangians.
It follows from the index computation in \cite[Lemma 3.22]{MR3294958}
that for $d \ge 2$, one has 
\begin{align} \label{eq:relindex}
 \dim(\cRbar^{d}(\Ybar,\bx))
  = \reldeg{x_0}-\sum_{i \neq 0} \reldeg{x_i} -2 (A_{\bx} \cdot E) + d - 2,
\end{align}
and similarly for $\cRbar^{1}(\Ybar,\bx)$ and $\cRbar^2_\bullet(\Ybar,\bx)$.

\begin{remark}
Floer theoretic operations will respect this second grading
after one inserts a formal parameter $t$ of cohomological degree $-2$
and curves $y$ are weighted  by $t^{y\cdot E}$,
where $y\cdot E$ is their intersection with the divisor $E$
(see \cite[Lemma 3.22 and Definition 5.1]{MR3294958}).
We will be primarily interested in the first grading,
but use this second grading
to rule out certain breaking configurations and
in certain arguments in \pref{sc:HW}.
\end{remark}

\subsection{}

Now we restrict our attention to the Lagrangian $L_0$,
which is the example treated in this paper.
Throughout the rest of this paper,
we will equip $L_0$ with the trivial spin structure
to view it as a Lagrangian brane.
We may assume that
all Hamiltonian chords are non-degenerate.
For example, we may take a fixed $\Hba$ and $\Hv$ of slope one
and then take $\Hbam$ and $\Hvm$
to be sufficiently small perturbations of $m\Hba$ and $m\Hv$
supported away from $Z$.

\begin{lemma} \label{lm:pi_1}
The relative homotopy group $\pi_1(Y, L_0)$
is naturally isomorphic to $N$.
\end{lemma}

\begin{proof}
The relative homotopy group $\pi_1(Y, L_0)$ is isomorphic
to the fundamental group $\pi_1(Y)$ since $L_0$ is contractible.
It is well-known that the blow-up $\Ybar \to \NCx \times \bC$ induces
an isomorphism
$
 \pi_1 \lb \Ybar \rb \simto \pi_1(\NCx \times \bC) \cong \pi_1(\NCx)
$
of the fundamental group.
The kernel of the map
$
 \pi_1(Y) \to \pi_1 \lb \Ybar \rb
$
is the normal subgroup
generated by the class of a loop
of the form $|w|=\mathrm{pt}$, $|v_0|=\epsilon$
(cf.~e.g.~\cite[Lemma 2.3(a)]{9801075}).
Such loops are contractible in $Y$ due to the singular conic bundle structure.
We therefore conclude that this map is an isomorphism.
The fundamental group of $\NCx$ is naturally isomorphic to $N$,
and \pref{lm:pi_1} is proved.
\end{proof}

Recall that we have assumed that near $D$,
our Hamiltonians have the form
\begin{align}\label{eq:sumform}
H_m=\Hbam+\Hvm
\end{align}
for $\Hbam$ and $\Hvm$ sufficiently generic so that all chords are non-degenerate.
It follows that for every Hamiltonian chord $p$ in $\scrX \lb \Lbar_0,\NCx;\Hbm \rb$,
there is a divisor chord in $\scrX \lb \Lc_0,\Ybar;H_m \rb$,
which we write as $\pd$.

\begin{lemma}
One has
$
 \reldeg{\pd}=2\floor{c_{-}\lambda_{m}}+1 +\abs{p}_{\NCx}.
$
\end{lemma}

\begin{proof}
This is obtained by observing that
we have a product splitting for the Lagrangians,
Hamiltonian flow, and the trivializations
under the product decomposition of $\NCx \times \bC$.
So it suffices to compute the contribution from the $\bC$ factor.
The short chord contributes $1$ to the Maslov index and
each rotation around the cylindrical end contributes
$2\floor{c_{-}\lambda_{m}}$
(see \cite[Section (11e)]{MR2441780}).
\end{proof}

\subsection{}
For
\begin{align}
 \bn
  \in N
  \cong \pi_1(Y,L_0)
  \cong \pi_1(\Ybar, \Lc_0)
  \cong \pi_1(\NCx, \Lbar_0),
\end{align}
we set
\begin{align}
 \scrX(\Lbar_0;\Hbm)_\bn
  \coloneqq \lc p \in \scrX(\Lbar_0;\Hbm) \relmid [p] = \bn \rc.
\end{align}
Let
\begin{align}
 \ell^{\sharp} \colon N \times \bZ^{>0} \to \bZ
\end{align}
be a function
satisfying the following conditions
for any $\bn \in N$;
\begin{align}
 \alitem{
$
 \ell^{\sharp}(\bn, m)
  < \inf_{p \in \scrX(\Lbar_0;\Hbm)_\bn} \reldeg{\pd}-1
$
for any $m \in \bZ^{>0}$,
}\\
 \alitem{
$\ell^\sharp(\bn, m) < \ell^\sharp(\bn, m+1)$
for any $m \in \bZ^{>0}$,
}\\
 \alitem{
$
 \lim_{m \to \infty} \ell^{\sharp}(\bn,m) = \infty.
$
}
\end{align}
We set
\begin{align} \label{eq:Floercomplex}
 \CF^*(L_0;H_m)
  \coloneqq \bigoplus_{x \in \scrX(L_0;H_m)} |\frako_x|,
\end{align}
where $|\frako_x|$ is a one-dimensional $\bC$-normalized orientation space
associated to $x$.
We also set
\begin{align}
 \scrX(L_0;H_m)^{\le \ell^\sharp} \coloneqq \lc x \in \scrX(L_0;H_m) \relmid
  \reldeg{x} \le \ell^{\sharp}([x],m) \rc
\end{align}
and
\begin{align} \label{eq:adFloercomplex}
 \CF^*(L_0;H_m)^{\leq \ell^{\sharp}}
  \coloneqq \bigoplus_{x \in \scrX(L_0;H_m)^{\le \ell^\sharp}} |\frako_x|.
\end{align}
The vector space
$\CF^*(L_0;H_m)^{> \ell^{\sharp}}$
is defined similarly.

\begin{definition}

A pair
$
 \lb \lb H_m \rb_{m=1}^\infty, \ell^{\sharp} \rb
$
is said to satisfy \emph{Assumption A}
if the following conditions are satisfied:
\begin{enumerate}[1.]
 \item
For every pair
$
 \bx = (x_0, x_1)
$
of chords
$
 x_0 \in \scrX(L_0;H_{m'})
$
and
$
 x_1 \in \scrX(L_0;H_m)
$
which lie in the same relative homotopy class,
there is a topological strip
$y_\bx$ in $\Ybar$ between $x_1$ and $x_0$ bounded by $\Lc_0$ such that
$y_\bx$ has intersection number zero
with both components $E$ and $F$ of $\Ebar$.
 \item
If a pair
$
 \bx = (x_0, x_1)
$
of chords
$
 x_0 \in \scrX(\Lc_0;H_{m})^{> \ell^{\sharp}}
$
and
$
 x_1 \in \scrX(\Lc_0;H_{m})^{\leq \ell^{\sharp}}
$
satisfies
$
 \vdim \cRbar^1 \lb \Ybar, \bx \rb = 0,
$
then
$
 \cRbar^1 \lb \Ybar, \bx \rb
$
is empty for generic
$
 J \in C^\infty([0,1],\scrJ(\Ybar)).
$
\end{enumerate}
\end{definition}

The following compactness lemma allows us to give
$\CF^*(L_0;H_m)^{\leq \ell^{\sharp}}$
the structure of a complex and
to define continuation maps which preserve the relative grading.

\begin{lemma} \label{lm:rel_cpt1}
If a pair
$
 \bx = (x_0, x_1)
$
of chords
$
 x_0 \in \scrX(L_0;H_{m'})^{\leq \ell^{\sharp}}
$
and
$
 x_1 \in \scrX(L_0;H_m)^{\leq \ell^{\sharp}}
$
for
$
 m \le m'
$
satisfies
$
 \reldeg{x_0} < \ell^{\sharp}(\bn,m),
$
then for a generic almost complex structure $J \in C^\infty \lb [0,1], \scrJ(\Ybar) \rb$,
the space
$
 \cRbar^1 \lb \Ybar, \bx \rb
$
consists of curves which are disjoint from $D$ and
with no strip-breaking along chords in $\scrX(L_0)^{> \ell^{\sharp}}$.
\end{lemma}

\begin{proof}
Since
\begin{itemize}
 \item
the intersection numbers of any Floer strip with $E$ and $F$ must be non-negative,
 \item
the relative homology class of any topological strip can be obtained
as the connected sum of $y_\bx$ and an integer multiple of $[E_\bw]$, and
 \item
the intersection numbers of $[E_\bw]$ with $E$ and $F$ are given by
\begin{align}
 [E_\bw] \cdot E = - [E_\bw] \cdot F = -1,
\end{align}
\end{itemize}
any Floer strip must have intersection number zero
(and hence, in fact, be disjoint) with $E$ (and $F$).
It follows that
any Floer strip must raise or at least preserve
the relative grading,
so that there can be no strip breaking along such chords with higher relative grading.
\end{proof}

\begin{definition}
A pair $\lb \lb H_m \rb_{m=1}^\infty, \ell^{\sharp} \rb$
is said to satisfy \emph{Assumption B}
if there exists a function
$
 s \colon N^2 \to \bZ^{\ge 0}
$
satisfying the following conditions:
\begin{enumerate}[1.]
 \item
For any $(\bn_1,\bn_2) \in N^2$,
any $(m_1, m_2) \in (\bZ^{>0})^2$
and any triple
$
 \bx = (x_0, x_1, x_2)
$
of chords
$
 x_0 \in \scrX(L_0;H_{m_1+m_2})_{\bn_1+\bn_2}^{\leq \ell^{\sharp}},
$
$
 x_1 \in \scrX(L_0;H_{m_1})_{\bn_{1}}^{\leq \ell^{\sharp}},
$
and
$
  x_2 \in \scrX(L_0;H_{m_2})_{\bn_{2}}^{\leq \ell^{\sharp}}
$
satisfying
\begin{align}
 \vdim(\cR(x_0,x_1,x_2)) \in \lc 0,1 \rc
\end{align}
and
\begin{align} \label{eq:divisorchords1}
\reldeg{x_0} \geq \reldeg{x_1} + \reldeg{x_2} + s(\bn_1,\bn_2),
\end{align}
the space $\cRbar \lb \Ybar, \bx \rb$ is empty
for a generic almost complex structures.
\item
For any chords
$
 x_1 \in \scrX(L_0;H_{m_1})_{\bn_{1}}^{\leq \ell^{\sharp}}
$
and
$
 x_2 \in \scrX(L_0;H_{m_2})_{\bn_{2}}^{\leq \ell^{\sharp}},
$
one has
\begin{align} \label{eq:triangleineq}
 \ell^{\sharp}(\bn_1+\bn_2, m_1+m_2)
  \ge \reldeg{x_1} + \reldeg{x_2} + s(\bn_1,\bn_2).
\end{align}
 \item
Let $x_1$, $x_2$ be as above and $x_0 \in \scrX(L_0;H_{m+n+r})^{\leq \ell^{\sharp}}$.
We assume that whenever
$
 \vdim(\cR^2_\bullet(x_0,x_1,x_2)) \in \lc 0,1 \rc
$
and \pref{eq:divisorchords1} holds,
the moduli spaces $\cR^2_\bullet(x_0,x_1,x_2)$ are empty
for generic complex structures.
We also require that the same holds
when the continuation maps are along the other positive strip like end. 
\end{enumerate}
\end{definition}

\begin{lemma} \label{lm:rel_cpt2}
\begin{enumerate}[1.]
 \item
If a triple of chords
$
 x_0 \in \scrX(L_0;H_{m+n})^{\leq \ell^{\sharp}},
$
$
 x_1 \in \scrX(L_0;H_{m})^{\leq \ell^{\sharp}},
$
and
$
 x_2 \in \scrX(L_0;H_n)^{\leq \ell^{\sharp}}
$
satisfies
\begin{align} 
& \reldeg{x_1} + \inf_{p \in \scrX(\Lbar_0;H_{\operatorname{ba},n})_{\bn_2}} \reldeg{\pd} > \reldeg{x_0}, \\
& \reldeg{x_2} + \inf_{p \in \scrX(\Lbar_0;\Hbam)_{\bn_1}} \reldeg{\pd} > \reldeg{x_0}, \quad \text{and} \\
& \vdim(\cR^2(x_0,x_1,x_2))=\lc 0,1 \rc,
\end{align}
then for generic surface-dependent almost complex structures,
the closure $\cRbar^2(x_0,x_1,x_2)$ consists of curves
which are disjoint from $D$ and
with no breaking along chords in $\scrX(L_0,L_0)^{> \ell^{\sharp}}$.
\item
If a triple of chords
$
 x_0 \in \scrX(L_0;H_{m_2+m_3})^{\leq \ell^{\sharp}},
$
$
 x_1 \in \scrX(L_0;H_{m_1})^{\leq \ell^{\sharp}}
$
and
$
 x_2 \in \scrX(L_0;H_{m_3})^{\leq \ell^{\sharp}}
$
for $m_1 < m_2$
satisfies
\begin{align}
& \reldeg{x_1}+ \inf_{p \in \scrX(\Lbar_0;H_{\operatorname{ba},m_2})_{\bn_2}} \reldeg{\pd} > \reldeg{x_0}, \\
& \reldeg{x_2} + \inf_{p \in \scrX(\Lbar_0;H_{\operatorname{ba},m_1})_{\bn_1}} \reldeg{\pd} > \reldeg{x_0}, \\
& \reldeg{x_0} < \ell^{\sharp}(\bn,m_1+m_3), \text{ and} \\
& \vdim(\cR^2_t(x_0,x_1,x_2)) \in \lc 0, 1 \rc,
\end{align}
then for generic surface-dependent almost complex structures,
the closure $\cRbar^{2}_\bullet(x_0,x_1,x_2)$
consists of curves
which are disjoint from $D$ and
which do not break along chords in $\scrX(L_0,L_0)^{> \ell^{\sharp}}$.
\end{enumerate}
\end{lemma}

\begin{proof}
A broken curve in this moduli space will have one component which is a Floer triangle $\Sigma$ and possibly many components which are spheres and strips. The most difficult case to rule out is breaking along a divisor chord. The curve cannot break along a divisor chord at a position in our tree after or as outputs of the curve $\Sigma$ for the same reason as in the preceding compactness lemma.  Now we consider the case of breaking along divisor chords before or as inputs of $\Sigma$. The two bulleted conditions rule that out since the Floer triangle would necessarily have negative virtual dimension and hence not exist for generic almost complex structures. 
\end{proof}

We define the product
$
 \frakm_2
$
of two (orientation lines corresponding to) chords
$x_{1} \in \scrX(L_0;H_{m})_{\bn_{1}}^{\leq \ell^{\sharp}}$
and
$x_2 \in \scrX(L_0;H_{n})_{\bn_{2}}^{\leq \ell^{\sharp}}$ as follows.
Let $m_{1}$ and $m_{2}$ be the smallest integers such that
\begin{align}
 \ell^{\sharp}(\bn_1,m_1) > \reldeg{x_1}+ s(\bn_1,\bn_2)
\end{align}
and
\begin{align}
 \ell^{\sharp}(\bn_2,m_2) > \reldeg{x_2}+ s(\bn_1,\bn_2).
\end{align}
If $m<m_{1}$ or $n<m_{2}$, we may define the Floer product by first applying the continuation maps into
$\CF^*(L_0;H_{m_{1}})^{\leq \ell^{\sharp}}$
and
$\CF^*(L_0;H_{m_{2}})^{\leq \ell^{\sharp}}$
and then applying the usual Floer product.
This is compatible with continuation maps in both variables. 

\begin{definition}
A pair
$
 \lb \lb H_m \rb_{m=1}^\infty, \ell^{\sharp} \rb
$
is said to satisfy \emph{Assumption C}
if there is a function
\begin{align}
 s \colon \lb \pi_1(Y, L_0) \rb^3 \to \bZ^{\ge 0}
\end{align}
such that for any
$
 \lb \bn_1, \bn_2, \bn_3 \rb \in \lb \pi_1(Y, L_0) \rb^3,
$
any
$
 (m_1, m_2, m_3) \in (\bZ^{>0})^3,
$
any
$
 m_4 \in \bZ^{\ge 0},
$
any
$
 \bx = (x_0, x_1, x_2, x_3)
$
consisting of
$
 x_i \in \scrX(\Lbar_0;H_{m_{i}})_{\bn_{i}}^ {\leq \ell^{\sharp}}
$
for $i=1, 2, 3$
and
$
 x_0 \in \scrX(L_0;H_{m_1+m_2+m_3+m_4})^{\leq \ell^{\sharp}}
$
satisfying 
\begin{align}
 \vdim(\cRbar^3(\Ybar, \bx)) \in \lc 0,1 \rc 
\end{align}
and
\begin{align} \label{eq:m3condition1}
 \reldeg{x_0} \ge s(\bn_i) + \sum_{i=1}^3 \reldeg{x_i},
\end{align}
the space $\cRbar^3(\Ybar, (x_0,x_1,x_2,x_3))$ is empty.
\end{definition}

A compactness result similar to \pref{lm:rel_cpt1} and \pref{lm:rel_cpt2} shows that
under Assumption C,
the restriction of the $A_\infty$-operation $\frakm_2$
to $\CF^*(L_0;H_m)^{\le \ell^\sharp}$
gives rise to an associative product
on the adapted wrapped Floer cohomology
defined below.

\begin{definition}
A pair
$
 \lb \lb H_m \rb_{m=1}^\infty, \ell^{\sharp} \rb
$
is \emph{admissible}
if it satisfies Assumptions A, B and C. 
Given an admissible pair,
the \emph{adapted wrapped Floer cohomology}
is the colimit
\begin{align}
 \HWad(L_0) \coloneqq \varinjlim_m  \HFad(L_0;H_m)
\end{align}
of the cohomology
$\HFad(L_0;H_m)$
of the complex
\pref{eq:adFloercomplex}.
\end{definition}

\subsection{}

Let us assume that our data is chosen to satisfy \pref{eq:sumform} everywhere.
For every $p$ in $\scrX(\Lbar_0;\Hbam)$,
in addition to $\pd$,
there are chords $p_i$
for every $i \in \bZ$
satisfying $\abs{i}<c_{\pm} \lambda_m$.

\begin{lemma}
We have
$
 \reldeg{p_{i}}
  = 2 i + \abs{p}_{\NCx}
$
and
$
 \abs{p_{i}}
  = \abs{\pd}
  = \abs{p}_{\NCx}.
$
\end{lemma}

\begin{remark} \label{rm:one_generator}
In fact,
we can arrange that
there is at most one generator
$
 p_{\bn}^m \in \scrX(\Lbar_0, H_m)
$
of grading 0
in each homotopy class
$\bn \in \pi_1(Y,L_0)$.
See the explicit constructions of \pref{sc:HW} for further details on this point.
\end{remark}

\begin{lemma} \label{lm:existence}
There exist an admissible pair
$
 \lb \lb H_m \rb_{m=1}^\infty, \ell^{\sharp} \rb.
$
\end{lemma}

\begin{proof}
Let us assume that our data is chosen to satisfy \pref{eq:sumform} everywhere and
to satisfy the condition stated in \pref{rm:one_generator}.
Then,
there are no differentials since everything lies in degree zero,
so we only have to consider continuation solutions to define the wrapped Floer groups,
Floer triangles to define the multiplication,
and one-dimensional moduli spaces of Floer quadangles
to ensure that the multiplication is associative.
A continuation solution $y$ with input $p_{\bn_1,i_1}^{m_1}$ does not intersect $\Ebar$ and
in particular does not intersect $E$.
Since the output must be of the form $p_{\bn_1,i_2}^{m_2}$,
we require that
\begin{align}
 \ell^{\sharp}(\bn_1,m_1) < \ell^{\sharp}(\bn_1,m_2).
\end{align}
Given a Floer triangle $y$ with inputs
$p_{\bn_1,i_1}^{m_1}$ and
$p_{\bn_2,i_2}^{m_2}$,
the intersection number $y \cdot \Ebar$ is fixed,
which is greater than or eqaul to $y \cdot E$,
and the output is necessarily of the form
$p_{\bn_1+\bn_2,i_3}^{m_1+m_2+r}$.
In  view of \pref{eq:relindex},
the degree of 
$p_{\bn_1+\bn_2,i_3}^{m_1+m_2+r}$ is at most
\begin{align}
 \sum_{i=1}^2 \reldeg{x_i}+ 2 y \cdot \Ebar +1
\end{align}
and the condition \pref{eq:divisorchords1} is satisfied.
The condition \pref{eq:triangleineq} is also satisfied
if 
\begin{align}\label{eq:triangles}
 \ell^{\sharp}(\bn_1,m_1)+\ell^{\sharp}(\bn_2,m_2) + 2 y\cdot \Ebar +1 \leq \ell^{\sharp}(\bn_1+\bn_2,m_1+m_2)
\end{align}
for every $\bn_1$, $\bn_2$.
Conversely,
these two conditions guarantee the admissibility of $\ell^{\sharp}$.
To see this,
we consider the curves in $\cR^3$ and the equations \pref{eq:m3condition1}.
\pref{eq:relindex} shows that for every $\bn_1$, $\bn_2$, $\bn_3$,
the grading for $x_0$ is at most
\begin{align} \label{eq:R3}
\sum_{i=1}^{3} \reldeg{x_i} + 2 y\cdot \Ebar 
\end{align}
where $y \cdot \Ebar$ is again determined by $\bn_i$.
Setting $s(\bn_i)=2y \cdot \Ebar +1$ shows that \pref{eq:m3condition1} is satisfied.
\end{proof}

\begin{remark}
We will see that we can take the function
$
 \ell^\sharp(\bn) = 2m - 2 \ell_1(\bn),
$
where $\ell_1$ is the function defined in \pref{sc:HW}
provided $c_{-} \lambda_m > m$.
\end{remark}

\begin{definition}
Given two positive integers $m$ and $m'$,
two functions
$
 k^\sharp, \ell^\sharp \colon \pi_1(Y, L_0) \times \bZ^{> 0} \to \bZ
$
is said to satisfy
$
 k^\sharp(\bullet, m) < \ell^\sharp(\bullet, m')
$
if
$
 k^\sharp(\bn, m) < \ell^\sharp(\bn, m')
$
for any $\bn \in \pi_1(Y, L_0)$.
\end{definition}
The following definition will be used in the proof of \pref{thm: isofloer}.
\begin{definition}
Two admissible pairs
$
 \lb \lb G_m \rb_{m=1}^\infty, \ell_G^{\sharp} \rb
$
and
$
 \lb \lb H_m \rb_{m=1}^\infty, \ell_H^{\sharp} \rb
$
are \emph{equivalent} if
\begin{enumerate}
\item
for every $m$ there are admissible homotopies: $H_{s,m}$
between $H_m$ and $G_{m'}$ for some $m'>m$
with $\ell^{\sharp}_{m',G}>\ell^{\sharp}_{m,H}$, and
$G_{s,m}$ between $G_m$ and $H_{m''}$ for $m''>m$ with
$\ell^{\sharp}_{m'',H}>\ell^{\sharp}_{m,H}$, and
\item
for any $m$ and $m'$, and $x \in \scrX(L_0;H_m)$ and
$y \in \scrX(L_0;G_{m'})$
which lie in the same relative homotopy class,
there is a topological strip $\Sigma$ in $\Ybar$
between $x$ and $y$ satisfies $\Sigma \cdot E=0$ and $\Sigma \cdot F=0$.
\end{enumerate}
\end{definition}

\begin{lemma} \label{lm:invarianceadp}
Given an equivalence between two fibered admissible families,
the induced maps
$\HWad^*(L_0;H_m) \to \HWad^*(L_0;G_m)$
is an equivalence.
\end{lemma}

\begin{proof}
The continuation maps induced by $G_{s,m}$ and $H_{s,m}$ preserve
the relative gradings by \pref{eq:relindex} and
hence the equivalence follows by a standard Floer theoretic argument.
\end{proof}

\begin{remark}
Fix finitely many base-admissible Lagrangian sections $L_i$.
By restricting our attention to suitable Hamiltonians
of the form
$
 H=\Hba+\Hv,
$
it is possible to formulate a generalization of the functions $\ell^{\sharp}$
for every pair $L_i$ and $L_j$,
to construct an $A_{\infty}$ category with objects $L_i$.
As we have emphasized,
we will discuss only the Lagrangian $L_0$ in this paper,
and hence do not pursue this generalization here.
\end{remark}

\section{Wrapped Floer cohomology ring of $L_0$}
 \label{sc:HW}

\subsection{Approximations of a quadratic Hamiltonian}

In this section,
we compute the adapted wrapped Floer cohomology ring
of the admissible Lagrangian submanifold $L_0$.
The first step in our computation is
to construct a particular sequence
$\lb H_m \rb_{m=1}^\infty$
of fibration-admissible Hamiltonians.
Our calculation will be modeled
on the following simplified situation:
The positive real part
$
 N_{\bR^{>0}} \times \bR^{>0}
$
is an open Lagrangian submanifold of $\NCx \times \bC$,
which can be separated from the center $Z \times 0$
of the blow-up by an open set.
By choosing a function supported in this open set
as the cut-off function $\chi$
appearing in the symplectic form \eqref{eq:symp_form}
on the blow-up,
this Lagrangian submanifold naturally lifts diffeomorphically
to the Lagrangian submanifold $L_0$ of $Y$.
Note that $\NCx \times \bCx$
has a natural quadratic Hamiltonian
\begin{align} \label{eq:Hquad}
 \Hquad(\bw, u) \coloneqq \frac{1}{2} |\br|^2 + \frac{1}{2} (\log |u|)^2.
\end{align}
The image of $ N_{\bR^{>0}} \times \bR^{>0}$
under the time 1 Hamiltonian flow
with respect to this quadratic Hamiltonian
is given by $\theta_1 = r_1$,
$\theta_2 = r_2$, and $\arg u = \log |u|$,
so that the set of Hamiltonian chords
is naturally in bijection with $N \times \bZ$.
This bijection sends a chord
to its class in the relative homotopy group
$\pi_1(\NCx \times \bCx,  N_{\bR^{>0}} \times \bR^{>0})$,
which can naturally be identified with $N \times \bZ$.
It is well known in this case that one can identify
the wrapped Floer cohomology ring of $L_0$ in $\NCx \times \bCx$
with the group ring $\bC[N \times \bZ]$.

We now regard $L_0$ as a Lagrangian submanifold of $Y$ and
construct a sequence
$\lb H_m \rb_{m=1}^\infty$
of Hamiltonians,
such that $H_m$ behaves like $\Hquad$
in a compact set and
has a slope $\lambda_m$ elsewhere.
To be more precise,
for each $m \in \bZ^{>0}$,
we consider a function of the form
$
 H_m = \Hbam + \Hvm.
$
We require that $\Hvm=f_m \circ \mu$ and that there exist sequences
$\{ b_m \}_{m=1}^\infty$, $\{ c_m \}_{m=1}^\infty$ of positive real numbers
satisfying the following conditions:
\begin{itemize}
 \item
$ \lim_{m \to \infty} b_m = \infty$
and $\lim_{m \to \infty} c_m = 0$.
\item
$f_m \colon \bR^{\ge 0} \to \bR^{> 0}$ is a convex function.
 \item
$f_m(x) = \lambda_m x$ for $x > b_m$.
 \item
$f_m(x) = \lambda_m(\epsilon-x)$ for $x < c_m$.
 \item
If $m < m'$,
then any Hamiltonian chord of $H_{\mathrm{v},m}$
is also a Hamiltonian chord of $H_{\mathrm{v},m'}$.
\end{itemize}
To construct the function $\Hbam$, we may embed the neighborhoods $U_{\Pi_c}$ and $U_{\Pi_i}$ in neighborhoods $W_{\Pi_i}$ and $W_{\Pi_c}$, where
\begin{align}
 W_{\Pi_c} \coloneqq \lc \br \in N_\bR \relmid
  \abs{\br} < R_1 \rc
\end{align}
for $R_0 < R_1 \ll R$ and
\begin{align}
 W_{\Pi_i}
  &= \lc \br \in N_\bR \relmid
   \abs{\br_{\bsalpha_i-\bsbeta_i}}< 2\varepsilon_{\mathrm{n}}
    \text{ and }
   \br_{(\bsalpha_i-\bsbeta_i)^\bot} \ge a_i + \varepsilon_{\mathrm{n}}
   \rc
\end{align}
Let
$
 \gdeg \colon Y \to \bR
$
be a smooth convex function satisfying
\begin{align} \label{eq:gdeg}
 \gdeg =
\begin{cases}
 0 & \br \in U_c, \\[2mm]
\dfrac{c_i^2}{2 R^2} \rho^2 + \psi_i(r_{\alpha_i-\beta_i})
  & \br \in W_i \setminus (W_i \cap W_c), \ i=1, \ldots, \ell, \\[4mm]
 \dfrac{S(\br)}{R^2} & \br \in N_\bR \setminus \lb W_c \cup \bigcup_{i=1}^\ell W_i \rb
\end{cases}
\end{align}
where
$
 \psi_i \colon \bR^{\ge 0} \to \bR^{\ge 0}
$
is a smooth convex function satisfying
\begin{itemize}
 \item
$\psi_i^{-1}(0) = [0, \varepsilon_{\mathrm{n}}]$ and
 \item
$\psi_i(r) = \frac{c_i r^2}{2 R^2}$ if $r \ge 2 \varepsilon_{\mathrm{n}}$
\end{itemize}
We let $\Hbar_{\operatorname{ba},m}$ be a function satisfying
\begin{align}
 \Hbar_{\operatorname{ba},m} =
\begin{cases}
 \gdeg & \rho < \lambda_m - \varepsilon_m,\\
 \lambda_m \rho + \text{constant} & \rho > \lambda_m + \varepsilon_m
\end{cases}
\end{align}
for some $0 <\varepsilon_m \ll \lambda_m$ and
$\Hbar_{\operatorname{ba},m}=h_i(\rho)$
in each tubular neighborhood $U_i$
for some function $h_i \colon \bR \to \bR$.
Finally, we let $\Hbam$ be a small perturbation of this function, where the perturbation is supported outside of $\pibar^{-1}(U_Z)$
(in $W_i$, it suffices to take a small perturbation of the function $\psi$).
We assume that
the Hamiltonian flow preserves $U_{Z \times 0}$,
so that the Hamiltonian chords for $L_0$ live
entirely in the region where the fibration
$\pibar_{\NCx} \colon \Ybar \to \NCx$ is trivial.
By choosing suitable slopes $\lambda_m$ and
perturbations in  the definition of  $\Hbam$,
we can assume
\begin{align}
 \alitem{
all Hamiltonian chords
are in the region where $\Hbam$ is quadratic,
}\\
 \alitem{
for any $m \in \bZ^{>0}$ and any $\bn \in N$ satisfying $|\bn| \le \lambda_m$,
there are exactly $2 \lambda_m+1$
Hamiltonian chords $\lc p^m_{\bn,i} \rc_{i=-m}^m$of $H_m$
in the homotopy class
$
\bn
$, and}\\
 \alitem{
all $p^m_{\bn,i}$ are non-degenerate.
}
\end{align}

We assume $\lambda_m >m$ and define $\ell^{\sharp}(\bn)=2k-2\ell_1(\bn)$. These functions are admissible and we may consider the adapted Floer groups $\HFad(L_0; H_m)$. For each Hamiltonian chord $p_{\bn,i}$,
there is an associated orientation space $\frako_{p_{\bn,i}}$.
However, as noted by Pascaleff \cite[Section 4.6]{MR3265556},
in our situation,
the kernel and cokernel of the local orientation operator $D_{p_{\bn,i}}$ are trivial.
It follows that $|\frako_{p_{\bn,i}}|$ is canonically trivialized and
has a canonical generator.
We denote this generator by $p_{\bn,i}$. In constructing homotopies between $H_m$ and $H_{m'}$, we can and will assume that our family of Hamiltonians satisfies $\partial_sH_{s,t} \leq 0$ everywhere.

\begin{proposition} \label{pr:cont}
If $m < m'$,
then the continuation map
\begin{align}
 \HFad(L_0; H_m) \to \HFad(L_0; H_{m'})
\end{align}
sends $p_{\bn,i}^m$ to $p_{\bn,i}^{m'}$.
\end{proposition}

\begin{proof}
Note that the continuation map preserves homotopy classes.
Since there is a unique Hamiltonian chord on $\NCx$
in each homotopy class,
the limits of the curve are the same
once projected to $\NCx$.
The projection of such a solution
$\pi_{\NCx}(y)$
is therefore a topological cylinder such that
$\pi_{\NCx}(y) \cdot Z=0$.
It follows that such solutions $y$ have intersection number zero
with $\Ebar$ and in particular $E$.
This means that these solutions lie in
$
 (\NCx \setminus Z) \times \bCx.
$

However, in this space, the orbits of our Hamiltonian all lie
in different homotopy classes
(or alternatively we can reach the same conclusion
by comparison of $\reldeg{-}$).
Our family of Hamiltonians satisfies $\partial_sH_{s,t} \leq 0$ everywhere.
Hnce we have $\Egeom(y) \le \Etop(y)$.
The quantity on the right-hand side is zero
since the solution has the same asymptotes at the positive and negative ends.
We therefore conclude that the solution is constant.
\end{proof}

\begin{corollary}
One has
\begin{align}
 \HWad(L_0)
  \coloneqq  \varinjlim_{m} \HFad(L_0; H_m)
  = \vspan \{ p_{\bn,i} \}_{(\bn,i) \in N \times \bZ},
\end{align}
where $p_{\bn,i} = \lim_{m \to \infty} p^m_{\bn,i}$.
\end{corollary}

\subsection{}

We may also consider the family
$
 \lb 2 H_m \rb_{m=1}^\infty
$
of Hamiltonians.
We equip these with the function
$
 \ell^{\sharp}(\bn,m)=4m-2\ell_1(\bn).
$
For any positive integer $m$ and any $\bn \in N$ satisfying $|\bn| \le 2 \lambda_m$,
there are exactly $4 \lambda_m+1$
Hamiltonian chords of $2 H_m$
in the homotopy class
$
 \bn
  .
$
All of them are non-degenerate,
and will be denoted as $q^m_{\bn, i}$
for $-2\lambda_m \le i \le 2\lambda_m$.

\subsection{Calculation of the product}
Let $\phi_{H_m}$ denote the time-one flow of $X_{H_m}$.
For this section,
we adopt the notation that
$L_1=\phi_{H_m}(L_0)$ and
$L_2=\phi_{2H_m}(L_0)$. Time-one chords in $\scrX_{H_m}(L_0)$  (respectively $2H_m$)  correspond to intersection points of  $\phi_{H_m}(L_0)$ and $\phi_{2H_m}(L_0)$. The first step in our proof is to relate our problem of calculating the product in wrapped Floer cohomology to the ordinary product in Lagrangian Floer cohomology for these Lagrangians.

For the purposes of our calculation,
we will define the Floer triangle product with respect to perturbation data of the form
\begin{align}
 X_K=X_{H_{m}} \otimes \gamma,
\end{align}
where $\gamma$ is a \emph{closed} form on the pair of pants such that $\epsilon_0^*\gamma=2dt$,
$\epsilon_1^*\gamma=dt$ and $\epsilon_2^*\gamma=dt$.
This implies that a Floer triangle can then be recast as a pseudo-holomorphic curve
for a \emph{domain-dependent} almost complex structure on $Y$.
Given
$
 y \colon \Sigma \to (Y, J,H),
$
we consider
$
 \ytilde \coloneqq \phi^{\tau}_{H_m} \circ y,
$
where
$
 \tau \colon \Sigma \to [0,2]
$
is the primitive of the closed one form $\gamma$
used to define the pair-of-pants product.
In particular, the product structure then agrees with the usual product operation
\begin{align}
 \Hom(L_1,L_2)\otimes \Hom(L_0,L_1) \to \Hom(L_0,L_2)
\end{align}
in Lagrangian Floer theory
for the domain dependent almost complex structure
$J_{\tau}=(\phi^{\tau})^{ -1}_*(J)$.
We remark that, after perhaps modifying our original $J$ at infinity,
we may assume that $J_\tau$ is fibration-admissible.
As a consequence of this,
for this section
we will consider only pseudoholomorphic maps
and all perturbation data $X_K$ will be taken to be trivial.

\begin{proposition}\label{pr:contII}
Let $J'$ be another fibration-admissible almost complex structure.
Then the continuation maps
\begin{align}
 \HFad(L_{i},L_0;J') \to \HFad(L_{i},L_0;J_\tau)
\end{align}
are the identity.
\end{proposition}

\begin{proof}
The proof follows that of \pref{pr:cont} quite closely.
Namely, observe that such continuation solutions
have intersection number zero with $E$.
By comparing the relative indices $\reldeg{-}$,
we see that the continuation maps must be constant.
\end{proof}

\begin{lemma}
Any Floer triangle $y$ with inputs $p_{\bn, i}$, $p_{\bn',i'}$
has a fixed intersection number $y \cdot \Ebar$ with $\Ebar$.
\end{lemma}

\begin{proof}
The projected input chords
$
 \pi_{\NCx} \lb  p_{\bn, i} \rb
$
and
$
 \pi_{\NCx} \lb p_{\bn', i'} \rb
$
determine
(the projected output chords and)
the relative homotopy class
of the projected Floer triangle $\pi_{\NCx}(y)$ uniquely.
This fixes the intersection number
$
 \pi_{\NCx}(y) \cdot Z
  = p(y) \cdot (Z \times \bC),
$
which is equal to $y\cdot \Ebar$
since $\Ebar = E \cup F$ is the total transform of $Z \times \bC$.
\end{proof}

We set
$
 j \coloneqq y \cdot E,
$
which is a non-negative integer satisfying
$
 j \le \jmax \coloneqq y \cdot \Ebar.
$
Note that $\jmax$ depends
only on $\bn$ and $\bn'$.
The relative gradings imply that one has
\begin{align} \label{eq:mult_rule}
 p_{\bn, i} \cdot p_{\bn',i'}
  = \sum_{j=0}^{\jmax}
   N_j(p_{\bn, i}, p_{\bn',i'}) q_{\bn+\bn', i+i'+j}
\end{align}
for some integers
$N_j(p_{\bn, i}, p_{\bn',i'})$.

The rest of this section is devoted to the computation of
$N_j(p_{\bn, i}, p_{\bn',i'})$.
To simplify notation,
we set
$
 \bsp_1 \coloneqq p_{\bn,i},
$
$
 \bsp_2 \coloneqq p_{\bn',i'},
$
$
 \bsp_0 \coloneqq q_{\bn+\bn',i+i'+j},
$
and
$
 \bsp \coloneqq (\bsp_0, \bsp_1,\bsp_2).
$
Fix $J \in \Jint(D \cup \Ebar)$ and $J_{\NCx \times \bC}$
such that $p$ is pseudoholomorphic.
Recall that
$
 \cR^2(Y, \bsp)
$
is the moduli space of pseudoholomorphic maps
$
 y \colon \Sigma \to Y
$
from disks
$
 (\Sigma, (\zeta_0, \zeta_1, \zeta_2))
$
with three boundary punctures
satisfying
\begin{align} \label{eq:bc1}
 y(\partial_i \Sigma) \subset L_i
\end{align}
and
\begin{align} \label{eq:bc2}
 \lim_{s \to \pm \infty} y(\epsilon_i(s, -)) = \bsp_i
\end{align}
for $i=0,1,2$.
We now introduce some relative moduli spaces
which will be needed in our argument.

\begin{definition}
\begin{enumerate}[1.]
%
 \item
Let
$
 \cR^{2,j}(Y, \bsp)
$
be the moduli space of triples
$
 ((\Sigma, \bszeta), \bsz, y)
$
consisting of disks
$
 (\Sigma, \bszeta)
$
with three boundary punctures
$
\bszeta = (\zeta_0, \zeta_1, \zeta_2),
$
$j$ marked points
$
 \bsz = (z_1, \ldots, z_j)
$
in the interior of $\Sigma$,
and pseudoholomorphic maps
$
 y \colon \Sigma \to Y
$
satisfying
\pref{eq:bc1} and
\pref{eq:bc2}.
 \item
Let
$
 \cR^{2,j}(Y, E, \bsp)
$
be the subspace of
$
 \cR^{2,j}(Y, \bsp)
$
consisting of pseudoholomorphic maps
$
 y \colon \Sigma \to Y
$
satisfying
\begin{align}
 \lc z_i \rc_{i=1}^j = y^{-1}(E).
\end{align}
\end{enumerate}
\end{definition}

Signs can be associated to these moduli spaces
exactly like their non-relative counterparts.
Namely,
given a pseudoholomorphic map $y$,
there exists an isomorphism
\begin{align}
 \abs{\frako_{\bsp_0}}
  \cong \abs{\det(D_y)}
  \otimes \abs{\frako_{\bsp_{1}}}
  \otimes \abs{\frako_{\bsp_{2}}},
\end{align}
where $D_y$ is the extended linearized operator at $y$.
Whenever all curves in $\cR^{2,j}(Y, \bsp)$
are regular and the natural evaluation map
\begin{align}
 \ev \colon \cR^{2,j}(Y, \bsp) \to Y^j
\end{align}
is transverse to $E^j$,
this enables us to assign a map
between orientation spaces for each isolated map
$y \in \cR^{2,j}(Y, E, \bsp)$.
This is because in this situation,
the determinant of the linearized operator can be expressed as
$
 \det(D_y) \cong \otimes _i \det(TY_{y(z_i)})
  \otimes \det (TE_{y(z_i)})^{\vee}
$
which is canonically oriented.
There is an orientation-preserving surjection from
$
 \cR^{2,j}(Y,E,\bsp)
$
to
$
 \cR^{2}(Y,\bsp),
$
which is $j!$ to $1$
because of the choice of an ordering of the marked points.

\begin{definition}
\begin{enumerate}[1.]
 \item
Let
$
 \cR^{2,j}( \NCx \times \bC, \bsp)
$
be the moduli space of triples
$
 ((\Sigma, \bszeta), \bsz)
$
consisting of disks
$
 (\Sigma, \bszeta)
$
with three boundary punctures
$
 \bszeta = (\zeta_0, \zeta_1, \zeta_2)),
$
$j$ marked points
$
 \bsz = (z_1, \ldots, z_j)
$
in the interior of $\Sigma$,
and pseudoholomorphic maps
$
 y \colon \Sigma \to \NCx \times \bC
$
satisfying
\pref{eq:bc1} and
\pref{eq:bc2}.
 \item
Let
$
 \cR^{2,j}(Y, E, \bsp)
$
be the subspace of
$
 \cR^{2,j}(Y, \bsp)
$
consisting of pseudoholomorphic maps
$
 y \colon \Sigma \to \NCx \times \bC
$
satisfying
\begin{align}
 \lc z_i \rc_{i=1}^j = y^{-1}(Z \times 0).
\end{align}
\end{enumerate}
\end{definition}

For suitably generic $J_{\NCx \times \bC}$,
the linearized operator is surjective at all curves in
$
 \cR^{2,j}(\NCx \times \bC, \bsp)
$
and the evaluation map
\begin{align} \label{eq:eval}
 \ev \colon \cR^{2,j}(\NCx \times \bC, \bsp) \to (\NCx \times \bC)^j
\end{align}
is transverse to $(Z \times 0)^j$.
For such $J_{\NCx \times \bC}$ and
when $\cR^{2,j}( \NCx \times \bC, \bsp)$ is isolated,
the moduli space can be given a relative orientation
in a similar fashion to $\cR^{2,j}(Y, E, \bsp)$.
We now show that split almost complex structures of the form
$J_{\NCx \times \bC}=(J_B, J_{\bC})$
on $\NCx \times \bC$ are sufficiently generic
to achieve these two conditions.
The Lagrangians $L_i$ are also split Lagrangians, i.e.,
can be written as $\Lbar_{i} \times \pi_{\bC}(L_{i})$.
Thus,
for this class of almost complex structures,
we view our holomorphic curve $p \circ y$ as pairs $(y_1,y_2)$ where
\begin{align} \label{eq:y1_cond}
\left\{
\begin{aligned}
 & y_1 \colon \Sigma \to \NCx, \\
 & y_1(\partial_k \Sigma) \subset \Lbar_k, \\
 & \lim_{s \to \pm \infty} y_1(\epsilon_k(s, -))
  = \underline{q}_{\bn+\bn'}, \underline{p}_{\bn}, \underline{p}_{\bn'},
\end{aligned}
\right.
\end{align}
and
\begin{align} \label{eq:y2_cond}
\left\{
\begin{aligned}
 & y_2 \colon \Sigma \to \bC, \\
 & y_2(\partial_k \Sigma) \subset \pi_{\bC}(L_k), \\
 & \lim_{s \to \pm \infty} y_2(\epsilon_k(s, -)) = q_{i+i'+j}, p_{i}, p_i'.
\end{aligned}
\right.
\end{align}
We have 
\begin{align}
 y_1 \cdot Z=\Ebar \cdot y=\jmax.
\end{align}
For generic $J$,
we can assume that $y_1^{-1}(Z)$ are distinct points, and
we fix an ordering $(z_1, \ldots, z_{\jmax})$ of these points.
Given an ordered subset
$
 \lb i_1, \ldots, i_j \rb
$
of
$
 \lc 1, \ldots , \jmax \rc,
$
the triple
$
 ((\Sigma, \bszeta), \bsz, (y_1, y_2))
$
with
$
 \bsz = (z_{i_1}, \ldots, z_{i_j})
$
and $(y_1, y_2)$ satisfying
\pref{eq:y1_cond} and \pref{eq:y2_cond}
lies in
$
 \cR^{2,j}(\NCx \times \bC, Z \times 0, \bsp)
$
if and only if
\begin{align} \label{eq:mapstoC}
 y_2^{-1}(0)= \lc z_{i_{1}},\cdots, z_{i_{j}} \rc.
\end{align}

\begin{lemma} \label{lm:splittrans}
If
$
 J_{\NCx \times \bC} = (J_B, J_\bC)
$
is the product of a generic almost complex structure $J_B$ adapted to $Z$
and the standard almost complex structure $J_\bC$,
then all pseudoholomorphic maps in
$
 \cR^{2,j}( \NCx \times \bC, \bsp)
$
are Fredholm regular and
the evaluation map \pref{eq:eval} is transverse to
$(Z \times 0)^j$.
\end{lemma}

\begin{proof}
We have split Lagrangian boundary conditions and
no curve is constant when projected to either of the product factors,
so it is clear that we can achieve surjectivity of the operator $D_y$
using split $J$.
The transversality of the evaluation map \pref{eq:eval} to $Z \times 0$ is equivalent
to the transversality of $y_1$ with $Z$,
which
can be achieved by a small perturbation of $J_B$,
and the transversality with $0^j$
of the evaluation map for the moduli space of curves of the form \pref{eq:mapstoC}
to $\bC^j$.
It suffices to prove that the tangent space to $y_2$
of the moduli space of pseudoholomorphic maps of the form $\pref{eq:mapstoC}$
is trivial.
An element of this tangent space is precisely an element of $\Ker(D_{y_{1}})$
which vanishes at $z_{i_{1}}, \ldots, z_{i_{j}}$.
The key for us is that $y_1^*(T\bC_u)$ is  a line bundle
and our element is a section satisfying a $\partialbar$-equation
with a prescribed number of zeroes.
We may therefore apply \cite[Proposition 11.5]{MR2441780},
which controls the number of zeroes on such solutions on line bundles,
to show that this element must be trivial.
\end{proof}

Fix $J \in \Jint(D \cup \Ebar)$ and $J_{\NCx \times \bC}$
be a split almost complex structure such that $p$ is pseudoholomorphic.
Given a pseudoholomorphic triangle $y$ in $\cR^2(Y,E,\bsp)$
with boundary on $L_i$,
we can compose it with the structure morphism
$p \colon \Ybar \to \NCx \times \bC$
to obtain a curve $p \circ y$.
Moreover given a curve $p \circ y$,
we can reconstruct $y$ as the proper transform of $p \circ y$.
The condition that $y \in R^{2,j}(Y,E,\bsp)$
is equivalent to
$
 p \circ y \in  \cR^{2,j}( \NCx \times \bC, Z \times 0, \bsp).
$
In particular, whenever curves in
$
 \cR^{2,j}( \NCx \times \bC, Z \times 0, \bsp)
$
are isolated, so are curves in $R^{2,j}(Y,E, \bsp)$.
In fact, this persists at the level of linearized operators.

\begin{lemma}
Let $J \in \Jint(D \cup \Ebar)$ and
$J_{\NCx \times \bC}$ be a split almost complex structure
satisfying the conditions of \pref{lm:splittrans}
such that $p$ is pseudoholomorphic.
Let further $y$ be the proper transform of a curve
in $\cR^{2,j}( \NCx \times \bC, Z \times 0, \bsp)$.
Then the linearized operator $D_y$ is surjective.
\end{lemma}
\begin{proof}
Since all of our moduli spaces are isolated, it suffices to check that any element $u \in \operatorname{ker}(D_y)$ such that the linearized evaluation map is tangent to $E$ is zero. Given such a $u$, we may push this down to an element $p \circ u \in \operatorname{ker}(D_{p\circ y})$ such that the linearized evaluation maps at $z_i$ are tangent to $Z \times 0$. By our hypothesis, such an element must be zero. For every point $z \neq z_i$  the map $y^*(TY)_z \to (p \circ y)^*(T(\NCx \times \bC)_p(z)$ induced by the projection from $y$ to $p\circ y$ is  injective and hence it, it follows that $u=0$.
\end{proof}
The key enumerative calculation is the following:

\begin{lemma} \label{lm:enum}
The numbers $N_{j}(p_{\bn, i}, p_{\bn',i'})$
appearing in \pref{eq:mult_rule} are $\displaystyle{\binom{\jmax}{j}}$.
\end{lemma}

\begin{proof}
We first determine the signed count of elements in the moduli space
$\cR^{2,j}( \NCx \times \bC, Z \times 0, \bsp)$,
up to dividing by $j!$
to remove for the redundancy from the ordering of the marked points.
After picking $j$ points $z_{i_{1}} \cdots, z_{i_{j}}$,
it suffices to prove that the signed count of pairs $(y_1,y_2)$ is 1.
Let $\mu$ be the standard moment map on the disc and
consider a vertically admissible Hamiltonian $v(\mu)$
which is of slope $\lambdav>1$ near $\mu=0$.
Let $\alpha$ be the Hamiltonian orbit
whose homotopy class in $\bCx$ is primitive.
Let $S$ be $\bP^1 \setminus \lc 0 \rc$
with a distinguished marked point at $z=\infty$
and a negative cylindrical end near $z=0$ given by
\begin{align}
 (s,t) \to e^{s+\sqrt{-1}t}.
\end{align}
Fix a subclosed one-form $\beta$
which restricts to $dt$ on the cylindrical end and which restricts to zero in a neighborhood of $z=\infty$.
To be explicit, we consider a non-negative, monotone non-increasing cutoff function $\rho(s)$ such that
\begin{align}
 \rho(s) =
\begin{cases}
 0 & s \gg 0, \\
 1 & s \ll 0,
\end{cases}
\end{align}
and set $\beta=\rho(s)dt$.
The remainder of our proof is divided into two steps:

\emph{Step 1.} Consider solutions $y_I$ to the equation
 \begin{align} \label{eq:mapstoC2}
\left\{
\begin{aligned}
 & y_I: S \to \bC \\
&  (dy_I- X_{v(\mu)} \otimes \beta)^{0,1} = 0 \\
 & y_I(0)=0 \\
 & \lim_{s \to - \infty} u(\epsilon(s, -)) = \alpha
\end{aligned}
\right.
\end{align}
There is a well-known isomorphism
\begin{align} \label{eq:stringtop1}
\SH^0(\bCx) \cong H_0(C_{1-*}(\mathcal{L} \bS^1)) \cong \bC[\alpha,\alpha^{-1}].
\end{align}

We wish to prove that the signed count of curves $y_I$ as above is $+1$.
Let $S_2$ be the disc with a fixed boundary marked point $\zeta$ and with one interior puncture. Equip $S_2$ with a positive strip-like end near the puncture. Let $L=\bS^1$ denote the zero section of $T^* \bS^1 \cong \bCx$ and $x \in L$. A consequence of the isomorphism \pref{eq:stringtop1} is that for suitable choices of Floer data, the signed count of curves with domain $S_2$ with Hamiltonian input $\alpha$ and with the marked point $\zeta$ passing through $x$ is +1.  Now view $\bC$ as a compactification of $\bCx$ and glue $S$ to $S_2$ along $\alpha$ to obtain a disc with boundary on $L$ and with a marked point passing through $z=0$. If we then the perturbation data to zero, the result then follows from the obvious fact that there is a unique disc in $\bC$ with a marked point at $z=0$ and a boundary marked point at any point on the Lagrangian $L$.

\emph{Step 2.} With this established, let $S_3$ be a Riemann surface with $j$ positive cylindrical ends, two positive strip like ends and one negative strip like end.  There is also an isomorphism \begin{align} \label{eq:stringtop2} \HW^0(T_q) \cong \bC[p_1,p_1^{-1}] \end{align}  After choosing suitable perturbation data it follows from \pref{eq:stringtop1}, \pref{eq:stringtop2}, together with the TQFT structure that the signed count of curves $S_3$ with chord inputs $p_i$, $p_{i'}$,  Hamiltonian orbit inputs $\bsalpha$ and output $q_{i+i'+j}$ must be $+1$. At each Hamiltonian input, we may then glue on the above curves with domain $S$ to obtain maps satisfying all of the requirements of \pref{eq:mapstoC} except that the marked points $z'_{i_j}$ may be different and there will be non-trivial perturbation data $X_K$ (these depend on the gluing parameters along the ends).

 We may deform away the Floer data and deform the marked points to $z_{i_{j}}$ as above. The signed count of these curves does not change because there is no possible breaking and it follows that the signed count of elements $y_1$  satisfying \pref{eq:mapstoC} is +1. The signed count of curves $y_2$ is also +1. Since we have split Lagrangian boundary conditions and almost complex structures, the count of configurations of curves of the form $(y_1,y_2)$ contribute +1. One readily checks that the bijection between $\cR^{2,j}( \NCx \times \bC, Z \times 0, \bsp)$ and $\cR^{2,j}(Y, E, \bsp)$ preserves orientations, so the signed count of curves $(y_1,y_2)$ contributes $+1$ to the Floer product as well. 
\end{proof}

\begin{theorem}
Given two generators, $p_{\bn, i}$ and $p_{\bn',i'}$, let $m= \Ebar \cdot y$.
Then we have
\begin{align} \label{eq:product2}
 p_{\bn, i} \cdot p_{\bn',i'}
  = \sum_{j=0}^{m}
   \binom{m}{j} q_{\bn+\bn', i+i'+j}.
\end{align}
\end{theorem}
	
\begin{proof}
This is a consequence of  \pref{lm:enum}.
\end{proof}

We next observe the following fact concerning continuation maps.
The above theorem allows us to calculate the continuation map:
\begin{align}
 \frakc: \HFad (L_0;H) \to \HFad(L_0;2H)
\end{align}

\begin{lemma}
For an arbitrary element $p_{\bn,i},$
one has
$
 \frakc(p_{\bn,i}) = q_{\bn,i}.
$
\end{lemma}

\begin{proof}
The arguments concerning relative index in \pref{pr:cont} and \pref{pr:contII}
together with the fact that all of our Floer products
are in fact defined over $\bZ$ imply that
the continuation map must send
$
 p_{\bn,i} \to \pm q_{\bn,i}.
$
A standard Floer theoretic argument shows that
the continuation map sends
$
 \frakc(p_{\bn,i})
  = p_{0,0} \cdot p_{\bn,i}.
$
The result follows immediately
from the fact that the Floer triangle $y$
with inputs $p_{0,0}$,  $p_{\bn,i}$ and outputs $q_{\bn,i}$
must have intersection number zero
since we already knew that
$
 \frakc(p_{\bn,i}) = \pm q_{\bn,i}.
$
\end{proof}

\begin{lemma}
The intersection number is given by
\begin{align}
 \Ebar \cdot y=\ell_2(\bn,\bn').
\end{align}
\end{lemma}

\pref{eq:product1} and \pref{eq:product2}
show that
the linear map
\begin{align}
 \Mirad^{L_0} \colon \HWad(L_0) \simto H^0(\cO_\Yv)
\end{align}
defined by
\begin{align}
 p_{\bn,i} \mapsto p^{i} \chi_{-\bn,\ell_1(\bn)}
\end{align}
is an isomorphism of algebras, and
\pref{th:main2} is proved.

\section{Comparison of Floer theories}
 \label{sc:Comp}

\subsection{The Liouville domain}
 \label{sc:Domain}

For this section,
we define
\begin{align} \label{eq:pi_Bhat}
 \pi_{\Bhat} \colon Y \to \Bhat \coloneqq N_\bR \times \bR, \quad
 x \mapsto \lb \log \abs{w_1}, \log \abs{w_2}, \log \mu \rb.
\end{align}
Let $r_3$ be the third coordinate on $N_\bR \times \bR$ and
\begin{align}
 \pi_{N_\bR}: \Bhat \to N_\bR
\end{align}
be the natural projection.
We will begin by considering a ``toric'' version $T_R$ of our Liouville domain.
The Liouville domain we will actually study, $Y_R$, will be constructed
by modifying $T_R$ along $\pi_{\NCx}^{-1}(U_Z)$. 
Define 
\begin{align} \label{eq:pi0_Bhat}
 \pi^0_{\Bhat} \colon \NCx \times \bCx \to \Bhat,
\end{align}
in which we replace the above moment map
by the toric moment map
$
 \mu_T= \abs{u}^2/2.
$
In coordinates $(r_i,\theta_i)$,
the manifold $\NCx \times \bCx$ is equipped with the symplectic form
\begin{align} \label{eq:omegator}
 \omega^T = dr_1 \wedge d \theta_1 + dr_2 \wedge d \theta_2 + e^{r_3} dr_3 \wedge d \theta_3
\end{align}
For some $R$,
$\mu_0$ and $\mu_1$
with $R \gg 0$ and $0 < \epsilon - \mu_0 \ll \epsilon \ll \epsilon + \mu_1$,
consider the region in $\Bhat$ defined by
$
 \epsilon-\mu_0 \le \mu_T \le \epsilon +\mu_1
$
and
$
 \Hb \leq R.
$
We choose $\mu_1$ sufficiently large so that
$\chi(u,\bw)=0$ for $\mu_T \ge \epsilon + \mu_1/2$.
This forms a cornered region in $\Bhat$,
with corners given by
\begin{align}
 \lc \Hb = R, \, \mu_T = \epsilon - \mu_0 \rc
  \cup
 \lc \Hb = R, \, \mu_T = \epsilon + \mu_1 \rc.
\end{align}
We may ``round corners'', i.e.,
make a  deformation in a neighborhood of these corners
to obtain a convex region $D_{R} \in \Bhat$ such that
the boundary $\partial D_R$ is smooth and the region 
\begin{align}
 T_{R} \coloneqq \lb \pi^0_{\Bhat} \rb^{-1}(D_R)
\end{align}
is a Liouville domain.
To be more precise,
we require the following conditions:
\begin{align}
 \alitem{
For every leg $\Pi_i$ and every point
$
 x \in \partial D_R \cap \pi_{N_\bR}^{-1}(U_{\Pi_i}),
$
the boundary $\partial D_R$ is locally defined by
$D_i(r_{\alpha-\beta}^{\perp},\mu_T)=0$
for some function $D_i \colon \bR^2 \to \bR$.
} \label{eq:D1} \\
 \alitem{
Fix a number $\delta_R$, which is relatively small compared to $R$.
For the region bounded by $\Hb \leq R-\delta_R$ 
under $\pi_{N_\bR}$,
$\partial D_R$ agrees with
$
 \lc \mu_T=\epsilon-\mu_0 \rc \cup \lc \mu_T = \epsilon+ \mu_1 \rc.
$
} \label{eq:D2} \\
 \alitem{
In $\pi_{N_{\bR}}^{-1}(U_{\Pi_{i}})$,
when $\epsilon-\mu_0/2 \leq \mu_T \leq \epsilon + \mu_1/2$,
$\partial D_R$ agrees with $\Hb=R$.
} \label{eq:D2'}
\end{align}
In particular,
\pref{eq:D1} implies that for any point
$
 x \in \partial D_R \cap \pi_{N_\bR}^{-1}(U_{\Pi_i}),
$
the projection to $N_\bR$ of the normal vector to $\partial{D_{R}}$ at $x$
is parallel to $\Pi_i$.
We wish to construct a version of this manifold which is adapted to $Y$.
We first consider
$
 Y^0_R \coloneqq \pi_{\Bhat}^{-1} (D_R),
$
which is a Liouville domain for a suitable choice of $D_R$.
However, for certain purposes,
it is convenient to modify this domain further
to have better control of the Reeb flow along the boundary
near the discriminant locus.
We remove a tubular neighborhood of each leg $\Pi_i$
from the boundary to obtain
$
 \partial Y^0_R \setminus (\pi_{\Bhat} \circ \pi_{N_\bR})^{-1}(\cup_i U_{\Pi_i}).
$
Over each region $(\pi_{\Bhat} \circ \pi_{N_\bR})^{-1}(U_{\Pi_i})$,
we may glue in the hypersurface given by $D_i(\rho,\mu)=0$
to obtain a hypersurface $\partial Y_R$.
 Let $Y_R$ be the Liouville domain bounded by this hypersurface. 

Outside of $U_{Z}$,
the manifolds $Y_R$ and $T_R$ are canonically identified.
Let $\partial_i Y^0_R$ be
the part of $\partial Y^0_R$ defined by
$
 c_i r_{\alpha-\beta}^\perp=R.
$
Over
$
 \lb \pi_{N_\bR} \circ \pi_{\Bhat} \rb^{-1} \lb \bigcup_i U_{\Pi_i} \rb,
$
we glue in the piece of the hypersurface $\partial Y_R$ where $D_i(\rho,\mu)=\rho$
to obtain a hypersurface with boundary
which we denote by $\partial_i Y$.
We view $\NCx \times \bCx$ as $T^*\Bhat/\bZ^3$
which we identify
using the standard Euclidean inner product
with $\Bhat \times \Bhat/ \bZ^3$.
For any point $x \in \partial D_R$,
let $\vec{n}_x$ denote the normal vector to $\partial D_R$
in the diagonal metric on $\Bhat$ with coefficients
\[
\begin{pmatrix}
1 & 0 & 0 \\
0 & 1 & 0 \\
0 & 0 & e^{r_3}
\end{pmatrix}.
\]

\begin{lemma} \label{lm:TR_Reeb}
The Reeb flow on $\partial T_R$ fibers over $\partial D_R$ and
points in the direction of $\vec{n}_x$ for every $x \in \partial D_R$.
\end{lemma}

\begin{proof}
The tangent space to $\partial T_R$ is given
by the span of $(T_{D_{R}}, \partial_{\theta_{i}})$.
Set  $\vec{v}=\sum_i a_i \partial_{\theta_{i}}$.
We have
$
 \omega^{T}(\vec{v}, \vec{v}') = 0
$
for any $\vec{v}'$
which is a linear combination of the $\partial_{\theta_{i}}$
and for any $\vec{v}' \in T_{D_{R}}$
if and only if
$
 a_1\partial_{r_{1}} +a_2\partial_{r_{2}}+ a_3\partial_{r_{3}}
$
is normal to $\vec{v}'$ in the above diagonal metric.
\end{proof}

\begin{definition}
A non-zero element $\bn$ of $N$ is \emph{regular}
if it does not point
in the direction of any leg $\Pi_i$ in $N_\bR$.
\end{definition}

We also consider the following conditions:
\begin{align}
 \alitem{
Every Reeb orbit of $\partial Y_R$,
whose homotopy class $\bn \in \pi_1(Y) \cong N$ is regular,
projects by $\pi_{N_\bR} \circ \pi_{\Bhat}$
to the chamber $C_{\alpha_0}$ for the $\alpha_0 \in A$
such that $\alpha_0(\bn)=\sup_{\alpha \in A} \alpha(\bn)$.
Every Reeb orbit with non-regular homotopy class projects
to the corresponding neighborhood $U_i$ of the leg.
} \label{eq:D3} \\
 \alitem{
The region where $\partial Y_R$ agrees with $\lc \mu=\epsilon-\mu_0 \rc$ or $\lc \mu=\epsilon + \mu_1 \rc$ is precisely where $\rho \leq R-\delta_R$.
} \label{eq:D4} \\
 \alitem{
For any $\bw \in \NCx$,
the image of
$
 \partial_i Y \cap \pi_{\NCx}^{-1}(\bw)
$
by $\mu$
is
$
 [\epsilon-\mu_0/2,\epsilon + \mu_1/2].
$
} \label{eq:D5}
\end{align}


We next deform the Liouville form $\theta_\epsilon$
to a Liouville form $\theta_\epsilon'$ on $Y_{R}$
so that $\Ebar$ is preserved by both the Liouville flow and
the Reeb flow along the boundary $\partial{Y_{R}}$.
This will be necessary since for some arguments
it will be necessary to choose the almost complex structure
to be both adapted with respect to the Liouville domain and
to preserve the divisor $\Ebar$.

\begin{lemma} \label{lm:deformingprim}
There is a Liouville one-form $\theta_{\NCx} '$ on $\NCx$
which agrees with $\theta_{\NCx}$ outside of $U_{Z}$
and such that the Liouville vector field preserves $Z$.
\end{lemma}

\begin{proof}
Let
$
 \iota_Z \colon Z \to \NCx
$
be the inclusion and
$
 \pi_Z \colon Z \times \bD_\delta \to Z
$
be the first projection,
where
$
 \bD_\delta \coloneqq \lc (x, y) \in \bR^2 \relmid x^2 + y^2 < \delta \rc
$
is a disk of radius $\delta$.
The symplectic tubular neighborhood theorem gives
an embedding
$
 \iota_{Z \times \bD_\delta} \colon Z \times \bD_\delta \to \NCx
$
satisfying
\begin{align}
 {\iota_{Z \times \bD_\delta}}^* \omega_{\NCx}
  = dx\wedge dy + {\pi_Z}^* {\iota_Z}^* \omega_{\NCx}.
\end{align}
Note that
$
 xdy + {\pi_Z}^* {\iota_Z}^* \theta_{\NCx}
$
is a primitive of ${\iota_{Z \times \bD_\delta}}^* \omega_{\NCx}$,
and any other primitive differs from this
by a closed form. Along each leg of $Z$, we can choose the standard tubular neighborhood of $Z$
defined by
$
 \lc \abs{\bw^{\bsalpha_i-\bsbeta_i}+t^{v(\bsbeta)-v(\bsalpha)}} \le \delta \rc,
$
and the coordinate of the disc factor to be
\begin{align}
 x \coloneqq \log \abs{\bw^{\bsalpha_i-\bsbeta_i}} -(v(\bsbeta)-v(\bsalpha)) \log t, \quad
 y \coloneqq \Arg \bw^{\bsalpha_i-\bsbeta_i}-\pi,
\end{align}
setting $x_{0,i}=-(v(\bsbeta)-v(\bsalpha)) \log t$,
we have
\begin{align} \label{eq:defprim}
 {\iota_{Z \times \bD_\delta}}^* \theta_{\NCx}
  = (x-x_{0,i}) dy + {\pi_Z}^* {\iota_Z}^* \theta_{\NCx}.
\end{align}
Therefore,
we have
\begin{align}
 {\iota_{Z \times \bD_\delta}}^* \theta_{\NCx}
  - \lb x dy + {\pi_Z}^* {\iota_Z}^* \theta_{\NCx} \rb
  = {\pi_Z}^* \eta + dh,
\end{align}
where $\eta$ is a compactly supported closed one-form on $Z$ and
$h=x_{0,i}y$  along
the legs of $Z$.
Let $\rho \coloneqq \sqrt{x^2+y^2}$ be the radial coordinate on $\bD_\delta$ and
consider a cutoff function $\chi(\rho)$
such that $\chi(\rho)=1$ for $\rho<\delta/4$ and
$\chi(\rho)=0$ for $\rho>\delta/2$.
The the one-form
\begin{align}
 \theta'_{Z \times D_{\bD_\delta}}
  \coloneqq
   x dy
   + {\pi_Z}^* {\iota_Z}^* \theta_{\NCx}
   + {\pi_Z}^* \eta
   + d(\chi h)
\end{align}
is a primitive of ${\iota_{Z \times \bD_\delta}}^* \omega_{\NCx}$ and
the corresponding Liouville vector field
is tangent to $Z$.
Since $\theta'_{Z \times D_{\bD_\delta}}$
agrees with ${\iota_{Z \times \bD_\delta}}^* \theta_{\NCx}$ when $\rho > \delta/2$,
one can glue $\lb {\iota_{Z \times \bD_\delta}}^{-1} \rb^* \theta'_{Z \times D_{\bD_\delta}}$
with $\theta_{\NCx}$ to obtain a Liouville one-form
$\theta'_{\NCx}$
such that the corresponding Liouville vector field preserves $Z$.
\end{proof}

Our proof of  \pref{lm:deformingprim} shows that
if we assume that the non-compact legs $\Pi_i$ pass through the origin,
then such a form can be chosen to agree with $\theta_{\NCx}$
outside of small neighborhoods of the vertices of $Z$.
Choose a primitive as in the proof of \pref{lm:deformingprim} and
let $\theta_{\NCx \times \bC}'$ be the induced primitive one form
of $\omega_{\NCx \times \bC}$ on $\NCx \times \bC$.
Finally, set
\begin{align}
 \theta_\epsilon' \coloneqq \thetavc +p^*\theta_{\NCx \times \bC}'.
\end{align}

The Liouville coordinate in the fiber is given
by $C|\mu-\epsilon|$ for some $C>0$.
Let $V_{\theta'_{\epsilon}}$ be the Liouville vector field
which is the $\omega_\epsilon$
dual of $\theta'_{\epsilon}$.

\begin{lemma} \label{lem:flowest}
There is a constant $B$ such that
when $\rho \geq R$ and
$\mu \leq \epsilon-\mu_0$ or $\mu \leq \epsilon + \mu_1$,
we have estimates
\begin{align} \label{eq:flowesthor}
 d\rho (V_{\theta'_{\epsilon}}) \leq B \rho
\end{align}
and
\begin{align} \label{eq:flowestver}
 d \abs{\mu-\epsilon}(V_{\theta'_{\epsilon}}) \leq  \abs{\mu-\epsilon}.
\end{align}
\end{lemma}

\begin{proof}
The key to this estimate is again the local $T^2$ action
in the tubular neighborhoods.
Both \pref{eq:flowesthor} and \pref{eq:flowestver} are immediate
outside of the tubular neighborhood.
Let us first consider \pref{eq:flowesthor},
when, inside of the tubular neighborhood,
we have
\begin{align}
d \rho (V_{\theta'_{\epsilon}})
& = \omega(V_{\theta'_{\epsilon}}, c_i\partial_{\theta_{\bsalpha-\bsbeta}}^\perp) \\
& =  p^*\theta_{\NCx \times \bC}'(c_i\partial_{\theta_{\bsalpha-\bsbeta}}^\perp)-  d^c\lb \chi(G_i)\log(F_i) \rb (c_i\partial_{\theta_{\bsalpha-\bsbeta}}^\perp) \\
& = \rho  
\end{align}

The second estimate  \pref{eq:flowestver} is similar.
Namely, we have
$
 \theta_\epsilon'(\partial_\theta)=\mu-\epsilon
$
and hence
$
 d \abs{\mu-\epsilon} (V_{\theta'_{\epsilon}})
  = \abs{\mu-\epsilon}.
$
\end{proof}

Let $r_Y$ denote the Liouville coordinate on the completion of $Y_R$.
It is not difficult to see that
the Liouville flow is not complete
as we approach the divisor $D$.
However, given  $\mu_0$,
we can assume that we have an embedding
of the region of the symplectic completion
\begin{align}
 Y_R \cup \partial Y_R \times [1, c_{\mu_0})  \to Y.
\end{align}
\pref{lem:flowest} implies the following direct analogue of
\cite[Lemma 5.7]{MR2497314},
which is crucial in McLean's proof of
\cite[Theorem 5.5]{MR2497314}.

\begin{corollary} \label{cr:estimate}
When $\rho \geq R$ and
$\mu \leq \epsilon-\mu_0$ or $\mu \leq \epsilon + \mu_1$,
one has
\begin{align} \label{eq:flowest1}
 \rho \leq e^B r_Y
\end{align}
and
\begin{align} \label{eq:flowest2}
 C|\mu-\epsilon|  \leq  r_Y.
\end{align}
\end{corollary}

\begin{proof}
This follows from \pref{lem:flowest}
because $r_Y$ is defined to be the integration
of the Liouville vector field for time $\log(r_Y)$.
\end{proof}

As a consequence of \pref{cr:estimate},
we may take $c_{\mu_0} \to \infty$ as $\mu_0 \to 0$.
For the moment take $\mu_0$ sufficiently small
so that $c_{\mu_0} \geq 2$. The key point of modifying the primitive can be seen in the following lemma:

\begin{lemma} \label{lm:flow}
For suitable choices of R and $\mu_1$, the Liouville flow on $Y_R$ preserves $\Ebar$ for $r_Y \geq 1$.
\end{lemma}

\begin{proof}
Consider equation \pref{eq:symp_form3}.
We may expand this out further to see that at points along points lying in $E$,
\begin{align} \label{eq:symp_formE}
  \omega_\epsilon &=
  \pi_{\NCx}^*\omega_{\NCx} +\frac{\sqrt{-1}}{4} |v_0|^2 dh \wedge d{\bar{h}}
   + \frac{\sqrt{-1} \epsilon}{2\pi}
    \partial \partialbar \lb \log \lb 1+|v_0|^2 \rb \rb
\end{align}
Similarly,
for the primitive $\theta_\epsilon$,
we have
\begin{align} \label{eq:primitiveE}
  \theta_\epsilon' &=
  \pi_{\NCx}^*\theta_{\NCx}' - \frac{\epsilon}{4\pi} d^c \lb \log \lb 1+|v_0|^2 \rb- \log |v_0|^2 \rb
\end{align}
along points lying in $E$.
The second term of \pref{eq:symp_formE} vanishes
on vector fields tangent to $Z$.
Since the Liouville vector field $V_{\theta'}$ on $\NCx$ is tangent to $Z$,
we have that viewing this chart as a product manifold with coordinates $(\bw,v_0)$
\begin{align}
 \pi_{\NCx}^*\omega_{\NCx} +\frac{\sqrt{-1}}{4} |v_0|^2 d h \wedge d {\bar{h}}((V_{\theta'},0),-)
  = \pi_{\NCx}^*\theta_{\NCx}'(-)
\end{align}
as required. On the other component of $\Ebar \cap \partial{Y}_R$, after choosing $\mu_1$ and $R$ appropriately assume that that the function $\chi(u,\bw)=0$ on $\Ebar \cap \partial{Y}_R$ except when $h=t^{-v(\bsalpha)}w^{\bsalpha}+ t^{-v(\bsbeta)}w^{\bsbeta}$.
Following the notation and proof of \pref{pr:Hba},
we have that the $\bw^{\bsalpha-\beta}$ components of $d^c(G_i)$ and $d^c(F_i)$ vanish over this locus which implies the same for $d^c\lb \chi(G_i)\log(F_i) \rb $. For any point in this portion of $\Ebar \cap \partial{Y}_R$, let $S$ be the subspace of the tangent space spanned by $(\partial_{r_{\alpha-\beta}},\partial_{\theta_{\alpha-\beta}})$. Restricting to $S$ at the points in question, we have that the one form $\theta'_\epsilon$ vanishes there. As can be seen from \pref{eq:express1}, we have that the symplectic orthogonal to the subspace $S$ is given by the span of $(\partial_{r_{\alpha-\beta}^{\bot}},\partial_{\theta_{\alpha-\beta}^{\bot}}, \partial_{|u|},\theta_u)$. It follows that the projection of the Liouville field to those coordinates vanishes and the result is proven. 
\end{proof}

\subsection{} \label{sc:Reeb_orbits}

Let us discuss the Reeb flow.
Away from $\pi_{\NCx}^{-1}(U_Z)$,
the Reeb flow corresponds with that on $\partial T_R$
as described  in \pref{lm:TR_Reeb}.
Within $\pi_{N_\bR}^{-1}(U_{\Pi_i})$,
the Reeb flow acts on the $T^2$ orbits of the local torus action
with slope $(\partial_{\rho} D_i(\rho,\mu), \partial_{\mu} D_i(\rho,\mu))$.
For all other points in $\pi_{\NCx}^{-1}(U_Z)$,
we have by \pref{eq:D2} that
the Reeb flow rotates the fibers of the conic bundle.
In particular,
the Reeb flow also preserves $\Ebar$
as well as the neighborhood $\pi_{\NCx}^{-1} \lb U_Z \rb$. 

Let $\partial L_0$ denote the Legendrian $L_0 \cap \partial Y_R$.
The Reeb flow on $\partial T_R$ or $\partial Y_R$ is
degenerate.
However,
we have good control over the families of Reeb orbits and
Legendrian Reeb chords of $\partial L_0$.
There are families $\cF_{(\bn,i)}$
of parameterized Reeb orbits
such that the entire family avoids $\pi^{-1}(U_Z)$
which come in families of the following type:
\begin{itemize}
 \item
(Type I)
For every $(\bn,0)$ with $\bn$ regular,
there is a manifold with boundary of parameterized Reeb orbits
diffeomorphic to $I \times T^3$.
These Reeb orbits occur
where the $r_3$-component of the normal vector to $\partial D_R$ vanishes.
By the assumption \pref{eq:D2'} on the shape of $D_R$,
this fibers over an interval in the $\bR^3$ projection.
 \item
(Type II)
For every other class $(\bn,i)$ with $\bn$ regular,
we may assume that we have a $T^3$ worth of orbits. 
\end{itemize}
In both cases we denote the family by $\cF_{(\bn,i)}$ for $\bn$ regular.
When viewed as Reeb orbits of $\partial T_R$,
these identifications are $T^3$ equivariant with respect  to the natural actions on both sides.
   
There are additional families of Reeb orbits in homotopy class
with $\bn \neq 0$ non-regular:
\begin{itemize}
 \item
(Type III)
Those families of orbits beginning at any point in $\partial_i Y$.
We will denote these by $\cF_{(\bn,0)}$ for $\bn$ non-regular. 
 \item
(Type IV)
Those families which contain Reeb orbits in $\pi^{-1}(U_{\Pi_j})$
with $\partial_{\mu} D_j(\rho, \mu) \neq 0$.
We may assume that these families are diffeomorphic to $I \times T^3$.
We will denote these by $\cF_{(\bn,i)}$
for $\bn$ non-regular and $i \neq 0$
corresponds to the winding number of the flow
in the fibers of the conic fibration.
\end{itemize}
Finally there are contractible Reeb orbits:
\begin{itemize}
 \item
(Type V)
When $\lc \mu= \epsilon-\mu_0 \rc$ if $i>0$ or
$\lc \mu=\epsilon-\mu_1 \rc$ if $i<0$.
These are denoted by $\cF_{(0,i)}$.
\end{itemize}
For each of these families $\cF_{(\bn,i)}$, the Reeb flow gives $\cF_{(\bn,i)}$ the structure of an oriented $\bS^1$ bundle over a manifold with corners. Every $\cF_{(\bn,i)}$ gives rise to a contractible family $\cF_{(\bn,i)}^{0} \coloneqq \partial L_0 \cap \cF$ of Reeb chords and these are all Reeb chords of $\partial{L}_0$. For families of Type I, $\cF_{(\bn,i)}^{0} \cong I$, an interval and we may assume this for families of Type IV as well. For families of Type II, $\cF_{(\bn,i)}^{0}$ is a point (non-degenerate chord). Finally, we have an $I^2$ in the Type III case and $D^2$ in the Type V case.

\begin{definition} \label{df:Liouvadmis}
Let $\Yhat_R$ denote the Liouville completion of $Y_R$ and
let $\hat{\theta'}_\epsilon$ denote the standard extension
of the Liouville form to the completion.
We say that an almost complex structure on $\Yhat_R$
is Liouville-admissible in a neighborhood of a level set $r_Y=c$
if there exists a positive real number $\delta$
such that
$
\hat{\theta'}_\epsilon \circ J= dr_Y
$
holds in the region defined by
$
 c-\delta < r_Y < c+\delta.
$
\end{definition}

It is these almost complex structures
which are used in \cite{MR2602848}.
We will make use of almost complex structure
which combine \pref{df:Liouvadmis} and \ref{sc:vaJ1}.

\begin{definition}
We let $J(Y_R,\Ebar)$ denote
the set of admissible almost complex structures on $\Ybar$
which are Liouville-admissible
in a neighborhood of the hypersurface where $r_Y=2$.
\end{definition}

\begin{lemma}
In the above situation,
the space $J(Y_R,\Ebar)$ is non-empty.
\end{lemma}

\begin{proof}
By \pref{lm:flow},
the Liouville vector fields and Reeb fields are tangent
to the two components of $\Ebar$.
It suffices to check that the almost complex structure
on the contact distributions can be chosen
so as to preserve both components of $\Ebar$.
This is a routine calculation left to the reader.
\end{proof}

\subsection{Floer theory for the Liouville domain}
 \label{sc:ASFloertheory}

The Lagrangian $L_0$ satisfies $\theta_\epsilon'|L=0$ everywhere on $Y$
and is in particular Legendrian at infinity for the Liouville structure.
In this subsection,
we recall in more detail the Liouville-admissible families of Hamiltonians
which are relevant to Abouzaid and Seidel's theory for Liouville domains.
We then recall a variant of their construction which uses Hamiltonians
which are constant outside of a compact region and
review the equivalence of the two approaches.
This is all completely standard material
in 
symplectic cohomology,
and it is easily adapted to the Lagrangian setting.

Choose a generic family of Liouville-admissible Hamiltonians $H_m$
of slope $\lambda_m$.
For the remainder of this article,
we assume that all Liouville-admissible Hamiltonians satisfy
\begin{align} \label{eq:adHams}
\begin{cases}
 \text{$H_m(x)=h_m(r_Y)$ for $r_Y \ge 1$}, \\
 \text{$h_m''(r) \ge 0$ for $r \ge 1$, and}, \\
 \text{$H_m(x)$ is $C^2$-small on $Y_R$}.
\end{cases}
\end{align}

For any Liouville-admissible Hamiltonian $H_m$,
notice that the flow of $H_m$ preserves the neighborhood
$\pi_{\NCx}^{-1}(U_{Z})$
along the cylindrical end, and
we may assume this is true everywhere.
Furthermore,
we can assume that it preserves $\pi_{\NCx}^{-1}(Z)$.
We first compute the wrapped Floer cohomology
\begin{align}
 \HW^*(L_0) \coloneqq
  \varinjlim_{m} \HF^*(L_0;H_{m})
\end{align}
as a vector space.
Since the flow of $H_{m}$ preserves the neighborhood $\pi_{\NCx}^{-1}(U_{Z})$,
all Hamiltonian chords of $\scrX (L_0,H_{m})$ lie in
$Y \setminus \pi_{\NCx}^{-1}(U_{Z})$.
We have seen in the discussion preceeding \pref{df:Liouvadmis}, that for each homotopy class $(\bn,i)$, there is a contractible submanifold with corners of Hamiltonian chords for $L_0$
(diffeomorphic to either a point, an interval $I$, or $I^2$).
For each of these submanifolds with corners,
we may perform an arbitrarily small Hamiltonian perturbation
supported in a neighborhood of these orbits,
so that there is a unique chord of index 0. We may therefore label generators by $\theta^{m}_{\bn,i}$,
which corresponds to the unique chord
of period less than or equal to $\lambda_m$
(for the perturbed Hamiltonian)
whose homotopy class in $\NCx \times \bCx$ is given by $(\bn,i)$.
As $\lambda_m\to \infty$,
all homotopy classes are eventually realized.

\begin{lemma} \label{lm:AScontinuation}
If the Hamiltonian chord
corresponding to $\theta^{m}_{\bn,i}$
has period $\lambda_0$,
then for any $\lambda_m' \gg \lambda_m \gg \lambda_0$,
the continuation map
$
 \HF^*(L_0;H_m) \to \HF^*(L_0;H_{m'})
$
sends
$
 \theta^{m}_{\bn,i} \to \theta^{m'}_{\bn,i}.
$
\end{lemma}

\begin{proof}
Because $d=0$ for grading reasons,
we may use any Hamiltonians of slope $\lambda_m$ and $\lambda_m'$
to compute this continuation map.
We again approximate by quadratic Hamiltonians,
e.g. by a Hamiltonian which is
\begin{itemize}
\item
$r_Y^2$ for $r_Y<a_{\lambda_{m}}$
\item
$\lambda' r_Y$ for $r_Y \geq a_{\lambda_{m}} +d_{\lambda_{m}}.$
\end{itemize}
Then the continuation maps are the identity on the Reeb chords
which arise in the quadratic part.
\end{proof}

\begin{corollary}
$\HW^*(L_0)$ is freely generated by
the images $\theta_{\bn,i}$
of $\theta^{m}_{\bn,i}$
for sufficiently high $m$.
\end{corollary}

Assume for the rest of this section that all Liouville-admissible Hamiltonians satisfy \eqref{eq:controlLiouville}. We now turn to the Hamiltonians which are constant outside of a compact set following \cite{MR2497314}. Choose a generic almost complex structure which is Liouville-admissible in a neighborhood of $r_Y=2$. For any $\lambda_m$, define $\mu(\lambda_m)>0$ to be some constant smaller than the distance between $\lambda_m$ and the action spectrum.
Define
$
 A(\lambda_m)\coloneqq3\lambda_m/\mu(\lambda_m).
$
By choosing $\mu(\lambda_m)$ small enough,
we may assume that $A(\lambda_m) \geq 4$.
We let $H_m^{c}$ be a function
satisfying the following conditions.
For $r_Y \geq 1$, $H_m^{c}$ depends only on $r_Y$ and is a (smoothing of a) function which satisfies
\begin{itemize}
\item for $r_Y \leq A-1$, $H_m^{c}=H_{m}$,
\item for $r_Y>A$ we assume that $H_m^{c}=\lambda_m(A-1)$
\item for $r_Y \geq A-1$, $(H_m^{c})'' \leq 0$.
\end{itemize}
Recall that the Floer differential increases the action.
We therefore have a subcomplex
$
 \CF^*(L_0; H_m^{c})_{(0,\infty)}
$
of generators whose action is greater than zero.
Define
\begin{align}
\CF^*(L_0; H_m^{c})_{p}
 \coloneqq \CF^*(L_0;H_m^{c})/\CF^*(L_0; H_m^{c})_{(0,\infty)}.
\end{align}
We may choose $H_{m}$ so that all time one chords $x$ satisfy
\begin{align}
 A_{H_{m}}(x) \leq 0.
\end{align}
An elementary calculation shows that the action of any time one chord
when $r_Y \geq 2$ all satisfy $A_{H_m^{c}}(x) \geq \lambda_m$.
It therefore follows that:

\begin{lemma}
 We have an isomorphism of cochain-complexes
\begin{align}
 \CF^*(L_0; H_{m}^{c})_{p} \cong \CF^*(L_0;H_m).
\end{align}
In the limit,
we obtain an isomorphism
\begin{align}
 \HW^*(L_0)  \cong \varinjlim_{m} \HF^*(L_0; H_m^{c})_{p}
\end{align}
of vector spaces.
\end{lemma}

\begin{proof}
The above discussion gives the identification on generators,
so it only remains to give the identification of differentials.
Since all generators of both complexes lie in the region
where $r_Y \leq 2$,
this follows from the ``no-escape lemma''
in \cite[Lemma 7.2]{MR2602848}.
This lemma also applies to continuation solution
to the result passes to the limit as claimed.
\end{proof}

\subsection{The main comparison theorem}

This section compares the ring structure on the two flavours
of wrapped Floer cohomology.
For this section,
we choose all of our almost complex structures to lie in $J(Y_R,\Ebar)$.
As we will need to vary $\mu_0$ in this section,
for clarity we use the notation $Y_{R,\mu_0}$
to indicate the dependence on $\mu_0$.
Notice that for $\mu_0'<\mu_0''$ the manifolds $Y_{R,\mu_0'}$
are Liouville isomorphic in a way which preserves $L_0$.
As a result, there are continuation isomorphisms
\begin{align}
 \HW^*_{Y_{R,\mu_0'}}(L_0) \to \HW^*_{Y_{R,\mu_0''}}(L_0)
\end{align}
between the wrapped Floer cohomologies.
In this case,
the continuation maps are given
by interpolating between $\mu_0'$ and $\mu_0''$
so that for every $s$ in the continuation family,
the Hamiltonian preserves the divisor $E$.
It follows that continuation map solutions have intersection number zero
with the divisor $E$ and thus the isomorphism
respects the free homotopy class of the chord
in $Y \setminus \pi_{\NCx}^{-1}(U_{Z})$.
We finally are in a position to state and prove
our main comparison theorem.

\begin{theorem} \label{thm: isofloer}
There is an isomorphism
\begin{align}
 \HW^*(L_0) \cong \HWad (L_0)
\end{align}
of $\bC$-algebras.
\end{theorem}

\begin{proof}
This is a modification of the proof of \cite[Theorem 5.5]{MR2497314}.
In this proof, we will have to discuss both fibered admissible and Liouville-admissible Hamiltonians. For this proof we will adopt the convention that fibered admissible Hamiltonians will be denoted by $G$  and Liouville-admissible Hamiltonians are denoted by $H$. We may consider a family of fibered admissible Hamiltonians
$G_{m}$
which has slope $\lambda_m$ in both directions,
which vanish on $Y_{R,\mu_0}$ and
such that for any time-one chord $x$, $A_{G_{m}}(x) \leq 0$.

Next, by shrinking $\mu_0$ to $\mu_{0,m}$,
we may ensure that $c_{\mu_{0,m}}$ above is arbitrarily large.
Thus, for any slope $\lambda'_m$ and for sufficiently small $\mu_{0,m}$,
the Hamiltonian $H_{{m}}^{c}$ makes sense
as a Hamiltonian on $Y$.
Given $H_{m}^{c}$,
we will construct a fibered admissible Hamiltonian $G_{h,m}$
which for $r<A$ agrees with $H_m^{c}$ and
for which all chords in $r_Y \geq 2$ have very positive action.
To do this,
we consider a shifted version of a standard fibered admissible Hamiltonian
$G_A$ which vanishes when
\begin{align}
 \lb C \abs{\mu-\epsilon} \le A + 1 \rb
  \cup \lb \rho \le B \cdot A + 1 \rb
\end{align}
We also require that whenever $\rho-B\cdot A \geq 1$ or $C \abs{\mu-\epsilon}-A \geq 1$, $G_A$ takes the form:
\begin{align}
 v(C \abs{\mu-\epsilon}-A)+ h(\rho-B\cdot A)
\end{align}
where $h$ and $v$ are suitable base-admissible and
admissible vertical Hamiltonians.
Finally, we require that $v(C|\mu-\epsilon|-A)$ and $h(\rho-B\cdot A)$ are linear
of slope $\sqrt{\lambda'_m}$
whenever $\rho-B\cdot A \geq 2$ or $C|\mu-\epsilon|-A \geq 2$.
It follows as in McLean Lemma 5.6 that the chords of the function $G_A$
 have action have action $A_{G_{A}}(x) \geq -\sqrt{\lambda'_m}(A+B\cdot A)$. In fact, the argument here is even easier because the chords live
over the part of the manifold where the symplectic fibration is trivial.
Now the function $G_{h,m}=G_A+H_m^{c}$ is fibered admissible. If follows that for any chord which lies in the region where $r_Y \geq 2$
\begin{align}
 A_{G_{h,m}}(x) \geq \lambda'_m(A-1) -\sqrt{\lambda'_m}(A+B\cdot A)>0.
\end{align}

Furthermore, it is a consequence of \pref{cr:estimate} that
for any $\lambda_m$, we may choose $\lambda'_m$, $\mu_{0,m}$ and
$G_{h,m}$ such that $G_{h,m}>G_{m}$ and
the slopes in both direction of $G_{h,m}$ are bigger than
$G_{m}$.
Next, we can construct fibered admissible Hamiltonians
$G'_{m}$
which have slope $\lambda''_m$ in both directions,
which vanish on $Y_{R,\mu_{0,m}}$ and
such the the action of all chords is $\leq 0$.
We assume that $\lambda''_m \gg \lambda'_m$.
We therefore have,
after choosing suitable functions $\ell^{\sharp}$,
continuation maps
\begin{align}
 \HFad^*(L_0;G_{m})
  \to \HFad^*(L_0;G_{h,m})_p
  \to \HFad^*(L_0;G'_{m}).
\end{align}
In the limit,
the composition of these two maps becomes an isomorphism and,
as a result,
the second map is an isomorphism as well.
Since these maps are ring maps,
this concludes the proof.
\end{proof}

\section{Homological Mirror Symmetry for McKay quivers}
 \label{sc:hms}

\subsection{Computation of $\SH^0(Y_R)$}

Fix a Liouville domain $\Xin$ and
a Liouville-admissible family of Hamiltonians $\lb H_m \rb_{m=1}^\infty$
on the Liouville completion $X$
of slope $\lambda_m$ with $\lambda_m \to \infty$
satisfying \pref{eq:adHams} and for simplicity also  \pref{eq:controlLiouville}.
A fixed point of a time-one flow
of the Hamiltonian vector field $X_H$
associated with a Hamiltonian function $H$
is called a \emph{time-one Hamiltonian orbit} of $H$.
The set of time-one Hamiltonian orbits
of $H$ will be denoted by $\scrX(X; H)$.
We may choose a small $\bS^1$-dependent perturbation
$
 H_m^\pert \colon \bS^1 \times X \to \bR
$
supported in a neighborhood of the union of all time-one Hamiltonian orbits of $H_m$,
so that any time-one Hamiltonian orbit
of the perturbed Hamiltonian
$
 H_m+ H_m^\pert,
$
is non-degenerate.
The Floer cochain complex is defined by
\begin{align} \label{eq:symp_coh_cpx}
 \CF^*(X;H_m+ H_m^\pert)
  \coloneqq \bigoplus_{x \in \scrX( X; H_m+ H_m^\pert)} \abs{\frako_x},
\end{align}
where
$\abs{\frako_x}$ is the orientation line
on the real vector space $\frako_x$
of rank one associated with
$x \in \scrX( X; H_m+ H_m^\pert)$.
The degree of $x \in \scrX( X; H_m+ H_m^\pert)$ is
the dimension of $X$ minus the Conley--Zehnder index of $x$;
\begin{align}
 \deg x = \dim X - \operatorname{CZ}(x).
\end{align}
The differential $\partial$ on the Floer complex
\pref{eq:symp_coh_cpx}
is given by counting solutions to Floer's equation
\begin{align}
 (d u - X_{H_m+ H_m^\pert} \otimes dt)^{0,1} = 0
\end{align}
for maps
$
 u \colon \bR \times \bS^1 \to X
$
from the cylinder,
with respect to an $\bS^1$-dependent almost complex structure $J$,
modulo translation in the $\bR$-direction.
A monotone homotopy from $H_m+ H_m^\pert$
to $H_{m+1}+ H_{m+1}^\pert$
gives the continuation map,
and the symplectic cohomology is defined as the colimit
\begin{align}
 \SH^*(X)
  \coloneqq \varinjlim_m \HF^*(X,H_m+ H_m^\pert).
\end{align}

For any conical spin Lagrangian $L$,
there is the \emph{closed-open map}
\begin{align} \label{eq:COmap}
 \CO \colon \SH^*(\Xin) \to \HW^*(L)
\end{align}
from the symplectic cohomology
to the wrapped Floer cohomology
(see e.g., \cite[Section 5]{MR2737980}).
This is defined by counting solutions to Floer's equation,
from a disk $\Sigma$
with one marked point on the boundary
(with a negative strip-like end)
and one marked point in the interior
(with a positive cylindrical end)
to $X$,
which asymptotes to a time-one Hamiltonian chord
in the negative strip-like end
and a time-one Hamiltonian orbit
in the positive cylindrical end.

The goal of this section is prove the following theorem:

\begin{theorem} \label{th:OCisomorphism}
The closed-open map
\begin{align}
 \CO: \SH^0(Y_R) \to \HW^0(L_0)
\end{align}
is an isomorphism.
\end{theorem}

In view of \pref{eq:adHams},
time-one Hamiltonian orbits on the cylindrical end are in bijection
with Reeb orbits of period less than $\lambda_m$.
We must begin by dealing with the fact that our Reeb flow is degenerate.
For each family $\cF_{(\bn,i)}$ of time-one Hamiltonian orbits
corresponding to the family $\cFbar_{(\bn, i)}$ of Reeb orbits,
we choose an isolating neighborhood $U_{(\bn,i)} \subset \Yhat_R$,
which contains only the orbit set $\cF_{(\bn,i)}$ and no others.
If the set $\cF_{(\bn,i)}$ occurs at a level $r_Y=r_0$,
then we may assume that the neighborhood
is the product
\begin{align}
 U_{(\bn,i)} = (r_0-\delta,r_0+\delta) \times \Ubar_{(\bn,i)},
\end{align}
of a small interval and a small open set
$
 \Ubar_{(\bn,i)} \subset \partial Y_R
$
containing $\cF_{(\bn,i)}$
and invariant under the natural local $\bS^1$-action
which extends the Reeb flow action on $\cFbar_{(\bn,i)}$.
Let
\begin{align}
 \mu_{(\bn,i)} \colon U_{(\bn,i)} \to \bR
\end{align}
denote the local moment map
for this $\bS^1$-action.
For definiteness,
let us fix choices of these open sets $\Ubar_{(\bn,i)}$: 
\begin{itemize}
 \item
Recall from \pref{sc:Reeb_orbits}
that
the family $\cFbar_{(\bn,0)}$
for regular $\bn$ is a $T^3$-fibration
over the product $(\text{point}) \times I$
of a point in $N_\bR$
and a closed interval in $\bR$
under the map $\pi_{\Bhat}$
given in \pref{eq:pi_Bhat}.
We choose the set $\Ubar_{(\bn,i)}$
to be the preimage by $\pi_{\Bhat}|_{\partial T_R}$
of the product of a small ball of the point in $N_\bR$
and an open interval in $\bR$ which is slightly larger than $I$.
 \item
When $\bn$ is regular and $i \ne 0$,
the family $\cFbar_{(\bn,i)}$ is the fiber of $\pi_{\Bhat}$
over a point in $N_\bR \times \bR$, and
we can choose $\Ubar_{(\bn,i)}$ to be the preimage
of a ball around this point.
 \item
When $\bn$ is not regular,
we choose the tubular neighborhood
to be where  $ \mu$ lies in $[\mu_{0,\bn,i}-\delta_{0,\bn,i}',\mu_0+\delta_{1,\bn,i}']$
and $|r_{\alpha-\beta}| \leq r_{\alpha-\beta,0}$ for appropriate intervals.
 \item
For $\cFbar_{(0,i)}$, we choose the neighborhood $\Ubar_{(\bn,i)}$
to be defined by $\rho < R-\delta'_R$ for $\delta'_R$ slightly smaller than $\delta_R$. 
\end{itemize}
We let $W^T_{(\bn,i)}$ be an open set in $U_{(\bn,i)}$
containing $\cF^0_{(\bn,i)}$ and
contained in the region where $\Yhat_R$ agrees with $\hat{T}_R$.

\begin{lemma} \label{lm:perturbations}
There is a function
$
 h_{t,(\bn,i)} \colon \bS^1 \times \Yhat_R \to \bR
$
supported in $U_{(\bn,i)}$, such that:
\begin{itemize}
 \item All Hamiltonian orbits of $H_{t,m, \epsilon}=H_m+\epsilon h_{t,(\bn,i)}$ inside of $U_{(\bn,i)}$ for $\epsilon$ sufficiently small are in bijection with the critical points $\alpha_j$ of an outward pointing Morse function
$
 \htilde \colon \cF_{(\bn,i)} \to \bR.
$
 \item
For $\epsilon$ sufficiently small, the gradings of the critical points coincide with the Morse-index of $\alpha_j$.
 \item
There is a unique orbit of index zero $\alpha_{(\bn,i)}$ which lies in $W^T_{(\bn,i)}$. 
\end{itemize}
\end{lemma}

\begin{proof} To do this, we apply elementary Morse-Bott theory for Reeb flows which generate $\bS^1$-actions (for a nice review with references to the original papers, we recommend \cite{MR3483060}), taking care of the fact that our Reeb orbits come in families which are manifolds with boundary. A standard trick allows us to ``unwrap" the Hamiltonian flow and thereby allows us to work in an equivalent Hamiltonian system where all of the orbits along $\cF_{(\bn,i)}$ of the flow are constant (and these are still the only orbits in $U_{(\bn,i)}$). We temporarily work in this equivalent Hamiltonian system.

  The sets $\Ubar_{(\bn,i)}$ naturally project to certain "thickenings" $\cF_{(\bn,i)}^t$ of $\cF_{(\bn,i)}$ of the same dimension. When the set $\cF_{(\bn,i)}$ are codimension zero in $\partial Y_R$, these are just $\Ubar_{(\bn,i)}$ and in the case when $\cF_{(\bn,i)}=T^3$, then $\cF_{(\bn,i)}^t=\cF_{(\bn,i)}$. In the case, when $\bn$ is regular, $\cF_{(\bn,0)}^{t}$ is the locus where $\mu \in [\mu_{\bn,0}-\delta_{0,\bn}',\mu_{\bn,0}+\delta_{1,\bn}']$ and $r_i=r_{i,\bn}$. Finally in the case when $\bn$ is not regular and $i$ is regular, $\cF_{(\bn,i)}^t$ is the locus where $\mu=\mu_0$ and $|r_{\alpha-\beta}| \leq r_{\alpha-\beta,0}$. Choose an outward pointing Morse-function $\tilde{h}: \cF_{(\bn,i)}^t  \to \bR$ such that all critical points lie in the interior of $\cF_{(\bn,i)}$. When $\cF_{(\bn,i)}^t$ is a manifold with boundary we postulate that on $\cF_{(\bn,i)}^t \setminus \cF_{(\bn,i)}$ near a boundary component cut out by an equation $f=c$ (either $\mu=c, r_{\alpha-\beta}=c$, or $\rho=c$) that $\tilde{h}=\tilde{h}(f)$. In the case when $\bn$ is not regular and $\cF_{(\bn,0)}^t$ is a manifold with corners we similarly require that near a corner $\tilde{h}=\tilde{h}(r_{\alpha-\beta},\mu)$ along an open smooth piece of the boundary that $h=h(r_{\alpha-\beta})$ or $h=h(\mu)$.

 Denote by $\pi_{\bn,i}: U_{(\bn,i)} \to \cF_{(\bn,i)}^t$. Choose a cutoff function $ \psi:\Yhat_R \to \bR$ supported in $U_{(\bn,i)}$ and such that $\psi=1$ on some smaller open set $V_{(\bn,i)}$ containing $\cF_{(\bn,i)}$. Finally consider the function $\psi\pi_{\bn,i}^*(\tilde{h})$. We take this function and translate it back to the orginal Hamiltonian system to obtain $h_{t,(\bn,i)}$ (Undoing the unwrapping process makes this function time dependent). The first bullet and second bullets now follow from standard Morse-Bott theory together with the fact that our Hamiltonian perturbations are chosen so that new orbits cannot appear near the boundary and hence all critical points occur when the flow is transversely non-degenerate. Finally the last point can be achieved by choosing $\tilde{h}$ to have a unique critical point of index $0$ which furthermore lies in $W^T_{(\bn,i)}$.
\end{proof} 

\begin{remark} \label{rm:eqperturb}
There is a variant construction of a perturbation $h_{t,(\bn,i)}$
satisfying the hypotheses of the lemma which proceeds in two steps
but allows one to make partial computations of the BV operator in symplectic cohomology.
Namely, one may without unwrapping consider a perturbation function
$\tilde{h}_{eq}$ on $\cF_{\bn,i}^t$
which is the pull-back of a Morse-function on $\cF_{\bn,i}^t/\bS^1$
satisfying the same properties near the boundary strata as $\tilde{h}_{eq}$ did above.
Extend this function to the neighborhoods $U_{\bn,i}$ by the formula $\psi\pi_{\bn,i}^*(\tilde{h}_{eq})$.
Consider the Hamiltonians
\begin{align}
 H_{m,eq}=H_m+\epsilon\psi\pi_{\bn,i}^*(\tilde{h}_{eq}).
\end{align}
For small $\epsilon$, we obtain a transversely non-degenerate $\bS^1$ family of Reeb orbits for every crticial point on $\cF_{\bn,i}^t/\bS^1$. Finally, we use a small time dependent Hamiltonian perturbation in a neighborhood of the orbits to obtain non-degeneracy. This construction can be done so that the third bullet point of the lemma is satisfied.   \end{remark}

We set
\begin{align}
 H^Y_{m, \pert}
  \coloneqq \sum_{(\bn,i), \cF_{\bn,i} \in \scrX(X,H_m)} h_{t,(bn,i)}. 
\end{align}
For certain purposes,
it makes sense to view
\begin{align}
 H^Y_{m,\pert}
  = H^\reg_{m,\pert} + H^\irr_{m,\pert},
\end{align}
where the first term is the sum over regular $(\bn,i)$ and
the other term is the sum over irregular $(\bn,i)$.
We set
\begin{align}
 H^Y_m \coloneqq H_m + \epsilon H^Y_{m,\pert}.
\end{align}
Similarly
we can take a small (time-independent) perturbation of
\begin{align}
 H^L_m \coloneqq H_m + \epsilon H^L_{m, \pert}
\end{align}
(the perturbation can be chosen to be supported in $W^T_{(\bn,i)})$,
so that there is a Reeb chord of $H^L_m$
for every suitable component of $\cF^0_{\bn,i}$. 

\begin{lemma} \label{lm:injective}
The map
\begin{align}
 \CO \colon \SH^0(Y_R) \to \HW^*(L_0)
\end{align}
is injective.
\end{lemma}

\begin{proof}
Observe that an element of the kernel $y$ must lie in some finite $\HF^*(\Yhat_R,H_{t,m,\epsilon})$. Because the maps $\HF^*(L_0,H_{m,\epsilon}) \to \HW^*(L_0)$ are injective, the map must further lie in the kernel of a map
\begin{align}
 \CO_m \colon \HF^*(\Yhat_R,H_{t,m,\epsilon}) \to \HF^*(L_0,H_{m,\epsilon})
\end{align}
This map depends only on the slope $\lambda_m$ at infinity and is hence independent of $\epsilon$ on homology. Let $\frako_{\alpha_{(\bn,i)}}$ denote the orientation space in $\CF^*(L_0,H_{m,\epsilon})$ corresponding to orbit $\alpha_{(\bn,i)}$. Choosing arbitrarily a generator for this space and giving it the same name by abuse of notation, the statement will follow from the fact that for $\epsilon$ small, on the cochain level the coefficient of $\theta_{(\bn,i)}$ in $\CO_m (\frako_{\alpha_{(\bn,i)}})$ is $\pm 1$. 

To see this, a Gromov compactness argument
(\cite[proof of Lemma 2.3]{MR2994824}) shows that for sufficiently small $\epsilon$ any curve contributing to this coefficient will necessarily lie in $W^T_{(\bn,i)}$. Observe that we may view the Hamiltonians $H_m$ as being Hamiltonians on $\hat{T}_R$ simply by replacing the Liouville coordinate $r_Y$ with $r_T$. We denote the resulting Hamiltonians by $H_m^{T}=H_m(r_T)$ for clarity. Similarly, we may construct perturbations $H_{m,\epsilon}^{T}$ and $H_{t,m,\epsilon}^T$ which agree over $W^T_{(\bn,i)}$ with the previously constructed perturbations over $W^{T}_{(\bn,i)}$ and so that for sufficiently small $\epsilon$ the orbit corresponding under this equivalence to $\alpha_{(\bn,i)}$ is a unique Hamiltonian orbit of degree 0 in homotopy class $(\bn,i)$. Again, for sufficiently small $\epsilon$, all curves contributing to the
\begin{align}
 \CO_m^{T}: \HF^*(\hat{T}_R,H_{t,m,\epsilon}^T) \to \HF^*(L_0,H_{m,\epsilon}^T)
\end{align}
lie in $W^T_{(\bn,i)}$. This map must be $\pm 1$ along the diagonal by the well-known calculations of $\SH^*(T_R)$ and $\HW_{T_{R}}^*(L_0)$. The result follows. 
\end{proof}

For any homotopy class $\bn$,
let $\SH^0(Y_R)_{\bn}$ and $\HW^*(L_0)_{\bn}$
denote the subvector space generated by orbits in class $\bn$.
We have a direct sum decomposition
\begin{align}
 \SH^0(Y_R) \coloneqq \bigoplus_{\bn} \SH^0(Y_R)_{\bn}
\end{align}
and similarly for wrapped Floer theory.
The map $\CO$ respects these direct sum decompositions.

\begin{lemma} \label{lm:regular}
The map
\begin{align}
 \CO_{\bn} \colon \SH^0(Y_R)_{\bn} \to \HW^*(L_0)_{\bn}
\end{align}
is an isomorphism for $\bn \neq 0$ regular.
Furthermore for any $\bn$ regular, let $\alpha_{(\bn,i)}$ be the Hamiltonian orbit on the cylindrical end of $H_{t,m,\epsilon}$ (as usual $\epsilon$ is sufficiently small) of index 0. Choose a generator, $\alpha^0_{(\bn,i)}$, of the corresponding orientation line. For $H_{\pert,n}$ sufficiently close to zero, $\CO(\alpha^0_{(\bn,i)})=\pm \theta_{(\bn,i)}$.
\end{lemma}

\begin{proof}
We may assume that the Hamiltonians $H_m$ are actually unperturbed near $\pi_{\NCx}^{-1}(Z)$ (i.e. $H_{\pert,n}$ is sent to zero) so that the contribution to the differential is local. The result then follows from a similar argument to Lemma \ref{lm:injective}.
\end{proof}

\begin{remark} \label{rm:multsurj}
Using the fact that
the open-closed map is a ring homomorphism,
\pref{lm:regular} implies that
many other elements are hit by the map $\CO$ as well.
To be more precise,
let $\alpha \in SH^0$ be an element such that
$\CO(\alpha)=\theta_{\bn,i}$ and
let $\beta \in SH^0$ be an element such that
$\CO(\beta)=\theta_{\bn',i'}$ with both $\bn$ and $\bn'$ regular.
Then
\begin{align}
\CO(\alpha\cdot \beta)
 = \sum_{j=0}^{\ell_2(\bn,\bn')}
 \binom{\ell_2(\bn,\bn')}{j}
  \theta_{\bn+\bn',i+i'+j}
 =(1+\theta_{0,1})^{\ell_2(\bn,\bn')} \theta_{\bn+\bn',i+i'} 
\end{align}
is in the image.
By choosing $\bn$ and $\bn'$ appropriately,
this implies that for any $\bn''$,
the element $(1+\theta_{0,1})^{d(\bn'')} \theta_{\bn'',i}$ is
the image of $\CO$ for some $d(\bn'') \ge 1$ and all $i$.
\end{remark}

We now prove the following lemma,
which allows us to compute some differentials in symplectic cohomology and which uses a technique in \cite{1304.5298}.
For the remainder of this subsection,
we use the perturbations of Remark \ref{rm:eqperturb} on the cylindrical end.
We also assume that our data is chosen so that
the restriction $H_{t,m,\epsilon}|_{Y_{R}}$ is Morse and time-independent
with unique critical point of index 0.

\begin{lemma} \label{lm:closedelements}
Let $\cF_{\bn,i}$ be the family of Hamiltonian orbits of $H_m$ of highest action in the homotopy class $\bn$.
Let $\alpha_{(\bn,i)}$ be the Hamiltonian orbit on the cylindrical end of $H_{t,m,\epsilon}$
(as usual $\epsilon$ is sufficiently small)
of index 0.
Choose a generator, $\alpha^0_{(\bn,i)}$,
of the corresponding orientation line.
Then
\begin{align}
 \partial(\alpha^0_{(\bn,i)})=0 \in \CF^*(\Yhat_R, H_{t,m,\epsilon}).
\end{align}
\end{lemma}

\begin{proof}
We focus on the more difficult case where $\bn=0$
and we want to prove that the corresponding $\alpha^0_{(0,i)}, i \neq 0$ is closed.
Choose a generator of the orientation line
corresponding to the Morse critical point in the interior which we will denote by $\alpha^0_{(0,0)}$.

For suitable such $H_{t,m,\epsilon}$ with $\epsilon$ sufficiently small and complex structures $J_t$,
we use the Morse-Floer correspondence to identify trajectories between orbits contained in $U_{(0,i)}$ with differentials of the Morse-complex of a time independent $\tilde{h}$ for the restriction of an almost-Kahler metric corresponding to a time-independent complex structure $J_0$. The pair $(h_t,J_t)$ and $(\tilde{h},J_0)$ are again related by unwrapping. Again, the only difference between this case and the classical case is that we must ensure that Floer trajectories do not escape to the boundary of $\cF_{\bn,i}$, which is easily arranged.
The upshot is that we have
\begin{align}
 \partial(\alpha^0_{(0,i)})=\beta^{1}_{(0,0)}
\end{align}
where $\beta^1_{(0,0)}$ is a sum of generators corresponding to critical points in the interior of $Y_R$. 
Let $\Delta$ denote the BV operator on symplectic cohomology.
Using the local analysis in \cite{1405.2084},
we can see that there is a cochain $\alpha^1_{(0,i)}$ in $\CF^1(\Yhat_R, H_{t,m,\epsilon})$ such that
\begin{align}
 \Delta(\alpha^1_{(0,i)})=\alpha^0_{(0,i)} + c \cdot \alpha^0_{(0,0)}
\end{align}
for some $c$ and such that
\begin{align}
 \partial(\alpha^1_{(0,i)}) = \beta^2_{(0,0)}
\end{align}
where $\beta^2_{(0,0)}$ is a sum of generators corresponding to critical points in the interior of $Y_R$.
Now $\Delta$ is a chain map,
so $\partial \circ \Delta=\Delta \circ \partial$.
The assumption that $H_{t,m,\epsilon}$ is time-independent when restricted to the interior implies that
$\Delta(\beta^2_{(0,0)})=0$ on the chain level.
Combining these statements implies the required result.
\end{proof}

We finally complete the proof of \pref{th:OCisomorphism}.

\begin{proof}[Proof of \pref{th:OCisomorphism}]
We can assume that
the contractible orbit of \pref{lm:closedelements} is $\alpha_{(0,\pm 1)}$.
Assume it is $\alpha_{(0,1)}$ as the other case is similar.
Then this implies that $\theta_{0 ,1}$ is in the image of $\CO$.
It proves that $\theta_{0,i}$ is in the image of $\CO$ for $i \geq 0$.
Combining this with \pref{rm:multsurj} implies that
the map $\CO_{[0]}$ is surjective.
Applying \pref{lm:closedelements},
this implies that for any $\bn$,
there is an $i_{0}$ for which $\theta_{\bn,i_{0}}$ is in the image of $\CO_{\bn}$. This combined with the surjectivity of $\CO_{[0]}$ implies that $\theta_{\bn,i}$ is in the image of $\CO_{\bn}$ for any $i$.  \end{proof}

\subsection{Generation for the mirror of the complement
of an anti-canonical divisor in $\bC^3$}

Let $Y$ be
the conic bundle defined by the equation 
\begin{align}
 1+w_1+w_2= uv.
\end{align}
The mirror in this case is the affine space
\begin{align}
 \Yv \coloneqq \Spec \lb \bC[x,y,z] \ld (xyz-1)^{-1} \rd \rb.
\end{align}
The identification between this ring and $H^0(\cO_\Yv)$ sends
$y$ to $\chi_{(0,-1),0}$,
$x$ to $\chi_{(-1,0),0}$ and
$z$ to $\chi_{(1,1),0}$.
The goal of this section is to prove the following theorem:

\begin{theorem} \label{th:generationL0}
There is an equivalence of categories
\begin{align}
 \psi \colon D^b \cW(Y_{R}) \to D^b \coh \Yv
\end{align}
sending $L_0$ to $\cO_{\Yv}$.
\end{theorem}

It suffices to prove that
the Lagrangian $L_0$ split-generates
the wrapped Fukaya category of $Y_R$.
Then it follows that $L_0$ generates $D^\pi \cW(Y_R)$
without the need to pass to direct summands.
The proof of this theorem works by checking a generation criterion
due to Abouzaid \cite{MR2737980}
(for an implementation with linear Hamiltonians,
see \cite{1201.5880v3}).
We will now describe it in the setting that we need.
Let $\HH_3(\HW^*(L_0))$ denote the third Hochschild homology of the wrapped Floer algebra.
This is of course isomorphic to $\HH_3(H^0(\cO_\Yv))$,
which in this case admits an explicit description via the HKR theorem as
\begin{align} \label{eq:HH3Yv}
 \operatorname{HKR} \colon \HH_3 \lb H^0 \lb \cO_\Yv \rb \rb
  \simto H^0 \lb \Omega^3_{\Yv} \rb.
\end{align}
Since $\Yv$ is
equipped with a natural choice of holomorphic volume form
\begin{align}
 \Omega_{\Yv} \coloneqq \frac{dx\wedge dy \wedge dz}{xyz-1},
\end{align}
the module
$H^0 \lb \Omega^3_{\Yv} \rb$ is free of rank one over $H^0 \lb \cO_\Yv \rb$
generated by $\Omega_{\Yv}$.

Assume that $\lambda_m=m$ for all $m$ and that \pref{eq:controlLiouville} holds.
Then parallel to \pref{eq:telescope},
the construction of the symplectic cohomology may be promoted
to a cochain complex $\SC^* \lb Y_R \rb$
whose cohomology is $\SH^* \lb Y_R \rb$.
Central to Abouzaid's criterion is the map of complexes
\begin{align}
 \scrOC \colon \CC_{3-*} (\CW^*(L_0,L_0)) \to \SC^*(Y_R)
\end{align}
from the Hochschild chain complex of the $A_\infty$-algebra.
The induced map on cohomology is denoted by
\begin{align} \label{eq:OCmap}
 \OC \colon \HH_3(\HW^*(L_0)) \to \SH^0(Y_R).
\end{align} 

\begin{theorem}[{\cite[Theorem 1.1]{MR2737980}}] \label{thm: generationcrit}
$L_0$ split-generates the wrapped Fukaya category
if the unit $[1]$ is in the image of \pref{eq:OCmap}.
\end{theorem}

In fact,
given our choice of data,
the wrapped Floer complex is concentrated in degree zero.
This allows for a considerable simplification
in the description of the map $\OC$
essentially because $\mu^{d,\bsp}=0$ for $d=1, |F|=0$ and $d>2$.
The underlying domains used for constructing this map will be Riemann surfaces $\Sigma_{4,1}$
with four positive strip like ends and one negative cylindrical end.
Let $\cRbar^{4,1}$ denote the moduli space of these structures. 

As before,
we consider perturbing functions of the form
$
 H^Y_\pert \colon \bS^1 \times \Yhat_R \to \bR
$
and
$
 H^L_\pert \colon \Yhat_R \to \bR.
$
Along the positive ends,
we consider Floer data
$
 H_i=H_{m_{i}}+\epsilon H^L_\pert
$
and setting $\mathbf{m}=\sum_i m_i$.
Along the negative cylindrical end,
we consider Floer data
$
 H_0=H_ {\mathbf{m}}+ \epsilon H^Y_\pert.
$
As usual,
we choose perturbation $K$ and admissible surface-dependent complex structure
compatible with the Floer data along the ends
over the moduli space and consistent with boundary strata.
Given chords $x_i$ and a Hamiltonian orbit $x_0$,
we let $\cRbar^{4,1} \lb \bx \rb$
denote the moduli space of solutions.
For generic data,
these moduli spaces have the expected dimension.
In particular,
for $x_i$ such that $\vdim(\cRbar^{4,1}(\bx))=0$,
every rigid solution $u$ gives rise to an isomorphism
\begin{align}
 \OC_u \colon \abs{\frako_{x_{4}}}
  \otimes \cdots \otimes \abs{\frako_{x_{1}}}
  \simto \abs{\frako_{x_{0}}}.
\end{align}
As usual,
the map $\OC$ is defined as the sum of $\OC_u$.

Since $\HW^*(L_0)$ is concentrated in degree zero and hence formal,
we may avoid the telescope complex and
realize the Hochschild chains as a direct limit
\begin{align}
 \CC_3(\HW^*(L_0))
  \cong \varinjlim_{m_{i}}  \HF^*(L_0,H_{m_{1}}+\epsilon H^{L}_{\pert})
   \otimes \cdots \otimes \HF^*(L_0,H_{m_{4}}+\epsilon H^{L}_{\pert})
\end{align}
In particular,
after writing any cycle $\Omega$ as a finite sum of simple tensors,
we obtain a cycle $\OC(\Omega) \in \SH^0(Y_R)$.
Standard arguments show that this is compatible with direct limits and
is independent of class $[\Omega] \in \HH_3(\HW^*(L_0))$.
This map coincides with the map in \pref{eq:OCmap}. 

Crucial for us is the property that the map $\OC$ is a module map over $\SH^0(Y_R)$,
where the module structure on the left hand side is via the map \pref{eq:COmap}.
Since $\HH_3(\HW(L_0))$ is a free $\SH^0(Y_R)$-module of rank one,
the image of \pref{eq:OCmap} is necessarily a principal ideal $(f) \subset \SH^0(Y_R)$. \pref{th:generationL0} therefore follows immediately from \pref{thm: generationcrit} if we can show that there are a set of non-zero elements $f_1, \cdots, f_n$ in the image of $\OC$ which have no common divisors (other than units). More precisely,
let $\alpha \in A$ be one of the three monomials
in the defining equation for $Y$ and
$\bn$ be a weight corresponding
to a monomial in the chamber $C_\alpha$.

\begin{definition}
Let
$
 \CC_*(\HW^*(L_0))_\bn
$
be the subcomplex of
$
 \CC_*(\HW^*(L_0))
$
of total weight $\bn$;
\begin{align} \label{eq:HHdecomp}
 \CC_d(\HW^*(L_0))_\bn
  \coloneqq \bigoplus_{\sum_{i=0}^{d} \bn_i=\bn}
   \HW^*(L_0)_{\bn_{0}} \otimes \cdots \HW^*(L_0)_{\bn_d}.
\end{align}
The subspace of
$
 \CC_d(\HW^*(L_0))_\bn
$
generated by
$
 \HW^*(L_0)_{\bn_{0}} \otimes \cdots \HW^*(L_0)_{\bn_d}
$
for $(\bn_0, \ldots, \bn_d) \in N^{d+1}$
satisfying
\begin{itemize}
 \item
all $\bn_i$ are in the chamber corresponding to $\alpha$, and
 \item
for any two neighboring $\bn_i, \bn_{i+1}$
with respect to the cyclic ordering,
$\bn_i +\bn_{i+1}$ lies in the chamber corresponding to $\alpha$.
\end{itemize} 
is denoted by
$
 \CC^{\alpha}_*(\HW^*(L_0))_\bn.
$
\end{definition} 

We have maps
\begin{align} \label{eq:classbn}
 \OC_{\bn} \colon \HH_3(\HW^*(L_0))_{\bn} \to \SH^0(Y_R,G_m)_{\bn}.
\end{align}
Our approach to making partial computations of the open closed map
is again via the relationship with fibration-admissible Floer theory.
We therefore temporarily assume that
when constructing our perturbations of Liouville-admissible Hamiltonians $H_m$,
we take all of the terms  $H^\irr_{m,\pert}$ to be zero. 
This is possible because we will only be making computations
in regular homotopy classes from now on.
Notice that in particular the map \pref{eq:classbn} can be defined with such Floer data.
Let $G$ be an admissible Hamiltonian in the sense of \pref{sc:aWFT}, and
$\bn$ be a regular homotopy class satisfying the following condition:

\begin{itemize} \label{it: homotopycondition}
 \item
There is no Hamiltonian orbit of $G$ in the homotopy class $\bn$
which intersects $\pi_{\NCx}^{-1}(U_Z)$.
\end{itemize} 

Then after making small time-dependent perturbations
$G_t=G+G_{\pert}$
in the neighborhood of the orbits in homotopy class $\bn$, we may define $\HFad^*(Y,G_t)_{\bn}$ analogously to \pref{sc:aWFT} (In fact, for the $G$ we need and since we never need to consider multiplicative structures on $\HF$, we can avoid altogether the use of the functions $\ell^{\sharp}$). Let $G_m$ now be an admissible family of such Hamiltonians and $\bn$ a homotopy class such that  
\ref{it: homotopycondition} is satisfied for all $m$. Then we may define the $\SH^*_{\ad}(Y,G_m)_{\bn}$ to be the direct limit of these Floer groups. 

Continue to fix a regular $\bn$.
The argument of \pref{thm: isofloer} shows that
we can find an admissible sequence of Hamiltonians $G_m$
together with an isomorphism
\begin{align}
 \phi_{\SH} \colon \SH^0(Y_R)_{\bn} \simto \SH^0_{\ad}(Y,G_m)_{\bn} 
\end{align}
for all such $\bn$.
For any
$
 \Omega_{\ad} \in \HH_3^{\alpha}(\HW_{\ad}(L_0))_{\bn},
$
we may define an element $\OC_{\ad,\bn}(\Omega_{\ad})$
just as in the case of the ordinary (i.e., not adapted) theory.
Moreover, given a cycle $\Omega \in \CC^{\alpha}(\HW(L_0))_{\bn}$, applying the isomorphisms of \pref{thm: isofloer} to each factor of the tensors defining the Hochschild complex gives rise to a cycle $\phi_{\CC}(\Omega) \in \CC^{\alpha}(\HW(L_0))_{\bn}$ (the use of $\CC^{\alpha}$ is again for simplicity so that we only multiply elements in the same chamber and hence can avoid using the functions $\ell^{\sharp}$ from \pref{sc:aWFT}). 

\begin{lemma}
For any such $\Omega$,
we have
\begin{align}
 \OC_{\ad,\bn}(\Omega_{\ad})=\phi_{\operatorname{SH}} \circ \OC(\Omega).
\end{align}
\end{lemma}

\begin{proof}
We may assume that $\CF^{-1}(Y,G_m)=0$ for all $m$.
The result then follows by a standard degeneration of domain argument.
\end{proof} 

We are now in a position to make the final partial computation of the map $\OC$.

\begin{lemma}
For some $i>0$,
the elements
$\alpha^0_{((-5,-5),-4i)}$,
$\alpha^0_{((5,0),-4i)}$, and
$\alpha^0_{((0,5), -4i)}$
are in the image of $\OC_{\ad}$,
where here we are identifying $\SH^0_{\ad}$ and $\SH^0$.
\end{lemma}

\begin{proof}
The conclusion is invariant under suitable deformations of the functions $G_m$.
We may therefore assume this sequence satisfies
$
 G_m=H_{\operatorname{ba},m}+ H_{\operatorname{v},m}.
$
We may assume that
$
 H_{\operatorname{v},m}=1/2\mu^2
$
for
$
 c_{m,0}-\delta \geq \mu \geq c_{m,1}
$
with $\delta>0$ and
$
 c_{m,1} \gg c_{m,0} \gg \epsilon
$
so that
$
 \mu=1/2|u|^2.
$
Assume that all of our complex structures respect
the standard product structure for all $\mu$
in an open neighborhood of $c_{m,0}-\delta$.
We can assume that for some $i,m>0$,
the chords $p_{(-1,-1),-i}$, $p_{(-1,-2),-i}$, $p_{(-2,-1),-i}$ all lie in the region
where
$
 c_{m,0}-\delta \geq \mu \geq c_{m,1}.
$
Consider the Hochschild cycle
given by the anti-symmetrization
\begin{align}
 \Omega_{\operatorname{HKR}}
  = \epsilon_3(p_{(-1,-1),-i} \otimes p_{(-1,-1),-i} \otimes p_{(-1,-2),-i} \otimes p_{(-2,-1),-i}).
\end{align}
Under the HKR map,
this corresponds to
\begin{align}
 \frac{x y}{(xyz-1)^i}
 &d \lb \frac{x y}{(xyz-1)^i} \rb
 \wedge d \lb \frac{x y^2}{(xyz-1)^i} \rb
 \wedge d \lb \frac{x^2 y}{(xyz-1)^i} \rb \\
 &= -i \frac{(xy)^5}{(xyz-1)^{4i}} \frac{dx\wedge dy \wedge dz}{xyz-1} \\
 &= -i p_{(-5,-5),-4i} \cdot \Omega_{\Yv}.
\end{align}
\begin{comment}
\begin{align}
 Q dQ\wedge d(xQ) \wedge d(yQ)
  &= Q^3dQ \wedge dx \wedge dy \\
  &= Q^3d(xy(xyz-1)^{-i})dx \wedge dy \\
  &= -i Q^3 (xy)^2(xyz-1)^{-i-1} dz\wedge dx \wedge dy \\
  &=-i (xy)^5 (xyz-1)^{-4i}\frac{dx\wedge dy \wedge dz}{xyz-1}
\end{align}
where $Q= xy(xyz-1)^{-i}$.
\end{comment}
By homotopy class restrictions,
the output of
$
 \Omega_{\operatorname{HKR}}
$
must be a non-trivial multiple of
$
 \alpha^0_{((-5,-5),-4i)}.
$
We assume that this orbit lies in the region where
$
 c_{m,0}-\delta \geq \mu \geq c_{m,1}
$
as well.
A barrier argument
(see e.g.~\cite[Appendix B]{1011.2478})
then shows that solutions for the map
$
 \OC_{\ad}(\Omega_{\operatorname{HKR}})
$
lie in the region where $\mu>c_{m,0}-\delta$
and hence must agree with the $\OC$ map
defined on $(\bCx)^3$.
The fact that $\alpha^0_{((-5,-5),-4i)}$ is in the image of $\OC_{\ad}$ follows directly from this.
The other cases are symmetric and hence the lemma is proved.
\end{proof} 

Now we return to considering the map $\OC$ and
we assume for the following proposition that
$\epsilon H_{\pert,n}$ is chosen sufficiently small
so that the conclusion of \pref{lm:regular} holds.

\begin{lemma}
Identify $\SH^0(Y_R)$ with $H^0(\cO_{\Yv})$
via the map $\CO$.
Then $\theta_{((-5,-5),0)}$, $\theta_{((5,0,)0)}$, $\theta_{((0,5),0)}$ are all
in the image of \pref{eq:OCmap}.
\end{lemma}

\begin{proof} The previous lemma implies that $\theta_{((-5,-5),-4i)}$, $\theta_{((5,0),-4i)}$, $\theta_{((0,5),-4i)}$ are in the imafge of $\OC$. Because $\OC$ is a module homomorphism, the result follows.  \end{proof}

Notice that under the isomorphism described at the beginning of this section, these elements correspond to the polynomials $(xy)^5, (yz)^5, (xz)^5$ and hence have no common divisors (other than units). This completes the proof of \pref{th:generationL0}.   

\subsection{Finite covers and applications} 

\pref{th:generationL0} has interesting applications to the symplectic topology of $Y$.
For example, we have the following:

\begin{proposition} \label{pr:embeddings}
Let $L$ be a compact oriented Maslov zero exact Lagrangian submanifold in $Y$.
Then $L$ has the cohomology group of a torus;
\begin{align}
 H^*(L;\bC) \cong H^*(T^3;\bC).
\end{align}
\end{proposition}

\begin{proof}
The mirror of $L$ is an object $\psi(L)$ of $D^b \coh(\Yv)$
satisfying
\begin{align}
 \Ext^0(\psi(L), \psi(L)) \cong \bC
\end{align}
and
\begin{align} \label{eq:ext-vanishing}
 \Ext^i(\psi(L),\psi(L))=0
  \quad \text{for any $i < 0$}.
\end{align}
Let $\frakm$ denote the annihilator of 
$\Ext^0(\psi(L),\psi(L))$ as a module over $\Gamma(\cO_{\Yv})$.
It is a maximal ideal of $\Gamma(\cO_{\Yv})$
since $\Ext^0(\psi(L),\psi(L)) \cong \bC$ is a field.
Let $y$ be the corresponding closed point of $\Yv$.
It follows from \cite[Lemma 4.8]{MR2198807}
that all cohomology sheaves $\cH^i(\psi(L))$
are $\cO_y$-modules.
As a consequence,
we have an isomorphism of all cohomology sheaves
$
 \cH^i(\psi(L)) \cong (\cO_y)^{\oplus n_i}
$
as $\cO_{\Yv}$-modules.

It suffices to prove that $\psi(L)$ has cohomology in a single degree,
since this implies that $\psi(L)$ is isomorphic to a sheaf,
which by the previous paragraph must be $\cO_y$.
Assume for contradiction that
\begin{align}
 l \coloneqq
  \max \lc i \in \bZ \relmid \cH^i(\psi(L)) \ne 0 \rc
\end{align}
is greater than
\begin{align}
 k \coloneqq
  \min \lc i \in \bZ \relmid \cH^i(\psi(L)) \ne 0 \rc,
\end{align}
and consider the cohomological spectral sequence
with second page
\begin{align}
 E_2^{p,q}=\bigoplus_i \Ext^{p}( \cH^i(\psi(L)), \cH^{i+q}(\psi(L)))
\end{align}
and converging to $\Ext^{p+q}(\psi(L), \psi(L))$.
Then one has
\begin{align}
 E_2^{0,k-l}
  \cong \operatorname{Hom}_{\cO_{y}}(\cO_y^{n_l}, \cO_y^{n_k})
  \ne 0,
\end{align}
which would survive the spectral sequence and
lead to a non-vanishing $\Ext^{k-l}$.
This contradicts \pref{eq:ext-vanishing},
and the conclusion follows.
\end{proof}

We conclude with the observation
that the generation criterion is easily seen to be compatible
with passage to finite covers
(this is a very easy special case of \cite[Corollary 1.4]{1201.5880v3}).
To be more precise,
let $N_0 \subset N$ be a finite index subgroup of the fundamental group of $Y$.
Denote the quotient group by $G=N/N_0$.
Let $Y_{N_0}$ be the corresponding finite cover and
$Y_{R,N_0}$ the corresponding Liouville domain.
There is a canonical lift of $L_0$ to this finite cover
and hence we may parameterize all possible lifts as $L_{g}$ with $g \in G$.
Observe that there is a natural action of $\Gv \coloneqq \Hom(G,U(1))$ on $\Yv$
coming from the SYZ construction.
In local coordinates,
these correspond to diagonal linear symmetries on $\Yv$
which preserve the holomorphic volume form.
As a result, we obtain a Calabi--Yau orbifold $\ld \Yv/\Gv \rd$.
Now we prove Corollaries
\ref{cr:hms_McKay}
and
\pref{cr:compact_embedding}:


\begin{proof}[Proof of \pref{cr:hms_McKay}]
The generation criterion shows that
the Lagrangians $L_{g}$ generated the wrapped Fukaya category of $Y_{R,N_0}$.
The objects  $\cO_{\Yv} \otimes V_g$ generate the category $D^b \coh(\Yv/\Gv)$.
The endomorphism algebras of these objects are easily seen
to agree with the crossed product algebra $\cO_{\Yv} \rtimes \Gv$.
\end{proof}

%

\begin{proof}[Proof of \pref{cr:compact_embedding}]
Given any Lagrangian sphere $S$,
consider the object $\psi_{N_{0}}(S)$ and suppose it is not supported at the origin. The object $\psi_{N_{0}}(S)$ viewed as a module over $\Gamma(\cO_{\Yv})^{\Gv}$ must then be supported at a single, regular, point. A variant of the argument from \pref{pr:embeddings} then shows that $\Ext^*(\psi_{N_{0}}(S), \psi_{N_{0}}(S))$ must be $H^*(T^3)$, which is impossible.
\end{proof}

\pref{pr:disjoin2} below is a slight strengthening
of \pref{pr:disjoin1}:

\begin{proposition} \label{pr:disjoin2}
Let $S_1, \cdots, S_r$ be a collection of Lagrangian spheres in $Y_{N_0}$
which are pairwise disjoint.
Then $r$ is less than or equal to the maximal dimension
of an isotropic subspace of the Grothendieck group of $\coh_0 \ld \Yv / \Gv \rd$
with respect to the Mukai pairing.
\end{proposition}

\begin{proof}
We follow the ideas of \cite[Section 15]{Seidel_CDST}.
Any one-dimensional subtorus acting on $\bC^3$
which acts with weight one
on the holomorphic volume form
induces a dilating $\bGm$-action
on the category $D^b \coh_0 \Yv / \Gv $
in the sense of \cite[Definition 15.4]{Seidel_CDST}.
The result then follows via the same technique
as \cite[Proposition 15.8]{Seidel_CDST},
replacing Hochschild homology with $K$-theory
as in \cite[Remark 15.16]{Seidel_CDST}.
\end{proof}

\section{Appendix} 

\subsection{A priori energy bounds}

%
The \emph{geometric energy}
and the \emph{topological energy}
of solutions to Floer's equation
\eqref{eq:Floer}
are respectively defined by
\begin{align}
 \Etop(u) \coloneqq \int_{\Sigma} u^* \omega - d(u^*K)
\end{align}
and
\begin{align}
 \Egeom(u) \coloneqq \frac{1}{2} \int_{\Sigma} \norm{du - X_{K}}.
\end{align}
Using $\omega(X, X) = 0$
and $d H = \iota_X \omega$, a standard calculation in local holomorphic coordinates
$z = s + \sqrt{-1} t$ on $\Sigma$ yields
\begin{align}
\label{eq:energy1}
 \Egeom(u)
  &= \Etop(u)+ \int_{\Sigma} \Omega_K
\end{align}
where the curvature $\Omega_K$ of a perturbation datum $K$ is the exterior derivative of $K$ in the $\Sigma$ direction.
\begin{definition}
We say that a perturbation datum $K$ is \emph{monotone} if
$
 \Omega_K \leq 0.
$
\end{definition}
In practice, the most typical monotone perturbation data are
of the form $H \otimes \gamma$
for $H \colon Y \to \bR^{\ge 0}$
and $\gamma$ a sub-closed one form on $\Sigma$.
Whenever the perturbation datum is chosen to be monotone, we have
\begin{align} \label{eq:energy_bound}
 \Egeom(u) \leq \Etop(u).
\end{align}
Note that the topological energy
is determined by the action
\begin{align}
 A_{H_{i}}(x) = \int_0^1 \big( - x^* \theta
  + H_i(x(t)) d t \big) + h(x(1)) - h(x(0))
\end{align}
as
\begin{align}
 \Etop(u)
  &= A_{H_{0}}(x_0) - \sum_{k=1}^d A_{H_{k}}(x_k),
\end{align}
where $h \colon L \to \bR$ is defined by
$\theta|_L = d h$.
Throughout this paper, all Floer data $K$ are chosen to be monotone
outside of a compact set in $\Sigma \times Y$.
This gives an a priori bound
for the geometric energy of a solution to Floer's equation \eqref{eq:Floer}
of the form
$
 \Egeom(u) \le \Etop(u)+C
$
for some constant $C$ which is independent of the given solution $u$.

\subsection{Maximum principle}

Let
$u \colon \Omega \to \bR^{\ge 0}$
be a smooth function
on a connected open set
$\Omega \subset \bR^n$
satisfying
$
 L u \ge 0
$
for an elliptic second order differential operator
\begin{align} \label{eq:elliptic_op}
 L = \sum_{i,j=1}^n a_{ij}(x) \partial_i \partial_j
  + \sum_{i=1}^n b_i(x) \partial_i
\end{align}
such that
\begin{itemize}
 \item
the matrix-valued function $\lb a_{ij}(x) \rb_{i,j=1}^n$
is locally uniformly positive-definite, and
 \item
the functions $b_i(x)$ are locally bounded
for $i=1, \ldots, n$.
\end{itemize}
The \emph{strong maximum principle} states that
if $u$ attains a maximum in $\Omega$,
then it must be constant.
\emph{Hopf's lemma} states that if
\begin{itemize}
 \item
$u$ extends smoothly to the boundary, and
 \item
there exists $x_0 \in \partial \Omega$
that $u(x_0) > u(x)$ for any $x \in \Omega$,
\end{itemize}
then one has
$
 d u( \nu) > 0,
$
where
$\nu$ is the unit outward normal vector to $\partial \Omega$
at $x_0$.
Proofs of the strong maximum principle
and Hopf's lemma can be found, e.g., in \cite[Section 6.4.2]{MR2597943}.

\bibliographystyle{amsalpha}
\bibliography{bibs}

\end{document}